\documentclass{amsart}
\usepackage{amsmath}

\usepackage{amssymb}
\usepackage{amsthm}
\usepackage{amscd,amsbsy}
\usepackage{xcolor}
\usepackage{tikz-cd}
\usetikzlibrary{arrows,matrix}
\usepackage{float}
\usepackage{mathtools}
\usepackage{hyperref}

\newtheorem{theorem}{Theorem}[section]
\newtheorem{lemma}[theorem]{Lemma}
\newtheorem{proposition}[theorem]{Proposition}
\newtheorem{corollary}[theorem]{Corollary}
\newtheorem{claim}[theorem]{Claim}

\theoremstyle{definition}
\newtheorem{definition}[theorem]{Definition}
\newtheorem{definitions}[theorem]{Definitions}
\newtheorem{example}[theorem]{Example}
\newtheorem{examples}[theorem]{Examples}
\newtheorem{definitions and remarks}[theorem]{Definitions and Remarks}

\theoremstyle{remark}

\newtheorem{remark}[theorem]{Remark}
\newtheorem{remarks}[theorem]{Remarks}

\numberwithin{equation}{section}

\newcommand{\codim}{\mathrm{codim}\,}

\newcommand{\graph}{\mathrm{graph}\,}

\newcommand{\bdry}{\mathrm{bdry}\,}

\newcommand{\Der}{\mathrm{Der}}
\newcommand{\lcm}{\mathrm{lcm}}

\newcommand{\reduced}{\mathrm{red}}

\newcommand{\al}{{\alpha}}
\newcommand{\be}{{\beta}}
\newcommand{\de}{{\delta}}

\newcommand{\De}{{\Delta}}
\newcommand{\ga}{{\gamma}}
\newcommand{\Ga}{{\Gamma}}

\newcommand{\la}{{\lambda}}
\newcommand{\La}{{\Lambda}}

\newcommand{\p}{{\partial}}
\newcommand{\s}{{\sigma}}

\newcommand{\vp}{{\varphi}}
\newcommand{\io}{{\iota}}

\newcommand{\Th}{{\Theta}}

\newcommand{\IN}{{\mathbb N}}
\newcommand{\IQ}{{\mathbb Q}}
\newcommand{\IA}{{\mathbb A}}
\newcommand{\IR}{{\mathbb R}}
\newcommand{\IC}{{\mathbb C}}
\newcommand{\IK}{{\mathbb K}}
\newcommand{\IL}{{\mathbb L}}

\newcommand{\IZ}{{\mathbb Z}}

\newcommand{\cC}{{\mathcal C}}
\newcommand{\cD}{{\mathcal D}}

\newcommand{\cF}{{\mathcal F}}

\newcommand{\cI}{{\mathcal I}}
\newcommand{\cJ}{{\mathcal J}}
\newcommand{\cK}{{\mathcal K}}
\newcommand{\cM}{{\mathcal M}}

\newcommand{\cO}{{\mathcal O}}
\newcommand{\cQ}{{\mathcal Q}}
\newcommand{\cR}{{\mathcal R}}
\newcommand{\cS}{{\mathcal S}}

\newcommand{\um}{\underline{m}}

\newcommand{\br}{\pmb{r}}

\newcommand{\bx}{\pmb{x}}

\newcommand{\by}{\pmb{y}}
\newcommand{\bu}{\pmb{u}}

\newcommand{\tbv}{\widetilde{\pmb{v}}}
\newcommand{\tbw}{\widetilde{\pmb{w}}}
\newcommand{\tbx}{\widetilde{\pmb{x}}}
\newcommand{\tby}{\widetilde{\pmb{y}}}
\newcommand{\tbz}{\widetilde{\pmb{z}}}
\newcommand{\tbu}{\widetilde{\pmb{u}}}
\newcommand{\bv}{\pmb{v}}
\newcommand{\bz}{\pmb{z}}
\newcommand{\bt}{\pmb{t}}

\newcommand{\bw}{\pmb{w}}
\newcommand{\bal}{\pmb{\al}}
\newcommand{\bbe}{\pmb{\be}}
\newcommand{\bga}{\pmb{\ga}}

\newcommand{\bla}{\pmb{\la}}

\newcommand{\bde}{\pmb{\de}}

\newcommand{\bxi}{\pmb{\xi}}

\newcommand{\bzeta}{\pmb{\zeta}}

\newcommand{\bzero}{\pmb{0}}
\newcommand{\bone}{\pmb{1}}

\newcommand{\tx}{{\widetilde x}}
\newcommand{\ta}{{\widetilde a}}
\newcommand{\tb}{{\widetilde b}}

\newcommand{\tg}{{\widetilde g}}
\newcommand{\tu}{{\widetilde u}}
\newcommand{\tv}{{\widetilde v}}
\newcommand{\tw}{{\widetilde w}}
\newcommand{\tz}{{\widetilde z}}

\newcommand{\tD}{{\widetilde D}}
\newcommand{\tE}{{\widetilde E}}
\newcommand{\tF}{{\widetilde F}}
\newcommand{\tG}{{\widetilde G}}
\newcommand{\tH}{{\widetilde H}}

\newcommand{\tM}{{\widetilde M}}
\newcommand{\tN}{{\widetilde N}}

\newcommand{\tR}{{\widetilde R}}

\newcommand{\tU}{{\widetilde U}}
\newcommand{\tV}{{\widetilde V}}
\newcommand{\tW}{{\widetilde W}}

\newcommand{\tbal}{{\widetilde \bal}}
\newcommand{\tbde}{{\widetilde \bde}}

\newcommand{\tbga}{{\widetilde \bga}}
\newcommand{\teta}{{\widetilde \eta}}

\newcommand{\tPhi}{{\widetilde \Phi}}
\newcommand{\tPsi}{{\widetilde \Psi}}

\newcommand{\hf}{{\widehat f}}
\newcommand{\hg}{{\hat g}}

\newcommand{\hPhi}{{\widehat \Phi}}

\newcommand{\wcO}{{\widehat \cO}}

\newcommand{\llb}{{[\![}}
\newcommand{\rrb}{{]\!]}}

\newcommand{\RN}[1]{%
  \textup{\uppercase\expandafter{\romannumeral#1}}%
}

\begin{document}
\title[Monomialization of a quasianalytic morphism]{Monomialization of a quasianalytic morphism}

\author[A.~Belotto da Silva]{Andr\'e Belotto da Silva}
\author[E.~Bierstone]{Edward Bierstone}
\address[A.~Belotto da Silva]{Universit\'e Aix-Marseille, Institut de Math\'ematiques de Marseille
(UMR CNRS 7373), Centre de Math\'ematiques et Informatique, 39 rue F. Joliot Curie, 13013 Marseille, France. Present address:  IMJ-PRG, CNRS 7586, Universit\'e de Paris, Institut de Math\'ematiques de Jussieu Paris Rive Gauche, B\^at. Sophie Germain, Place Aur\'elie Nemours, F-75013, Paris, France.}
\email{andre.belotto@imj-prg.fr}
\address[E.~Bierstone]{University of Toronto, Department of Mathematics, 40 St. George Street, Toronto, ON, Canada M5S 2E4}
\email{bierston@math.utoronto.ca}
\thanks{Research supported by NSERC Discovery Grant RGPIN-2017-06537 (Bierstone)}



\begin{abstract}
We prove a monomialization theorem for mappings in general classes of infinitely differentiable
functions that are called quasianalytic. Examples include Denjoy-Carleman classes, the class of $\cC^\infty$ functions definable in a
polynomially bounded $o$-minimal structure, as well as the classes of
real- or complex-analytic functions, and algebraic functions over any field of characteristic zero.
The monomialization theorem asserts that a mapping in a quasianalytic class can be transformed
to a mapping whose components are monomials with respect to suitable local coordinates, by
sequences of simple modifications of the source and target---local blowings-up and power substitutions
in the real cases, in general, and local blowings-up alone in the algebraic or analytic cases. 
Monomialization is a version of resolution of singularities for a mapping. We show that
it is not possible, in general, to monomialize by global blowings-up, even in the real-analytic case.
\end{abstract}

\maketitle

\vspace{-1.3cm}
\renewcommand{\abstractname}{R\'{e}sum\'{e}}
\begin{abstract}
On d\'emontre un th\'eor\`eme de monomialisation pour les morphismes dans une classe de fonctions quasi-analytiques. Ces
classes comprennent, par exemple, les classes de Denjoy-Carleman, la classe des fonctions $\cC^\infty$ d\'efinissables dans une structure $o$-minimale 
polynomialement born\'ee, ainsi que les classes des fonctions analytiques r\'eelles ou complexes, 
et de fonctions alg\'ebriques sur un corps de caract\'eristique nulle. Le th\'eor\`eme de monomialisation affirme 
qu'on peut transformer un morphisme dans une classe quasi-analytique en un autre morphisme dont les composantes 
sont des mon\^omes dans des coordonn\'ees locales convenables, par une suite de modifications de la source et du but. Dans le cas r\'eel g\'en\'eral, celles-ci sont des \'eclatements locaux et ramifications; dans le cas analytique ou alg\'ebrique,
elles sont simplement des \'eclatements locaux. On montre qu'il n'est pas possible, en g\'en\'eral, de monomialiser par des \'eclatements globaux, 
m\^eme dans le cas analytique r\'eel.
 \end{abstract}

\setcounter{tocdepth}{1}
\tableofcontents

\section{Introduction}\label{sec:intro}
The subjects of this article are monomialization and relative desingularization theorems
for mappings in general classes that are called quasianalytic.
Quasianalytic classes are classes of infinitely differentiable functions that are characterized by three simple axioms
including \emph{quasianalyticity}---injectivity of the Taylor series homomorphism at any point---and the implicit 
function theorem (see Definition \ref{def:quasian}).
Examples over $\IR$ include the class or real-analytic functions,
quasianalytic Denjoy-Carleman classes (objects of study in
classical real analysis), and the class of $\cC^\infty$ functions which are definable in a given 
polynomially bounded $o$-minimal structure (in model theory). 
The quasianalytic framework covers also complex-analytic mappings, as well as algebraic morphisms 
(regular morphisms of schemes of finite type over a field $\IK$ of characteristic zero)
provided that a neighbourhood of a point, in this case, is 
understood to mean an \'etale neighbourhood, so that the implicit function theorem holds. 

We prove that 
a mapping in a real quasianalytic class can be transformed by sequences of simple modifications
of the source and target (sequences of local blowings-up and local power substitutions) to a mapping whose
components are monomials with respect to suitable local coordinate systems; see Theorems 
\ref{thm:mainR3} and \ref{thm:mainR2}.
In the general algebraic and analytic cases, we show, moreover, that only local blowings-up are needed for monomialization 
(Theorems \ref{thm:mainA3} and \ref{thm:mainA2}).

Our results are best possible in the real cases, because of the following two phenomena.
First, we show that it is not possible, in general, to monomialize
a proper mapping by global blowings-up of the source and target, even in the real-analytic case 
(see Theorem \ref{thm:counterex} and Example \ref{ex:nonreg}, which is related to \cite[Sect.\,3]{BP}; there is no such counterexample 
for algebraic or proper complex-analytic morphisms). 
Secondly, we show that the use of power substitutions
is necessary in general quasianalytic classes because of a phenomenon that does not occur in
the classes of algebraic
or analytic functions---a quasianalytic function defined on a half-line may admit no extension to a quasianalytic
function in a neighbourhood of the boundary point (Nazarov, Sodin, Volberg \cite{NSV}; see \S\ref{subsec:power}).
The quasianalytic continuation theorem of \cite{BBC} intervenes in our proof of monomialization for real quasianalytic
classes, in this context. Applications of the monomialization theorem in quasianalytic geometry and model theory
are given in \S\ref{subsec:power}.

Our results in the algebraic and analytic cases include new proofs of Cutkosky's local monomialization theorems
\cite{C1,C3}. The latter establish local monomialization along a valuation. A basic difference in our approach 
is that we give a winning strategy for a version of Hironaka's game in the context of monomialization---Alice 
wins by applying our
method to choose the next local blowing-up at \emph{any} point of the fibre, which her opponent Bob is free
to pick on his turn.
On the other hand, we use \'etale-local blowings-up in the algebraic case (see Definitions \ref{def:defs}),
rather than Zariski-local blowings-up as in \cite{C1,C1Fourier}. Cutkosky shows that (along a valuation) 
the stronger result follows from the \'etale version; we plan to include this in a later article using methods developed 
here, but in this paper we prefer to emphasize
the quasianalytic framework.

Monomialization is a version of resolution of singularities for a mapping (classical resolution of singularities
is the case of target dimension one).
The monomialization problem in the algebraic case
has an extensive literature, going back at least to Zariski, and including or related to
the work of Akbulut and King on mappings
of complex surfaces \cite{AkKing}, the toroidalization, factorization and semistable reduction
results of Abramovich, Adiprasito, 
Denef, Karu, Liu, Matsuki, Temkin and W{\l}odarczyk \cite{ADK,AK,AKMW,ALT,W}, and the monomialization theorems of Cutkosky and of Denef \cite{D}. 
Toroidal and semistable morphisms are, in particular, monomial; the main distinction
of monomialization in relation to the preceding results on toroidalization and semistable reduction, is the requirement
of concrete transformation of the source and target of a given morphism by blowings-up with smooth centres.
In addition to \cite{C1,C3}, Cutkosky has proved monomialization by global blowings-up
for dominant projective morphisms in dimension three \cite{C2}. 

Our point of view is rather that of analysis, and algebro-geometric techniques used in the preceding
results (for example, valuations, Zariski-Riemann manifold) do not appear in our work, although the preceding articles
have had an important impact.
Indeed, even basic notions of commutative algebra like flatness are
not available in local rings of quasianalytic functions, which are not known to be Noetherian in general. On the other
hand, resolution of singularities does hold in quasianalytic classes \cite{BMinv,BMselecta}, even
for ideals that are not necessarily finitely generated \cite[Thm.\,3.1]{BMV}. The axioms for a quasianalytic class
are meant to capture the minimal properties of a ring of functions that are needed for desingularization
of ideals. In the analytic and algebraic cases, the proofs of our main theorems admit significant
simplifications, as we will point out in Section \ref{sec:LocalMon}.

The main new techniques developed in this article are based on logarithmic derivatives (logarithmic
with respect to the exceptional divisor that intervenes on blowing up) that are tangent to the fibres
of morphism (see \S\ref{subsec:deriv} and Section \ref{sec:tang}). In this context, we formulate
a theorem on resolution of singularities of an ideal relative to a monomial morphism (see Theorems
\ref{thm:corideal} and \ref{thm:ideal}) that we will prove together with the monomialization theorems 
above, within a common inductive scheme.
The notion of \emph{log derivations tangent to a morphism} coincides with that of \emph{relative
log derivations}, of origin in work of Grothendieck and Deligne (see \cite[Ch.\,IV]{Ogus}),
and plays a part here, as well as in \cite{ATWprog}, which is analogous to that played by standard logarithmic
derivations in the proof of resolution of singularities in \cite{BMfunct}. 
We believe the techniques of this article may be useful for global monomialization in the algebraic and
proper complex-analytic cases, 
perhaps combined with methods of \cite{ATW,ATWprog} (and \cite{BBAdvances} in low dimensions).

The paper of Dan Abramovich, Michael Temkin and Jaros{\l}aw
W{\l}odarczyk \cite{ATWprog} on relative desingularization is an independent work, of which we learned of a preliminary version 
in July 2019, before posting our article on arXiv that month. We are grateful to the authors for pointing out
connections between some of the techniques in our Sections \ref{sec:fit}, \ref{sec:tang} and logarithmic algebraic
geometry, and also to Mikhail Sodin for details of the non-extension phenomenon in \cite{NSV}. We would
like to thank also the reviewers and editors for their very helpful detailed comments.

\medskip
We begin by stating our main results. Details on the notions of blowing up and power substitutions
involved will be given in \S\ref{subsec:thms} following.

Let $\cQ$ denote a quasianalytic class, and let $M,\,N$ denote
manifolds (smooth spaces) of class $\cQ$; say $m=\dim M$, $n=\dim N$ (see Section \ref{sec:quasian}).
Let $\Phi: M\to N$ denote a mapping of class $\cQ$; we will say that $\Phi$ is a $\cQ$-morphism, or
simply a morphism if the class is understood.

\begin{definition}\label{def:monom1} \emph{Monomial morphism.} The $\cQ$-morphism $\Phi$ is
\emph{monomial at} a point $a\in M$ is there are coordinate systems (of class $\cQ$),
\begin{equation}\label{eq:coords1}
\begin{aligned}
(\bu,\bv,\bw) &= (u_1,\ldots,u_r, v_1,\ldots,v_s, w_1,\ldots, w_t),\\
(\bx,\by,\bz) &= (x_1,\ldots,x_p, y_1,\ldots,y_q,z_{q+1},\ldots, z_{s'}),
\end{aligned} 
\end{equation}
centred at $a$ and $b=\Phi(a)$ (respectively),
where $r+s+t=m$, $p+s'=n$ and $q\leq s\leq s'$, in which $\Phi$ can be written
\begin{equation}\label{eq:mon1}
\begin{aligned}
x_j &= \bu^{\bal_j},  &\quad j&=1,\ldots,p,\\
y_k &= \bu^{\bbe_k}(\xi_k + v_k),  &\quad k&=1,\ldots,q,\\
z_l &= v_l, &\quad l&=q+1,\ldots,s,\\
z_l &= 0, &\quad l&=s+1,\ldots,s',
\end{aligned}
\end{equation}
where
\begin{enumerate}
\item the exponent vectors $\bal_j = (\al_{j1},\ldots,\al_{jr}) \in \IN^r$, $j=1,\ldots,p$, are linearly independent over $\IQ$ (in particular, $r\geq p$);
\item for each $k=1,\ldots,q$, $\bbe_k = (\be_{k1},\ldots,\be_{kr}) \in \IN^r$ is nonzero and $\IQ$-linearly dependent on
$\{\bal_1,\ldots,\bal_p\}$, and $\xi_k \neq 0$.
\end{enumerate}
We say that $\Phi$ is a \emph{monomial morphism} if it is monomial at every point $a\in M$.
\end{definition}

\begin{remarks}\label{rem:mon1}
(1) A morphism of the form
\begin{equation}\label{eq:mon2}
\begin{aligned}
x_j &= \bu^{\bal_j},  &\quad j&=1,\ldots,p,\\
z_l &= v_l, &\quad l&=1,\ldots,s,\\
z_l &= 0, &\quad l&=s+1,\ldots,s',
\end{aligned}
\end{equation}
with respect to coordinate systems $(\bu,\bv,\bw) = (u_1,\ldots,u_r, v_1,\ldots,v_s, w_1,\ldots, w_t)$ and
$(\bx,\bz) = (x_1,\ldots,x_p, z_1,\ldots, z_{s'})$, where $0\leq s\leq s'$ and the exponent vectors $\bal_j = (\al_{j1},\ldots,\al_{jr}) \in \IN^r$, $j=1,\ldots,p$, are linearly independent over $\IQ$, can be rewritten 
in the form  \eqref{eq:mon1} in coordinates centred at 
any point $a$ and at $b=\Phi(a)$ (respectively). On the other hand, 
a monomial morphism can be transformed to a morphism satisfying \eqref{eq:mon2} in suitable coordinate charts, by performing local blowings-up; see Remark \ref{rk:ReducingNormalForm} for a precise statement. 
Compare with Theorem \ref{thm:mainR4}.

\medskip\noindent
(2) In the algebraic case, \eqref{eq:mon1} (or \eqref{eq:mon2}) should be understood at a (closed) point $a\in M$
in the following way. Let $\IL = \IK_a$ denote the residue field of $a$. Then there is an \'etale
morphism $\varepsilon: U\to M$ such that $\varepsilon(a')=a$, $\IK_{a'} = \IL$, and $\Phi\circ\varepsilon$ is given by 
\eqref{eq:mon1} with $\xi_k \in \IL$, with respect to \'etale local coordinates (or ``uniformizing parameters'') at $a'$ and $b$.
(In general, $\IL$ may be strictly larger than $\IK_b$, so the mapping given by \eqref{eq:mon1}
involves $\times_{\IK_b} \IL$ in the image). See also \S\ref{subsec:alg}.
\end{remarks}

As in resolution of singularities, our proofs of the monomialization theorems involve keeping
careful track of the exceptional divisors of the blowings-up involved, and exploiting their combinatorial
structure. It is natural to assume from the beginning that there are simple normal crossings divisors
$D$ on $M$ and $E$ on $N$. A \emph{simple normal crossings (SNC) divisor} means a union of
smooth closed hypersurfaces that simultaneously have only transverse intersections. A
$\cQ$-\emph{morphism} $\Phi: (M,D) \to (N,E)$ denotes a $\cQ$-morphism $\Phi: M\to N$ such that
$\Phi^{-1}(E)$ is SNC \emph{as a space} and its support lies in that of $D$ (i.e., the
ideal of $\Phi^{-1}(E)$ is locally a product of powers of the ideals of certain components of $D$).

\begin{definition}\label{def:monom2}\emph{Monomial morphism compatible with divisors.}
We say that a $\cQ$-morphism $\Phi: (M,D) \to (N,E)$ is \emph{monomial at} a point $a\in M$ if
$\Phi: M \to N$ is monomial at $a$ as in Definition \ref{def:monom1}, with respect to
local coordinate systems \eqref{eq:coords1} in which $D=\{u_1 \cdots u_r =0\}$ and 
$E =\{x_1\cdots x_p \cdot y_1 \cdots y_q=0\}$.

In this case, the coordinate systems $(\bu,\bv,\bw)$ and $(\bx,\by,\bz)$ will be called 
$\Phi$-\emph{monomial}. The variables $\bw = (w_1,\ldots,w_t)$ will be called $\Phi$-\emph{free}.
A coordinate system in which every component of a divisor is a coordinate subspace will
be called \emph{compatible} with the divisor.
\end{definition}

In our monomomialization theorems \ref{thm:mainR3} and \ref{thm:mainA3} following,
$\Phi: (M,D) \to (N,E)$ denotes a $\cQ$-morphism.

\begin{definition}\label{def:semiproper} (In the real quasianalytic or complex analytic cases) 
a family of morphisms $\{\s_\la: V_\la \to M\}$ is a \emph{semiproper
covering} of $M$ if $\{\s_\la\}$ is a covering of $M$ (i.e., $\bigcup \s_\la(V_\la) = M$)
and, for every compact subset $K$ of $M$, there are finitely many
compact $K_i \subset V_{\la(i)}$ such that $\bigcup_i\s_{\la(i)}(K_i) = K$. 
\end{definition}

\begin{theorem}[Monomialization of a real quasianalytic morphism]\label{thm:mainR3}
Suppose that $\cQ$ is a real quasianalytic class, and that $\Phi:M\to N$ is proper. Then there is
a countable family of commutative diagrams
\begin{equation}\label{eq:mainR2}
\begin{tikzcd}
(V_\la, D_\la) \arrow{d}{\Phi_{\la}} \arrow{r}{\s_\la} & (M,D) \arrow{d}{\Phi}\\
(W_\la, E_\la) \arrow{r}{\tau_\la} & (N,E)
\end{tikzcd}
\end{equation}
where
\begin{enumerate}
\item each $\s_\la$ and $\tau_\la$ is a composite of finitely many smooth local blowings-up
and local power substitutions, compatible with $D$ and $E$, respectively;
\item the families of morphisms $\{\s_\la: V_\la \to M\}$ and $\{\tau_\la: W_\la \to N\}$
are semiproper coverings of $M$ and $N$;
\item each $\Phi_\la: (V_\la, D_\la) \to (W_\la, E_\la)$ is a monomial morphism.
\end{enumerate}
\end{theorem}

\begin{theorem}[Monomialization of an analytic or algebraic morphism]\label{thm:mainA3}
Suppose that $\cQ$ is either the class of $\IR$- or $\IC$-analytic functions, or the class of algebraic
functions over a field $\IK$ of characteristic zero. In the analytic case, suppose also that
$\Phi:M\to N$ is proper. Then there is a countable family (in the analytic case), or a finite family
(in the algebraic) of commutative diagrams \eqref{eq:mainR2}, where
\begin{enumerate}
\item each $\s_\la$ and $\tau_\la$ is a composite of finitely many smooth local blowings-up,
compatible with $D$ and $E$, respectively;
\item the families of morphisms $\{\s_\la\}$ and $\{\tau_\la\}$ are coverings of $M$ and $N$,
semi-{\allowbreak}proper in the analytic case;
\item each $\Phi_\la: (V_\la, D_\la) \to (W_\la, E_\la)$ is a monomial morphism.
\end{enumerate}
\end{theorem}

In the case that $\dim M = \dim N$ and $\Phi$ is dominant, we can prove 
Theorem \ref{thm:mainR3} with the stronger conclusion that each $\Phi_\la =$ identity 
(see Theorem \ref{thm:transftoid}).

\begin{theorem}[Counterexample to global monomialization of
a real-analytic morphism]\label{thm:counterex}
A proper morphism $\Phi: V\to W$ of real-analytic manifolds
cannot, in general, be monomialized by global blowings-up of the source and target, 
even over some neighbourhood of a given point of $\Phi(V)$. 
\end{theorem}

We will show in Section \ref{sec:ex}, in fact, that
$\Phi$ cannot, in general, be transformed by a proper real-bimeromorphic morphism $\tau$
of the target and a proper surjection $\s$ of the source, to a mapping that is regular in the sense of Gabrielov (Definition \ref{def:regular}).

We are grateful to Jan Denef for asking about a bimeromorphic morphism $\tau$.

\begin{theorem}[Relative desingularization of an ideal]\label{thm:corideal}
Let $\Phi: (M,D) \to (N,E)$ denote a monomial morphism of class $\cQ$ and let
$\cI$ denote an ideal sheaf on $M$ (we assume that the stalk of $\cI$ at any point
is generated by local sections defined in a common neighbourhood, though we do not assume
that $\cI$ is finitely generated; see Definition \ref{def:priv}). 
Then there is a countable family (finite, in the algebraic case) of commutative diagrams \eqref{eq:mainR2},
where conditions (1), (2), (3) of Theorem \ref{thm:mainR3} (or of Theorem \ref{thm:mainA3}, in the analytic
and algebraic cases) hold, together with the following: 
\begin{enumerate}
\item[(4)] for each $\la$, the pull-back $\s_\la^*(\cI)$ is principal and monomial.
\end{enumerate}
\end{theorem} 

\subsection{Monomialization theorems}\label{subsec:thms}

\begin{definitions and remarks}\label{def:defs} 
(1) A \emph{smooth blowing-up} means a blowing-up with smooth centre. A \emph{(smooth) local blowing-up} 
$\s: V' \to V$ is a composite $\io\circ\rho$ of an inclusion of an open subset $\io: U\hookrightarrow V$, and a (smooth)
blowing-up $\rho: V'\to U$. We will always assume that our blowings-up are smooth, unless we state otherwise.
(In the algebraic case, an open subset is understood to mean with respect to the \'etale topology, so 
that a \emph{local blowing-up} $\s$ means the composite $\varepsilon\circ\rho$ of an \'etale mapping
$\varepsilon: U\to V$ and a smooth blowing-up $\rho: V'\to U$.)

\medskip
\noindent
(2) A \emph{power substitution} means, \emph{roughly speaking}, a morphism $\rho: U' \to U$ of the form
$u_i = \tu_i^{k_i}$, $i=1,\ldots,r$, where each $k_i$ is a positive integer, and $U$ is a
chart with coordinates $\bu=(u_1,\ldots,u_r)$. In order to cover $U$, given $\rho$, we introduce the
mapping
$$
P: \coprod_{\pmb{\varepsilon}} U_{\pmb{\varepsilon}} \to U,
$$
where $\coprod$ denotes disjoint union, $\pmb{\varepsilon}= (\varepsilon_1,\ldots,\varepsilon_r)\in \{-1,1\}^r$, each $U_{\pmb{\varepsilon}}$ is 
a copy of $U'$, and $\rho_{\pmb{\varepsilon}} := P|_{U_{\pmb{\varepsilon}}}$ is given by $u_i = \varepsilon_i\tu_i^{k_i}$, $i=1,\ldots,r$
(so that $\rho = \rho_{\bone}$, $\bone = (1,\ldots, 1)$), and where the product is taken over
all  $\pmb{\varepsilon} \in \{-1,1\}^r$ such that $\varepsilon_i = 1$ whenever $k_i$ is odd.
We assume that $U$ is chosen in suitable symmetric product
form, so that $P$ will be surjective. In our monomialization theorems, power substitutions $\rho$
will always be extended to more general \emph{power substitutions} $P$ defined in this way. 
(Of course, we can accomplish the same thing by using power substitutions $u_i = \pm\tu_i^{k_i}$ and
increasing the number of morphisms $\s_\la$, $\tau_\la$.)
A \emph{local
power substitution} $V' \to V$ is a composite $\io\circ P$ of the inclusion of a suitable coordinate
chart $\io: U\hookrightarrow V$ and an extended power substitution $P: V' \to U$.
\end{definitions and remarks}

\begin{definitions}\label{def:compat}\emph{Blowings-up and power substitutions compatible with divisors.}
We say that a (local) blowing-up $\s: M' \to M$ is \emph{compatible} with an SNC divisor $D$ on $M$ if 
$D$ and the centre $C$
of $\s$ have simple normal crossings (i.e., locally, there are coordinates in which each component of $D$
is a coordinate hypersurface and $C$ is a coordinate subspace). In this case, $\Phi^{-1}(D)$ is SNC; in 
particular, the \emph{transform} $D' := \Phi^{-1}(D)_\reduced$ of $D$ by $\s$ is an SNC divisor
(where the subscript \emph{red} means ``with its reduced structure'').

A blowing-up $\s: M' \to M$ will be called \emph{combinatorial} if it satisfies the stronger condition that
its centre is (locally) an intersection of components of $D$.

We say that a (local) power substitution $\rho: M'\to M$ is \emph{compatible} with $D$ if it is defined as in Definitions 
\ref{def:defs}(2) with respect to coordinates $(\bu,\bv)=(u_1,\ldots,u_r,v_1,\ldots,v_s)$, where
$D=\{u_1 \cdots u_r =0\}$ and $\rho$ is trivial in every variable $v_k$. Again, $\Phi^{-1}(D)$ is SNC, and
the \emph{transform} 
$D' := \Phi^{-1}(D)_\reduced$ of $D$ is an SNC divisor.

We inductively define \emph{compatibility} with $D$ of a morphism $\s: \tM \to M$ given by a finite 
composite of local blowings-up and power substitutions. If $\s: \tM \to M$ is compatible with $D$,
then the \emph{transform} $\tD$ of $D$ by $\s$ is well-defined as an SNC divisor on $\tM$, and
$\s^{-1}(D)_\reduced = \tD$, so that $\s$ induces a morphism $\s: (\tM,\tD)\to (M,D)$.
\end{definitions}

In Theorems \ref{thm:mainR2}, \ref{thm:mainA2}, \ref{thm:mainR4} and \ref{thm:mainR5} following, 
$\Phi: (M,D) \to (N,E)$ denotes a $\cQ$-morphism. Theorems \ref{thm:mainR3} and \ref{thm:mainA3} are 
immediate consequences of Theorems \ref{thm:mainR2} and \ref{thm:mainA2}.

\begin{theorem}[Local monomialization of a real quasianalytic morphism compatible with divisors]\label{thm:mainR2}
Suppose that $\cQ$ is a real quasianalytic class. Let $a\in M$. Then there are open neighbourhoods
$V$ of $a$ and $W$ of $b= \Phi(a)$, and a finite number of commutative diagrams \eqref{eq:mainR2}, where
\begin{enumerate}
\item each $\s_\la$ and $\tau_\la$ is a composite of finitely many smooth local blowings-up
and local power substitutions, compatible with $D$ and $E$, respectively;
\item the families of morphisms $\{\s_\la\}$ and $\{\tau_\la\}$ cover $V$ and $W$, respectively,
and there are compact subsets $K_\la \subset V_\la$, $L_\la \subset W_\la$, such that
$\bigcup \s_\la(K_\la)$ and $\bigcup \tau_\la(L_\la)$ are (compact) neighbourhoods of $a$ and $b$,
respectively;
\item each $\Phi_\la$ is a monomial morphism (i.e., locally of the form \eqref{eq:mon1}, where
$D_\la=\{u_1 \cdots u_r =0\}$ and 
$E_\la =\{x_1\cdots x_p \cdot y_1 \cdots y_q=0\}$).
\end{enumerate}
\end{theorem}

\begin{theorem}[Local monomialization of an analytic or algebraic morphism compatible with divisors]\label{thm:mainA2}
Suppose that $\cQ$ is either the class of $\IR$- or $\IC$-analytic functions, or the class of algebraic
functions over a field $\IK$ of characteristic zero. Let $a\in M$. Then there are open neighbourhoods
$V$ of $a$ and $W$ of $b= \Phi(a)$, and a finite number of commutative diagrams \eqref{eq:mainR2},
where
\begin{enumerate}
\item each $\s_\la$ and $\tau_\la$ is a composite of finitely many smooth local blowings-up,
compatible with $D$ and $E$, respectively;
\item the families of morphisms $\{\s_\la\}$ and $\{\tau_\la\}$ cover $V$ and $W$, respectively, 
and, in the analytic case, 
there are compact subsets $K_\la \subset V_\la$, $L_\la \subset W_\la$, such that
$\bigcup \s_\la(K_\la)$ and $\bigcup \tau_\la(L_\la)$ are neighbourhoods of $a$ and $b$,
respectively;
\item each $\Phi_\la$ is a monomial morphism (i.e., locally of the form \eqref{eq:mon1}, where
$D_\la=\{u_1 \cdots u_r =0\}$ and 
$E_\la =\{x_1\cdots x_p \cdot y_1 \cdots y_q=0\}$).
\end{enumerate}
\end{theorem}

\begin{remark}\label{rem:idea} 

\emph{Rough idea of the proofs.} Theorems \ref{thm:mainR2} and
\ref{thm:mainA2} will be proved by transforming $\Phi$ to monomial form, one component at a time (with respect to local coordinates). 
Suppose that $\Phi = (\Phi_1,\ldots\Phi_n)$, where $\Phi^{(k)} := (\Phi_1,\ldots,\Phi_k)$ is in monomial form \eqref{eq:mon1}
($k<n$); in particular, $\Phi^{(k)}=(\Theta_r,0,\ldots,0)$, $r\leq k$, where $\Theta_r$ is \emph{dominant} (or generically submersive; i.e., $s'=s$).
We can, therefore, reduce to the case that $\Phi^{(k)}$ is monomial and dominant by replacing $\Phi^{(k)}$ with $\Theta_r$;
we keep the notation 
$\Phi^{(k)}$ here for convenience. Note that $\Phi^{(k+1)} = (\Phi^{(k)},\Phi_{k+1})$ is not necessarily dominant. Then we
will make transformations preserving monomiality of $\Phi^{(k)}$, so that $\Phi_{k+1}$ is transformed to the form required, within an inductive 
scheme that involves three steps
(see Section \ref{sec:LocalMon}).

\medskip\noindent
I.A. Resolution of singularities of an ideal relative to a monomial morphism (see Theorem \ref{thm:ideal}
below).

\medskip\noindent
I.B. Transformation of $\Phi_{k+1}$ to a function of class $\cQ$ of the form $g(u,v) + \vp(u,v,w)$, with the property that
$\Psi^{(k+1)} := (\Phi^{(k)}, \vp) =(\Phi^{(k)}, \Phi_{k+1} - g)$ is in monomial form and
\begin{equation}\label{eq:introg}
dg \wedge d \Phi_1 \wedge \cdots \wedge d \Phi_k = 0.
\end{equation}
(equivalently, the formal Taylor expansion of $g$ at $a$ 
has the form $\sum g_{\bga\bde}u^{\bga}v^{\bde}$,
where $\bga$ is $\IQ$-linearly dependent on $\bal_1,\ldots,\bal_p$ whenever $g_{\bga\bde}\neq0$.)
In this case, we will say that $\Phi^{(k+1)}$ is in \emph{pre-monomial} form. (For simplicity of language,
 we are keeping the same notation $\Phi^{(k)}$ and $\Phi^{(k+1)}$ here, after the transformations
involved; also in II following.) See Definition \ref{def:PreMonomial} and Example \ref{ex:PreMonomial} below, where we also elucidate the difference between $\Phi^{(k+1)}$ and $\Psi^{(k+1)}$.

\medskip\noindent
II. Transformation to $g=0$; i.e., transformation of a pre-monomial morphism $\Phi^{(k+1)}$ to a monomial morphism $\Phi^{(k+1)}$.

\medskip
Steps I.A and I.B involve 
techniques of a differential nature (see \S\ref{subsec:deriv} and Section \ref{sec:tang}), whereas
Step II uses methods that are largely combinatorial, developed in Section \ref{sec:premon}.
Section \ref{sec:premon} is the only part of the paper where there is a dichotomy between the cases of
analytic or algebraic morphisms, and the general real quasianalytic case (in particular, the
only part where power substitutions are needed in the latter case; cf.\ \S\ref{subsec:power}). 

In \S\S\ref{subsec:deriv}, \ref{subsec:power}, we will try to give an idea of new techniques introduced, as well as of applications.
\end{remark}

Each of the theorems above can be stated equivalently in the following two different ways.
(We give the equivalent statements only for Theorem \ref{thm:mainR3}, and leave the
other theorems to the reader.)

\begin{theorem}\label{thm:mainR4}
Suppose that $\cQ$ is a real quasianalytic class, and that $\Phi:M\to N$ is proper. Then there is
a countable family of commutative diagrams \eqref{eq:mainR2}, where conditions (1),
(2) of Theorem \ref{thm:mainR3} hold, together with the following variant of (3):
\begin{enumerate}
\item[(3)] each $\Phi_\la: (V_\la, D_\la) \to (W_\la, E_\la)$ is a proper morphism of the
form \eqref{eq:mon2} in suitable local coordinate charts of class $\cQ$, where
$D_\la=\{u_1 \cdots u_r =0\}$ and 
$E_\la =\{x_1\cdots x_p = 0\}$.
\end{enumerate}
\end{theorem}

Theorem \ref{thm:mainR4} $\implies$ \ref{thm:mainR3}, by Theorem \ref{thm:LogFittingCharacterization},
and \ref{thm:mainR3} $\implies$ \ref{thm:mainR4}, by Lemma \ref{lem:TranslationTrick} (see Remark \ref{rk:ReducingNormalForm}). The
equivalence of \ref{thm:mainR3} and \ref{thm:mainR5} following is a topological exercise.

\begin{theorem}\label{thm:mainR5}
Suppose that $\cQ$ is a real quasianalytic class, and that $\Phi:M\to N$ is proper.
Then there is a commutative diagram of $\cQ$-morphisms
\begin{equation}\label{eq:mainR5}
\begin{tikzcd}
(\tM, \tD) \arrow{d}{\tPhi} \arrow{r}{\s} & (M,D) \arrow{d}{\Phi}\\
(\tN, \tE) \arrow{r}{\tau} & (N,E)
\end{tikzcd}
\end{equation}
where
\begin{enumerate}
\item over every relatively compact open subset of $M$ (or $N$),
$\s$ (or $\tau$) is a composite of finitely many morphisms
of the following three kinds:
\begin{enumerate}
\item (semiproper) morphism given by a covering $\coprod U_i \to U$ by finitely many 
coordinate charts of class $\cQ$ compatible with the divisor;
\item compatible (global) blowing-up;
\item compatible power substitution (in coordinate charts $U_i$);
\end{enumerate}
\smallskip
\item $\tPhi$ is a monomial morphism.
\end{enumerate}
\end{theorem}

\subsection{Derivations tangent to a morphism}\label{subsec:deriv} The purpose of this
subsection is to motivate the calculus of derivations tangent to the fibres of a morphism
that we will develop in Section \ref{sec:tang}, and, in this context, to formulate a theorem
on resolution of singularities of an ideal relative to a monomial morphism (see Theorems
 \ref{thm:ideal} and \ref{thm:corideal}) that we will 
prove in Section \ref{sec:LocalMon} simultaneously with the monomialization theorems 
above, as part of the inductive scheme.

We recall that, in resolution of singularities of a sheaf of ideals $\cI$ on a smooth space $M$ in characteristic zero, the
\emph{derivative ideals} $\cD^k(\cI)$ generated by derivations of order $\leq k$
of local sections of $\cI$ play a very important part; in particular, they determine a \emph{coefficient
ideal} which, after restriction to a \emph{maximal contact hypersurface}, can be
assumed to admit a resolution of singularities, by induction on dimension. The ideal $\cI$
has order $\mu$ at a given point if and only if an appropriately defined coefficient ideal
has order greater than or equal to a corresponding value depending on $\mu$ (coming from the fact that,
if $\cI$ has order $\mu$ at a point, then $\cD^k(\cI)$ has order $\geq \mu - k$). The problem of order reduction of $\cI$ (leading to desingularization)
is therefore equivalent to that of order reduction of the coefficient ideal, provided
we know that, after a sequence of suitable blowings-up, the transform of the coefficient
ideal of $\cI$ equals the coefficient ideal of the transform of $\cI$.

The coefficient ideal (and maximal contact) are local constructions that are not
uniquely determined by $\cI$, so that some notion of functoriality must be built into the
inductive scheme to guarantee that the coefficient ideal is resolved by blowings-up that 
are independent of the local choices. In the approach of \cite{BMfunct}, this is achieved
by using ideals $\cD_D^k(\cI)$ of derivatives that are logarithmic with respect to the
exceptional divisor $D$. In a local coordinate system $(\bu, \bv) = (u_1,\ldots,u_r,v_1,\ldots,v_s)$
where $D= \{u_1\cdots u_r=0\}$, the sheaf $\cD_D$ of \emph{logarithmic derivations}
(sometimes denoted $\Der_M(-\log D)$) is generated by
\begin{equation}\label{eq:deriv1}
u_i\frac{\p}{\p u_i}, \ \ i=1,\ldots, r,\quad \text{ and }\quad \frac{\p}{\p v_l}, \ \ l=1,\ldots, s.
\end{equation}
Logarithmic derivatives have a particularly simple law of transformation by blowings-up
that are compatible with $D$ \cite[Lemma 3.1]{BMfunct}, which can be used to show that
the transformed coefficient ideal equals the coefficient ideal of the transform of $\cI$ in
a very natural way \cite[Thm.\,3.10]{BMfunct}.

A key technique in our proof of monomialization is a calculus of \emph{logarithmic
derivatives that are tangent to the fibres of a morphism} (Section \ref{sec:tang};
some of the ideas appear
in \cite{BdSJA,BdS,BdSIbero}). 
If $\Phi: (M,D) \to (N,E)$ is a monomial morphism, written, say, in the form \eqref{eq:mon1}
in coordinates at a point $a\in M$, where $D = \{u_1\cdots u_r=0\}$, then the stalk at $a$ of 
the sheaf $\De^\Phi$ of \emph{logarithmic derivations} \emph{tangent to} $\Phi$ is generated by
\begin{equation}\label{eq:logder}
\sum_{i=1}^r \ga_{ji} u_i\frac{\p}{\p u_i}, \ \ j=1,\ldots, r-p,\quad \text{ and }\quad \frac{\p}{\p w_l}, \ \ l=1,\ldots, t,
\end{equation}
where $\bga_j = (\ga_{j1},\ldots,\ga_{jr})\in \IQ^r$, $j=1,\ldots,r-p$, form a basis of the
orthogonal complement of the $\IQ$-linear subspace spanned by $\bal_1,\ldots,\bal_p$
with respect to the standard scalar product $\langle \bga,\bal\rangle$. The sheaf of derivations $\De^\Phi$
is the subsheaf of $\cD_D$ of log derivations that annihilate the components of $\Phi$; note that
(assuming $\Phi$ is dominant; i.e., $s'=s$)
$\De^\Phi$ also annihilates any function $g(u,v,w)$ in our class such that 
$dg\wedge d \Phi_1 \wedge \cdots \wedge d \Phi_n = 0$
(cf.\ Remark \ref{rem:idea}).

Let us write $\cQ$ to denote also the sheaf of germs of functions of class $\cQ$ on $M$. 
Recall that, in general,
we are not assuming that the stalks of $\cQ$ are Noetherian. In a quasianalytic class,
the appropriate analogue of sheaves of ideals of finite type is the following.

\begin{definition}\label{def:priv}
We say that a sheaf of ideals $\cI \subset \cQ$ (or, more generally, a sheaf of $\cQ$-modules
$\cM$ is \emph{privileged} if the stalk of $\cI$ (or $\cM$) at any point $a$
is generated by local sections defined in a common neighbourhood of $a$. If $\cI \subset \cQ$
is a privileged ideal (sheaf), then  $\cQ/\cI$ is quasicoherent.
\end{definition}

Given a privileged ideal (sheaf) $\cI$, we define a chain of privileged ideals (in analogy with
the log derivative ideals $\cD_D^k(\cI)$),
$$
\cI = \cI_0^{\Phi} \subset \cI_1^{\Phi} \subset \cdots \subset \cI_k^{\Phi} \subset \cdots,
$$
by $\cI_{k+1}^{\Phi} := \cI_k^{\Phi} + \De^{\Phi}(\cI_k^{\Phi})$, $k=0,1,\ldots,$ and we set
$\cI_\infty^{\Phi} := \sum \cI_k^{\Phi}$; cf.\,\cite[\S\,3.1]{BdSJA},\,\cite[\S\,5.1]{BdS}.

\smallskip
We introduce the \emph{log differential order} $\nu_a(\cI,\Phi)$ of $\cI$ \emph{relative to} $\Phi$
as the smallest $\nu \in \IN \cup \{\infty\}$ such that $\cI_\infty^{\Phi} := \cI_\nu^{\Phi}$. Note that,
if $\Phi =$ constant and $D=\emptyset$, then $\nu_a(\cI,\Phi)$ equals the standard order of $\cI$ at $a$.
(Compare with \cite[\S\,3.1]{BdSJA},\,\cite[Def.\,5.1]{BdS}.
In the case $\Phi =$ constant, the log differential order has also been defined
in \cite[3.6.9]{ATW}.)

In Section \ref{sec:LocalMon}, the ideal $\cI_\infty^\Phi$ and reduction of the log differential order
relative to $\Phi$ play parts analogous to those of the coefficient ideal and reduction of the order
of $\cI$, in resolution of singularities. There is, however, an important technical difficulty---the
transformation law for $\cI_\infty^\Phi$ is not as simply related to that for $\cI$ as in the
case of the standard log derivations $\cD_D$. For this reason, we have to build appropriate
control of the log derivations $\De^\Phi$ tangent to $\Phi$ into the inductive scheme. We will,
in fact, prove the monomialization theorems above simultaneously with the following.

\begin{theorem}
\label{thm:ideal}
Let $\Phi: (M,D) \to (N,E)$ denote a monomial morphism of class $\cQ$ and let
$\cJ$ denote a privileged ideal sheaf on $M$ which is $\De^\Phi$-closed; i.e., $\De^\Phi(\cJ)\subset \cJ$.
Then there is a countable family (finite, in the algebraic case) of commutative diagrams \eqref{eq:mainR2},
where conditions (1), (2), (3) of Theorem \ref{thm:mainR3} (or of Theorem \ref{thm:mainA3}, in the analytic
and algebraic cases) hold, together with the following: for each $\la$,
\begin{enumerate}
\item[(4)] $\s_\la^*(\De^\Phi) = \De^{\Phi_\la}$, where $\s_\la^*(\De^\Phi)$ denotes the pull-back 
by $\s_\la$ of the sheaf of log derivations $\De^\Phi$ tangent to $\Phi$;
\smallskip
\item[(5)] the pull-back $\s_\la^*(\cJ)$ of $\cJ$ is principal and monomial.
\end{enumerate}
\end{theorem} 

Theorem \ref{thm:corideal} will be proved as a corollary of
Theorem \ref{thm:ideal} also in Section \ref{sec:LocalMon}.

\subsection{Power substitutions---significance and applications}\label{subsec:power}
In this subsection, we show that a morphism in a quasianalytic class
(in particular, in a quasianalytic Denjoy Carleman class, or in a class of $\cC^\infty$
functions definable in a polynomially bounded $o$-minimal structure; see Section \ref{sec:quasian})
cannot, in general, be monomialized by sequences of only local blowings-up of the source and
target. We also give applications of monomialization by local blowings-up and power
substitutions, to problems in subanalytic (or sub-quasianalytic) geometry
which cannot be solved using local blowings-up alone, even in the real-analytic case.
It may be interesting to compare the use of power substitutions in our monomialization theorems
with the rather different way that they intervene in \cite{ATW}.

\begin{examples}\label{ex:halfline}
These examples are based on a construction due to Nazarov, Sodin and Volberg \cite[\S\,5.3]{NSV}:
Given a suitable quasianalytic Denjoy-Carleman class $\cC_M$ (see \S\ref{subsec:quasian}), there is
a function $g \in \cC_M([0,1))$ which admits no extension to a function in $\cC_{M'}((-\de,1))$, for any 
quasianalytic Denjoy-Carleman class $\cC_{M'}$ (for example, we can take the
sequence $M=(M_k)_{k\in \IN}$, where $M_k = (\log k)^k$).

By adding a linear term to $g$ if necessary, we can assume that $g(0)=0$ and $g'(0)\neq 0$.
Let $\Phi: (-1,1)\to \IR^2$ denote the $\cC_M$-morphism $\Phi(x) = (x^2,g(x^2))$. Assume
that $\Phi$ can be monomialized by local blowings-up. There are no nontrivial blowings-up
of the source, and blowings-up over the origin in the target lead to a morphism of the same
form as $g$. Monomialization of $\Phi$ must necessarily be achieved, therefore, by a
coordinate change in the target; this is not possible because of non-extendability of $g$.

We can also consider the $o$-minimal structure $\IR_{\cC_M}$ given by expansion of the
real field by restricted functions of class $\cC_M$ (cf.\ \cite{RSW}), and the quasianalytic
class $\cQ$ of $\cC^\infty$ functions that are locally definable in $\IR_{\cC_M}$. By
\cite[Thm.\,1.6]{BBC}, any function $h\in\cQ((-1,1))$ belongs to a Denjoy-Carleman
class {$\cC_{M^{(p)}}$}, for some positive integer $p$, where $M^{(p)}$ is the sequence
obtained from $M = (M_k)$ by a shift $M^{(p)}_k := M_{pk}$. For example, the
sequence $M=(M_k)_{k\in \IN}$, where $M_k = (\log(\log k))^k$, also satisfies the conditions
of \cite{NSV} and, in this case, each shifted class $\cC_{M^{(p)}}$ is quasianalytic 
(cf.\,\cite[Example 6.7]{Nelim}); it follows that the morphism
$\Phi$ above also cannot be monomialized in the class $\cQ$ by local blowings-up alone.
\end{examples}

Power substitutions are important not only for technical reasons as suggested by Examples \ref{ex:halfline}. Theorem \ref{thm:transftoid} following is a
consequence of our monomialization theorems. 
(In the case that $\dim M = \dim N$, this is a \emph{local factorization} 
statement for a real quasianalytic morphism.)

\begin{theorem}\label{thm:transftoid}
Suppose that $\cQ$ is a real quasianalytic class and that $\Phi: (M,D)\to (N,E)$ is 
a proper $\cQ$-morphism, where $\Phi: M\to N$ is generically of rank $m=\dim M$.
Then Theorem \ref{thm:mainR3} holds with the stronger conclusion that
each $\Phi_\la = \text{identity}_{V_\la}\times 0$; i.e., $(\bx,\bz,\tbz) = (\bu,\bv,\bzero_{n-m})$
in local coordinates. (In the real algebraic case, $\{\Phi_\la\}$ is
a finite family.)
\end{theorem}

\begin{proof} This follows from Theorems \ref{thm:mainR3}, \ref{thm:mainA3}
and Lemma \ref{lem:transftoid}.
\end{proof}

For a real quasianalytic class $\cQ$, we can define \emph{sub-quasianalytic} functions
and sub-quasianalytic sets in an obvious way generalizing subanalytic. (In particular,
a sub-quasianalytic function is a function whose graph is sub-quasianalytic.)

\begin{theorem}[Characterization of a sub-quasianalytic function]\label{thm:subquasian}
Let $\cQ$ denote a real quasianalytic class, and let $f:N\to \IR$ be a continuous
function. Then $f$ is sub-quasianalytic if and only if there is a countable semiproper covering 
$\{\tau_\la: W_\la \to N\}$ of $N$ of class $\cQ$, such that
\begin{enumerate}
\item each morphism $\tau_\la$ is a composite of finitely many smooth local blowings-up
and power substitutions;
\smallskip
\item each function $f_\la := f\circ\tau_\la$ is quasianalytic of class $\cQ$.
\end{enumerate}
\end{theorem}

\begin{proof} The ``if'' direction is clear. Conversely, suppose that $f$ is sub-quasianalytic. Then
there is a proper $\cQ$-morphism $\Psi: M\to N\times \IR$ such that $\dim M = \dim N$ and $\Psi(M) = \graph f$.
Let $\pi_N$ and $\pi_{\IR}$ denote
the projections of $N\times\IR$ to $N$ and $\IR$ (respectively), and set $\Phi := \pi_N\circ\Psi: M \to N$,
so that $f\circ\Phi = \pi_{\IR}\circ\Psi$ is quasianalytic.

By Theorem \ref{thm:transftoid}, there are countable semiproper coverings $\{\s_\la\}$ and $\{\tau_\la\}$
of $M$ and $N$ 
by finite sequences of local blowings-up
and power substitutions, such that, for each $\la$, $\Phi\circ\s_\la = \tau_\la$; therefore, 
$f\circ\tau_\la = (f\circ\Phi)\circ\s_\la$ is quasianalytic.
\end{proof}

\begin{remark}\label{rem:subquasian}
If $f$  is sub-quasianalytic and $N$ has an SNC divisor $E$, then there is a semiproper covering 
$\{\tau_\la\}$ as in the theorem, 
where each $\tau_\la$ is compatible with $E$.
\end{remark}

Theorem \ref{thm:subquasian} in the real
analytic case was proved in \cite{BMarcan}. Note that, even in the real-analytic case, the
analogous assertion using local blowings-up alone is not true. (A real-valued function $f$
can be transformed to analytic by sequences of local blowings-up of the source if and only if $f$ is subanalytic
and \emph{arc-analytic}; i.e., analytic on every real-analytic arc \cite{BMarcan}.) Theorem
\ref{thm:subquasian}, in general, has been proved by Rolin and Servi using model-theoretic techniques \cite[Thm.\,4.2]{RSer}.
Rolin and Servi show that quantifier elimination in the structure defined by restricted functions of class $\cQ$
(together with reciprocal and $n$th roots) is a direct consequence of Theorem \ref{thm:subquasian} (see
\cite[Sect.\,4]{RSer}). 
Theorems \ref{thm:rect1},  \ref{thm:rect2} following are strong versions for real quasianalytic classes, in general, 
of Hironaka's rectilinearization theorem for subanalytic sets \cite[Thm.\,7.1]{HiroPisa},
\cite[Thm.\,0.2]{BMihes}, and a variant in \cite[Sect.\,4]{RSer}.

\begin{theorem}[Rectilinearization I]\label{thm:rect1} 
Let $N$ denote a manifold of real quasianalytic
class $\cQ$, and let $X$ denote a sub-quasianalytic subset of $N$. Assume that $N$ has pure dimension $n$.
Then there is a countable semiproper covering 
$\{\tau_\la: W_\la \to N\}$ of $N$ of class $\cQ$, such that, for each $\la$,
\begin{enumerate}
\item $W_\la \cong \IR^n$ and  $\tau_\la$ is a composite of finitely many smooth local blowings-up
and power substitutions compatible with the exceptional divisors;
\smallskip
\item $\tau_\la^{-1}(X)$ is a union of orthants of $\IR^n_\la$.
\end{enumerate}
\end{theorem}

\begin{proof}
We can assume that $N = \IR^n$ and that $X = \bigcup_{k=1}^p (X_{k1}\backslash X_{k2})$,
where each $X_{kj}$ is closed sub-quasianalytic. For all $k$ and $j=1,2$, let $f_{kj}(x) = d(x, X_{kj})^2$, 
were $d$ denotes the Euclidean distance. Then each $f_{kj}$ is
sub-quasianalytic. By Theorem \ref{thm:subquasian}, there is a countable semiproper covering
$\{\tau_\la\}$ of $N$ of class $\cQ$ satisfying condition (1), such that, for all $k,\,j,\,\la$, $f_{kj}\circ\tau_\la$ is quasianalytic
of class $\cQ$. By resolution of singularities of $\prod_{k,j}f_{kj}\circ\tau_\la$, for each $\la$, 
after further local blowings-up if necessary, we can
assume that each $f_{kj}\circ\tau_\la$ is a monomial times a unit; thus each $\tau_\la^{-1}(X)$ is a union
of orthants in $\IR^n$. 
\end{proof}

\begin{theorem}[Rectilinearization II]\label{thm:rect2} 
Let $N$ denote a manifold of real quasianalytic
class $\cQ$, with SNC divisor $E$. Assume that $N$ has pure dimension $n$. Let $X$ denote a sub-quasianalytic 
subset of $N$. Then there is a countable semiproper covering 
$\{\tau_\mu: W_\mu \to N\}$ of $N$ of class $\cQ$, such that, 
\begin{enumerate}
\item for each $\mu$, $W_\mu$ is a copy\, $\IR^n_\mu$ of $\IR^n$ 
and  $\tau_\mu$ is a composite of finitely many smooth local blowings-up and power substitutions compatible 
with the divisors;
\smallskip
\item $X\backslash E = \cup_\mu \tau_\mu(\IR^{k(\mu)}_\mu\backslash E_\mu)$, where,  for each $\mu$, 
$\IR^{k(\mu)}_\mu$ is a coordinate subspace of $\IR^n_\mu$ (of dimension $k(\mu)$)
transverse to the divisor $E_\mu \subset \IR^n_\mu$, and $\tau_\mu$ restricts to 
an embedding on each connected component of\,  $\IR^{k(\mu)}_\mu\backslash E_\mu$.
\end{enumerate}
\end{theorem}

\begin{proof} 
Let $k= \dim (X\backslash E) \leq n$. Let
$\Phi: M\to N$ be a proper $\cQ$-morphism from a $\cQ$-manifold $M$ of dimension $k$, such that
$\Phi(M) \subset \overline{X\backslash E}$ with complement of dimension $<k$,
$\dim X\backslash (E \cup \Phi(M)) < k$, and $\Phi$ has generic rank $k$. 
By resolution of singularities of $\Phi^{-1}(E)$, we can assume that $M$ has an SNC divisor $D$ and
that $\Phi$ is a morphism $(M,D) \to (N,E)$.

By Theorem \ref{thm:transftoid}, there are countable semiproper 
coverings $\{\s_\la\}$ and $\{\tau_\la\}$ of $M$ and $N$ (respectively) by finite sequences of local blowings-up
and power substitutions, compatible with the exceptional divisors, 
such that, for each $\la$, $\Phi\circ\s_\la = \tau_\la\circ \Phi_\la$ and 
$\Phi_\la = \text{identity}_{V_\la}\times 0$. We can assume that, for each $\la$, $W_\la$ is a copy $\IR^n_\la$
of $\IR^n$ (a coordinate chart) and $\Phi_\la(V_\la)$ is a $k$-dimensional coordinate subspace $\IR^k_\la$ of $\IR^n_\la$,
transverse to the exceptional divisor $E_\la$ of $\tau_\la$.

Since all local blowings-up and power substitutions are compatible with the divisors,
we can assume that each $\tau_\la$ restricts to an embedding on every component of $\IR^n_\la\backslash E_\la$.
The result follows easily in the case that $X$ is closed, using induction on $k$.

In general, by induction on $k$, it is enough to find a sub-quasianalytic subset $Y$ of $X$ such that 
$\dim Y < k$ and $X\backslash Y$ satisfies the conclusion of the theorem. Let $\{\tau_\mu: W_\mu \to N\}$
denote a covering satisfying the conclusion of the theorem for $\overline{X}$. For each $\mu$, there is a subset
$X_\mu$ of $\IR^{k(\mu)}_\mu \backslash E_\mu$ given by all points of the latter
that are mapped to the complement of $X\backslash E$;
$X_\mu$ is sub-quasianalytic and $\dim X_\mu < k$.

After applying Theorem \ref{thm:subquasian}
and Remark \ref{rem:subquasian} to the sub-quasianalytic functions
$f_\mu(x)=d(x,X_\mu)^2$ on $\IR^{k(\mu)}_\mu$ and the SNC divisor $E_\mu$, and then monomializing
the pullbacks of the $f_\mu$ (exactly as in the proof of Theorem \ref{thm:rect1}), we obtain a covering
(for which we use the same notation $\{\tau_\mu: W_\mu \to N\}$) satisfying conditions (1) and (2) of the theorem,
except that now $ \cup_\mu \tau_\mu(\IR^{k(\mu)}_\mu\backslash E_\mu) \subset X\backslash E$, with complement
$Y$ in $X\backslash E$ of dimension $< k$, as required.
\end{proof}

\section{Counterexample to monomialization by global blowings-up}\label{sec:ex}

\begin{definition}\label{def:regular}
A morphism $\vp: V\to W$ of real-analytic manifolds is \emph{regular} at a point $a\in V$,
in the sense of Gabrielov \cite{Gab}, if $\dim \cO_{\vp(a)}/\ker \vp^*_a$ equals the generic
rank of $\vp$ at $a$. We recall that $\vp^*_a: \cO_{\vp(a)} \to \cO_a$ denotes the homomorphism
of local rings induced by $\vp$, and that $\dim \cO_{\vp(a)}/\ker \vp^*_a$ is the smallest dimension
of an analytic subset of some neighbourhood of $\vp(a)$ in $W$ that contains the image of a
small neighbourhood of $a$. We say that $\vp: V\to W$ is \emph{regular} if it is regular at 
every point of $V$.
\end{definition}

A monomial mapping (or any polynomial mapping) is regular since its image is
semialgebraic (with respect to monomial coordinates), and every semialgebraic set lies in a real algebraic set of the same
dimension, by a theorem of {\L}ojasiewicz \cite{Loj} (see \cite[Thm.\,2.13]{BMihes}).

We will construct a proper real-analytic morphism which cannot be transformed
to a regular morphism by global bowings-up of the source and target; see Example \ref{ex:nonreg} below.
This example, in fact, provides a stronger negative result: we construct a morphism $\Phi: S^3 = V \to W =\IR^4$ with
the following property. Given a commutative diagram
\begin{equation}\label{eq:transfreg1}
\begin{tikzcd}
V' \arrow{d}{\Phi'}\arrow{r}{\s} & V \arrow{d}{\Phi}\\
W' \arrow{r}{\tau} & W\,,
\end{tikzcd}
\end{equation}
where $\tau$ is a proper bimeromorphic morphism, and $\s: V'\to V$ is a proper surjective real-analytic mapping
of generic rank $3$ from a real-analytic space $V'$ of pure dimension $3$, then $\Phi'$ is \emph{not regular}.
We can even establish the preceding property after shrinking over any neighbourhood
of $0\in W$.

Example \ref{ex:nonreg}
a variation of \cite[Example 3.1]{BP}, and involves the following
two mappings.
\begin{enumerate}
\item The function $g: (-\infty, 1/\de\pi) \to \IR$ given by
$$
g(s) := \sin \theta(s), \  \text{ where }\  \theta(s) := \frac{1}{\de s - 1/\pi}\,,
$$
and $\de>0$ is small. The graph of $g(s)$ lies in no $1$-dimensional
subanalytic subset of $\IR^2$.
\item The \emph{Osgood mapping} $\Th: \IR^2\to\IR^3$ defined as
$$
\Th(u,v) := (u,uv,uv e^v).
$$
No nonzero analytic function defined in a neighbourhood of $0 \in \IR^3$ vanishes
of the image of any neighbourhood of $0\in\IR^2$.
\end{enumerate}

Using (1), it is easy to define a function $h: (-\infty, 1/\de\pi) \to \IR^2$
such that every point of the disk $\{(s,x,y)\in\IR^3: s=1/\de\pi,\, x^2+y^2\leq 1\}$ is a limit
point of the graph of $h(s)$, and therefore the graph of $h$ 
lies in no subanalytic subset of $\IR^3$ of dimension $< 3$; for example
\begin{equation}\label{eq:h}
h(s) := \sin \theta(s) \left(\cos \theta(s)^2, \sin \theta(s)^2\right) .
\end{equation}

\medskip
We will need Lemma \ref{lemma:bimero} below, on the indeterminacy locus of (the inverse of) a proper bimeromorphic 
morphism with smooth target. Definition of a real bimeromorphic morphism, in general, raises technical questions that we do not
address, but the following remark provides a working definition in the case of smooth target.

\begin{remark}\label{remark:bimero}
Let $\tau: Z' \to Z$ denote a morphism of real-analytic spaces (of the same dimension), where $Z$ is smooth
(say, of pure dimension). We say that $\tau$ is
\emph{bimeromorphic}
if there are lower-dimensional closed analytic subsets $Y\subset Z$, $Y'\subset Z'$ such that 
$\tau(Y') \subset Y$ and $\tau$ restricts
to an isomorphism $Z'\backslash Y' \to Z\backslash Y$, where the inverse $F = \tau^{-1}$ is a \emph{meromorphic
mapping} $Z \dashrightarrow Z'$. The latter means that, for any $a\in Z'$, if we write
$F = (F_1,\ldots,F_N)$ in a neighbourhood of $b=\tau(a)$ with respect to coordinates in a local smooth embedding space for $Z'$ at $a$,
then each component function $F_j$ of $F$ extends to some neighbourhood of $b$ as a quotient
of (locally defined) analytic functions. The \emph{indeterminacy locus} $E$ of $F$ is the subset of $Z$
consisting of points to which $F$ does not admit an analytic extension; $E$ is a closed analytic subset
of $Y$. Let $D:=\tau^{-1}(E)$. If $\tau$ is \emph{proper}, then $\tau(D)=E$.

For example, if $Z'$ denotes Whitney's umbrella $\{x^2-y^2z = 0\}$, then the projection $\tau$
of $Z'$ to the $(x,y)$-plane $Z$ is bimeromorphic (in fact, birational). The indeterminacy locus $E$ of $\tau^{-1}$
is $\{y=0\}$; note that $\tau(D) = \{0\}$ and  $Z'\backslash D$ is not dense in $Z'$.

\end{remark}

\begin{lemma}\label{lemma:bimero}
Let $\tau: Z' \to Z$ denote a proper bimeromorphic morphism of real-analytic spaces,
where $Z$ is smooth. Then the indeterminacy locus of $\tau^{-1}$ has codimension $\geq 2$.
\end{lemma}

\begin{proof}
Let $E$ denote the indeterminacy locus of $F = \tau^{-1}$.
Suppose that $E$ has codimension $1$. Let $D := \tau^{-1}(E)$. Since $\tau$ is proper, $\tau(D)=E$,
$\codim D = 1$ and there is an open subset $V$ of $D$ at any point of which $\tau: D\to E$
is a local isomorphism of codimension $1$ submanifolds.

Let $a\in V$ and let $b=\tau(a)$. Write $F = (F_1,\ldots,F_N)$ in a neighbourhood of $b$, 
as in Remark \ref{remark:bimero}. Then there are local coordinates 
$(x,y) = (x,y_1,\ldots,y_n)$ for $Z$ at $b$ in which
$E= \{x=0\}$ and any component $F_j$ of $F$ can be expressed in a neighbourhood of $b$ as a \emph{bounded} quotient of
real-analytic functions $f(x,y)/g(x,y)$. Then, after factoring $x$ to a high enough power from the numerator and denominator,
we can assume that $g(0,y)$ is not identically zero, contradicting the condition that all nearby points of $E$ are indeterminacies.
\end{proof}

\begin{example}\label{ex:nonreg}
Define $\Phi: V \to W$ by
$\Phi := \pi_0\circ \psi$, where $\pi_0: Z\to W:=\IR^4$ is the blowing-up of the origin,
and $\psi: V:=S^3 \to Z$ is the composite of
\begin{align*}
p: S^3 &\to \IR^4\\
(s,u,v,w) &\mapsto (s, u+h_1(s), uv + h_2(s), uve^v),
\end{align*}
where $S^3 = \{(s,u,v,w): s^2+u^2+v^2+w^2 = 1\}$, 
$h(s) = (h_1(s),h_2(s))$ is given by \eqref{eq:h}, and
$$
\io: \IR^4 \hookrightarrow Z,
$$
where $\io$ is the inclusion of the ``$z$-coordinate chart''; i.e., the coordinate chart on which $\pi_0$
is given by $(t,x,y,z) \mapsto (tz, xz, yz, z)$ ($\{z=0\}$ is the exceptional divisor of $\pi_0$
in this chart). Note that
$\ker p^*_a = 0$ at every $a \in \{u=v=0\}$.

Consider a commutative diagram \eqref{eq:transfreg1} as above. We claim that
$\Phi'$ is not \emph{regular}; more precisely, 
$\ker (\Phi')^*_a = 0$ at certain points $a\in V'$. To see this,
consider the commutative diagram
\begin{equation}\label{eq:strtransf}
\begin{tikzcd}
Z' \arrow{d}{\pi}\arrow{r}{\tau'} & Z \arrow{d}{\pi_o}\\
W' \arrow{r}{\tau} & W\,,
\end{tikzcd}
\end{equation}
where $\pi$ is the strict transform of $\pi_0$; i.e., $\pi$ is the induced morphism from the smallest
closed subspace of the fibre product $W' \times_W Z$ containing the complement of $p_1^{-1}(\tau^{-1}(0)$, where
$p_1: W' \times_W Z \to W'$ is the projection. Then $\tau'$ is a proper bimeromorphic morphism, so that, by
Lemma \ref{lemma:bimero}, the indeterminacy locus $E\subset Z$ of $\tau'$ has codimension $\geq 2$.
(In the special case that $\tau$ is a composite of finitely many smooth blowings-up, there is a commutative
diagram \eqref{eq:strtransf}, where $\tau'$ is also a finite composite of smooth blowings-up.)

Note that $\dim Z' = 4$ and that $\pi$ has rank $4$ outside an analytic subset of dimension $<4$. In
particular, the homomorphism of local rings $\pi_b^*$ is injective, for all $b$ in (the $4$-dimensional part of) $Z'$.
Let $D = (\tau')^{-1}(E)$; thus $\tau'$ is an isomorphism outside $D$.

Now consider the commutative diagram 
\begin{equation}\label{eq:transfreg2}
\begin{tikzcd}
V'' \arrow{rd}{\psi'} \arrow[hookrightarrow]{r} & V'\times_{W'} Z' \arrow{d}{\vp} \arrow{r} & V'\arrow{d}{\Phi'}\\
& Z' \arrow{r}{\pi} & W'\,,
\end{tikzcd}
\end{equation}
where $V''$ is the smallest closed analytic subspace of the fibre-product $V'\times_{W'} Z'$ containing the
complement of $\vp^{-1}(D)$; $\dim V'' = 3$. Let $\psi': V''\to Z'$ and
$\s': V'' \to V'$ denote the induced morphisms.

Note that, if $\Phi'$ is regular, then $\psi'$ is regular, at least at every point $c$ in the $3$-dimensional part of $V''$, since
any nonzero element of $\ker (\Phi')^*_{\s'(c)} \subset \cO_{W', \Phi'(\s'(c))}$ pulls
back to a nonzero element of $\ker (\psi')^*_c$. If $U$ is an open subset of $Z$ over which $\tau'$ is an
isomorphism (which we regard as the identity mapping), then $\Phi'|_{(\psi\circ\s)^{-1}(U)}$ factors
through $Z'$ as $\pi\circ (\psi\circ\s)$, and, over $U$, we can identify $V''$ with $V'$, and
$\psi'$ with $\psi\circ\s$.

Let $\Ga\subset Z$ denote the curve 
$\{(t,x,y,z)\in \IR^4: x=h_1(t),\,y = h_2(t),\, z=0,\,t < 1/\de\pi\}$. The curve $\Ga$ cannot lie entirely in the
indeterminacy locus $E$ of $\tau'$, since any subanalytic set containing $\Ga$ is of dimension $\geq 3$.
Therefore, $\Ga$ lifts to a unique curve $\Ga' \subset Z'$, and $\Ga'$ intersects $D$ in a discrete
set. It follows that 
$\psi'$ is not regular at most points of the inverse image of $\Ga'$. Therefore, $\Phi'$ is not regular.
\end{example}

\section{Quasianalytic classes}\label{sec:quasian}

Quasianalytic classes are classes of infinitely differentiable functions which can 
be characterized by three simple axioms. We list the axioms in \S\ref{subsec:quasian}
on real quasianalytic classes below (see Definition \ref{def:quasian}), but the classes
of compex-analytic functions or algebraic functions over any field $\IK$ of characteristic zero
also satisfy these axioms. Derivatives in the complex analytic case are with respect to
complex analytic (i.e., holomorphic) coordinates, and in the algebraic case with respect
to \'etale coordinates (or uniformizing parameters).

Throughout the article, we try to use a language that makes sense for all classes considered. This
presents a challenge in the algebraic case (particularly when the ground field $\IK$ is not algebraically
closed), because local neighbourhoods and local blowings-up mean local in the \'etale topology, 
and the value of a function at a closed point $a$ makes sense only as an
element of the residue field $\IK_a$, which is a finite
extension of $\IK$. Some of the background notions for the algebraic case are recalled
in \S\ref{subsec:alg}. The use of \'etale neighbourhoods and the necessary field extensions
will usually be implicitly understood in the remainder of the paper, but we will try to provide
guidance where it seems important.

In the remaining sections of the article, when we deal simultaneously with all classes
considered, we will use $\cQ$ to denote the class of functions. In particular,
in the analytic or algebraic cases, $\cQ$ will be used for the classes $\cO$ of analytic
or regular functions. We will use the notation $\cO$ only when the discussion is
restricted to the analytic or algebraic cases. We will write $\cQ_M$ (or simply $\cQ$, if
there is no possibility of confusion) to denote the sheaf of local rings of functions of
class $\cQ$ on a smooth space $M$ in the category. See \S\ref{subsec:quasian} following.

We use standard multiindex notation: Let $\IN$ denote the nonnegative integers. If $\al = (\al_1,\ldots,\al_n) \in
\IN^n$, we write $|\al| := \al_1 +\cdots +\al_n$, $\al! := \al_1!\cdots\al_n!$, $x^\al := x_1^{\al_1}\cdots x_n^{\al_n}$,
and $\p^{|\al|} / \p x^{\al} := \p^{\al_1 +\cdots +\al_n} / \p x_1^{\al_1}\cdots \p x_n^{\al_n}$. We write $(i)$ for the
multiindex with $1$ in the $i$th place and $0$ elsewhere.

\subsection{Real quasianalytic classes}\label{subsec:quasian}
We consider a class of functions $\cQ$ given by the association, to every 
open subset $U\subset \IR^n$, of a subalgebra $\cQ(U)$ of $\cC^\infty (U)$ containing
the restrictions to $U$ of polynomial functions on $\IR^n$, and closed under composition 
with a $\cQ$-mapping (i.e., a mapping whose components belong to $\cQ$). 
We assume that $\cQ$ determines a sheaf of local $\IR$-algebras of $\cC^\infty$ functions on $\IR^n$,
for each $n$, which we denote $\cQ_{\IR^n}$ (or simply $\cQ$).

\begin{definition}\label{def:quasian}\emph{Quasianalytic classes.}
We say that $\cQ$ is \emph{quasianalytic} if it satisfies the following three axioms:

\begin{enumerate}
\item \emph{Closure under division by a coordinate.} If $f \in \cQ(U)$ and
$$
f(x_1,\dots, x_{i-1}, a, x_{i+1},\ldots, x_n) = 0,
$$
where $a \in \IR$,  then $f(x) = (x_i - a)h(x),$ where $h \in \cQ(U)$.

\smallskip
\item \emph{Closure under inverse.} Let $\varphi : U \to V$
denote a $\cQ$-mapping between open subsets $U$, $V$ of $\IR^n$.
Let $a \in  U$ and suppose that the Jacobian matrix
$(\partial \varphi/\partial x)(a)$ is invertible. Then there are neighbourhoods $U'$ of $a$ and $V'$ of 
$b := \varphi(a)$, and a $\cQ$-mapping  $\psi: V' \to U'$ such that
$\psi(b) = a$ and $\psi\circ \varphi$  is the identity mapping of
$U '$.

\smallskip
\item \emph{Quasianalyticity.} If $f \in \cQ(U)$ has formal Taylor expansion zero
at $a \in U$, then $f$ is identically zero near $a$.
\end{enumerate}
\end{definition}

\begin{remarks}\label{rem:axioms} (1)\, Axiom \ref{def:quasian}(1) implies that, 
if $f \in \cQ(U)$, then all partial derivatives of $f$ belong to $\cQ(U)$. 

\smallskip\noindent
(2)\, Axiom \ref{def:quasian}(2) is equivalent to the property that the implicit function theorem holds for functions of 
class $\cQ$.  It implies that the reciprocal of a nonvanishing function of class $\cQ$ is also of class $\cQ$.

\smallskip\noindent
(3)\, Our two main examples of quasianalytic classes are quasianalytic Denjoy-Carleman classes (see \S\ref{subsec:DC}),
and the class of $\cC^\infty$ functions definable in a given polynomially bounded
$o$-minimal structure. In the latter case, we can define a quasianalytic class $\cQ$ in the axiomatic
framework above by taking $\cQ(U)$ as the subring of $\cC^\infty(U)$ of functions
$f$ such that $f$ is definable in some neighbourhood of any point of $U$ (or, equivalently,
such that $f|_V$ is definable, for every relatively compact definable open $V\subset U$); 
the axiom of quasianalyticity is satisfied by \cite{Mil}, and the division and inverse properties
are immediate from definability and the corresponding $\cC^\infty$ assertions.
\end{remarks} 

The elements of a quasianalytic class $\cQ$ will be called \emph{quasianalytic functions}. 
A category of manifolds and mappings of class $\cQ$ can be defined in a standard way. The category 
of $\cQ$-manifolds is closed under blowing up with centre a $\cQ$-submanifold \cite{BMselecta}.

Resolution of singularities holds in a quasianalytic class \cite{BMinv,BMselecta}. Resolution of
singularities of a sheaf of ideals requires only that the ideal (sheaf) be privileged (Definition \ref{def:priv}), rather than
finitely generated; see \cite[Thm.\,3.1]{BMV}.

\subsection{Quasianalytic Denjoy-Carleman classes}\label{subsec:DC}

\begin{definition}\label{def:DC}
Let $M = (M_k)_{k\in \IN}$ denote a sequence of positive real numbers such that $M_0\leq M_1$ and $M$ is \emph{logarithmically
convex}; i.e., the sequence $(M_{k+1} / M_k)$ is nondecreasing.
A \emph{Denjoy-Carleman
class} $\cQ = \cC_M$ is a class of $\cC^\infty$ functions determined by the following condition: A function 
$f \in \cC^\infty(U)$ (where $U$ is open in $\IR^n$) is of class $\cC_M$ if, for every compact subset $K$ of $U$,
there exist constants $A,\,B > 0$ such that
\begin{equation}\label{eq:DC}
\left|\frac{\p^{|\al|}f}{\p x^{\al}}\right| \leq A B^{|\al|} \al! M_{|\al|}
\end{equation}
on $K$, for every $\al \in \IN^n$.
\end{definition}

We use the notation $M$ to denote a sequence (as opposed to a manifold or smooth space)
only in this section and in Examples \ref{ex:halfline}, in order to be consistent with standard
notation for Denjoy-Carleman classes.

\begin{remark}\label{rem:DC}
The logarithmic convexity assumption implies that
$M_jM_k \leq M_0M_{j+k}$, for all $j,k$, and that
the sequence $((M_k/M_0)^{1/k})$ is nondecreasing.
The first of these conditions guarantees that $\cC_M(U)$ is a ring, and 
the second that $\cC_M(U)$ contains the ring $\cO(U)$ of real-analytic functions on $U$,
for every open $U\subset \IR^n$.
(If $M_k=1$, for all $k$, then $\cC_M = \cO$.)
\end{remark}

If $X$ is a closed subset of $U$, then $\cC_M(X)$ will denote the ring of restrictions to $X$
of $\cC^\infty$ functions which satisfy estimates of 
the form \eqref{eq:DC}, for every compact $K\subset X$.

A Denjoy-Carleman class $\cQ = \cC_M$ is a quasianalytic class in the sense of Definition \ref{def:quasian}
if and only if the sequence
$M = (M_k)_{k\in \IN}$ satisfies the following two assumptions in addition to those
of Definition \ref{def:DC}.
\begin{enumerate}
\item[(a)] $\displaystyle{\sup \left(\frac{M_{k+1}}{M_k}\right)^{1/k} < \infty}$.

\smallskip
\item[(b)] $\displaystyle{\sum_{k=0}^\infty\frac{M_k}{(k+1)M_{k+1}} = \infty}$.
\end{enumerate}

It is easy to see that the assumption (a) implies that $\cC_M$ is closed under differentiation.
The converse of this statement is due to S. Mandelbrojt \cite{Mandel}. In a Denjoy-Carleman class
$\cC_M$, closure under differentiation is equivalent to the axiom \ref{def:quasian}(1) of closure under division by a
coordinate---the converse of Remark \ref{rem:axioms}(1) is a consequence of the fundamental
theorem of calculus.

According to the Denjoy-Carleman theorem, the class $\cC_M$ is quasianalytic (axiom \ref{def:quasian}(3))
if and only if the assumption (b) holds \cite[Thm.\,1.3.8]{Horm}.

Closure of a Denjoy-Carleman class $\cC_M$ under composition is due to Roumieu \cite{Rou} and closure under
inverse to Komatsu \cite{Kom}; see \cite{BMselecta} for simple proofs. A Denjoy-Carleman class $\cQ = \cC_M$ satisfying 
the assumptions (a) and (b)
above is thus a quasianalytic class, in the sense of Definition \ref{def:quasian}.

If $\cC_M$, $\cC_N$ are Denjoy-Carleman classes, then $\cC_M(U) \subseteq \cC_N(U)$, for all $U$,
if and only if $\sup \left(M_k /N_k\right)^{1/k} < \infty$ (see \cite[\S1.4]{Th1}); in this case, we write
$\cC_M \subseteq \cC_N$. For any given Denjoy-Carleman class $\cC_M$, there is a function in 
$\cC_M((0,1))$ which is nowhere in any given smaller class \cite[Thm.\,1.1]{Jaffe}.

\begin{remark}\label{rem:model} \emph{Model theory.}
Let $\cC_M$ denote a quasianalytic Denjoy-Carleman class, and let $\IR_{\cC_M}$ denote the expansion 
of the real field by
restricted functions of class $\cC_M$ (i.e., restrictions to closed cubes of $\cC_M$-functions, extended 
by $0$ outside the cube).
Then $\IR_{\cC_M}$ is an $o$-minimal structure, and $\IR_{\cC_M}$ is both polynomially bounded 
and model-complete \cite{RSW}. 
\end{remark}

\subsection{The algebraic case}\label{subsec:alg}
A \emph{manifold} in this case will mean a smooth variety, i.e., a smooth separated
scheme of finite type over a field $\IK$ of characteristic zero, 
and a \emph{morphism} will mean a regular morphism, i.e., a
morphism of such schemes. We will call a regular morphism also an \emph{algebraic morphism}
(or an \emph{algebraic function} if the target scheme is the affine line)---perhaps a mild abuse
of terminology, but convenient in analogy with ``analytic'' or ``quasianalytic'' morphism.

In the case of algebraic (regular) functions over $\IK$,
formal Taylor expansions and partial derivatives are defined with respect to \'etale coordinates
at a closed point $a$ of a manifold $M$. 
It is enough to prove our local monomialization theorems at a closed point $a$ because
any open covering of the set of closed points of a closed subset $X$ of $M$ is an open
covering of $X$.

The \emph{value} $f(a)$ of an algebraic function $f$ at a (closed) point $a\in M$
makes sense as the element induced by $f$ in $\cO_a/\um_a$; 
$\IK_a := \cO_a/\um_a$ is the \emph{residue field} of $a$, and is a finite extension of the base field $\IK$
($\cO$ denotes the structure sheaf of $M$).
Note that, if $\Phi: M \to N$ is a morphism and $b=\Phi(a)$, then $\IK_b \subset \IK_a$,
but they need not be equal, in general.

An \emph{\'etale coordinate chart} $U \subset M$ is a Zariski-open subset $U$ of $M$
together with an \'etale morphism $U \to \IA^n_{\IK}$. The affine coordinates $x_1,\ldots,x_n$ of
$\IA^n_{\IK}$ are regular functions on $U$; $(x_1,\ldots,x_n)$ is called a system of \emph{\'etale
coordinates} or \emph{uniformizing parameters} (see \cite[Lemma 3.3]{BMinv} for an
explanation).

A formula in \'etale coordinates for a morphism $\Phi: U\to V$, where $U$ and $V$ are \'etale coordinate
charts at $a$ and $b=\phi(a)$ (respectively) should be understood as a formula for the
induced morphism $\Phi\times_{\IK}\IL: U\times_{\IK}\IL\to V\times_{\IK}\IL$, where $\IL$ is either the
residue field $\IK_a$, or a finite extension of $\IK_a$.

An \'etale coordinate chart $U$ determines a \emph{formal Taylor series homomorphism}
$T_a: \cO_a \to \IK_a\llb X\rrb$, $X=(X_1,\ldots,X_n)$, at every $a\in U$; $T_a$ is the unique
ring homomorphism $\cO_a \to \IK_a\llb X\rrb$ such that $T_a x_i = x_i(a) + X_i$, $i=1,\ldots,n$. The
Taylor homomorphism induces an isomorphism of complete local rings $\wcO_a \to \IK_a\llb X\rrb$.

\emph{Partial derivatives} of a regular (or algebraic) function on $U$ with respect to \'etale coordinates 
$(x_1,\ldots,x_n)$ are regular functions on $U$ which can be defined as follows. If $\al \in \IN^n$ and
$f\in \cO(U)$, then there is a unique element $f_\al \in\cO(U)$ such that $D^\al(T_a f)(X) 
= (T_a f_\al)(X)$, for all $a\in U$, where $D^\al$ denotes the formal derivative $\p^{|\al|}/\p X^\al$.
We call $f_\al$ the \emph{partial derivative} $\p_\al f$ or $\p^{|\al|}f/\p x^\al$.
See \cite[Section 3]{BMinv} for a concrete description of the notions above; in particular, $\p_\al$ is given 
by an explicit formula in \cite[Lemma 3.5]{BMinv}.

It is important to distinguish between the notions of \'etale coordinate chart, as above,
and \emph{\'etale open neighbourhood} of a point or \emph{\'etale open subset} of $M$,
which means not a (Zariski-)open subset at all, but rather an \'etale mapping $\varepsilon: V \to M$
(i.e., the analogue for the \emph{\'etale topology} of a Zariski-open subset). An \'etale mapping
is open in the Zariski topology, so that the image of an \'etale mapping $\varepsilon$ is a Zariski-open subset.
An \'etale neighbourhood $\varepsilon: V \to M$ and \'etale coordinates (or uniformizing parameters)
at a point of $V$ will often be used together; cf.\ Remark \ref{rem:mon1}(2).

\section{Logarithmic Fitting ideal characterization of a monomial morphism}\label{sec:fit}

The purpose of this section is to show that a dominant morphism $\Phi: (M,D) \to (N,E)$ of
quasianalytic class $\cQ$ is monomial if and only if the log Fitting ideal of maximal minors
(of the log Jacobian matrix) of $\Phi$ is everywhere generated by a unit 
(Theorem \ref{thm:LogFittingCharacterization}---a version of the \emph{rank theorem}
for logarithmic dervatives; cf.\ \cite[Ch.\,IV]{Ogus}). 
This characterization of monomiality is used, for example, to show that monomial is an open condition
in the source of a morphism (Corollary \ref{cor:MonOpen}).

Let $(M,D)$ denote a manifold with SNC divisor, of class $\cQ$. We denote by $\Omega^1_M(-\log D)$ 
the sheaf of modules over the structure sheaf $\mathcal{Q}$ of $M$ of logarithmic differential 1-forms. In coordinates
\begin{equation}\label{eq:logcoords1}
(\bu,\bw) = (u_1,\ldots,u_r, w_1,\ldots,w_t)
\end{equation}
compatible with $D$ at a point $a\in M$ (where $D=\{u_1 \cdots u_r =0\}$), the local sections of $\Omega^1_M(-\log D)$ 
are generated by
 \begin{equation*}
 \frac{du_1}{u_1}, \ldots, \frac{du_r}{u_r},\,  dw_{1}, \ldots, dw_{t}.
 \end{equation*}
The sheaf $\Omega^k_M(-\log D)$ of logarithmic $k$-forms over $M$ is defined in terms of $\Omega^1_M(-\log D)$
in the standard way.

Consider a $\cQ$- morphism $\Phi: (M,D) \to (N,E)$. Let $n=\dim N$.

\begin{lemma}[Pullbacks of log differentials are well-defined]\label{lem:LogFittingWellDefined}
For all $k=1,\ldots,n$,
$\sigma^{\ast}\left(  \Omega^k_N(-\log E) \right)$ is a subsheaf of $\Omega^k_M(-\log D)$.
\end{lemma}

\begin{proof}
Consider coordinate systems (compatible with the divisors) \eqref{eq:logcoords1} at a point $a\in M$ and
\begin{equation}\label{eq:logcoords2}
(\bx,\bz) = (x_1,\ldots,x_p,z_{1},\ldots, z_s)
\end{equation}
at $b=\Phi(a)\in N$, where $E=\{x_1\cdots x_p=0\}$.
Since $\Phi^{-1}(E)$ is SNC and $\Phi^{-1}(E)_{\reduced} \subset D$,
each $x_j = \pmb{u}^{\pmb{\gamma}_j} U_j$ at $a$,
for some $\pmb{\gamma}_j\in\IN^r$ and unit $U_i$. Therefore,
\[
\frac{dx_j}{x_j} = \frac{d \pmb{u}^{\pmb{\gamma}_j}}{\pmb{u}^{\pmb{\gamma}_j}} + \frac{d U_j}{U_j},\quad j=1,\ldots,p,
\]
and the assertion follows.
\end{proof}

\begin{definition}\label{def:LogFittingIdeal}
The \emph{logarithmic Fitting ideal sheaf} $\mathcal{F}_{n-k}(\Phi)$ associated to $\Phi$ is the ideal subsheaf of $\mathcal{Q}$ 
whose stalk $\mathcal{F}_{n-k}(\Phi)_a$ at $a\in M$ can be described (in coordinates \eqref{eq:logcoords1} at $a$)
in the following way. 
If $\omega$ is a logarithmic $k$-form at $b=\Phi(a)$, then
\begin{equation*}
\Phi^*\omega = \sum_{I,J} B^{\omega}_{I,J}(\pmb{u},\pmb{w}) \frac{du_{i_1}}{u_{i_1}}\wedge\cdots\wedge\frac{du_{i_l}}{u_{i_l}}
\wedge dw_{j_1} \wedge\cdots\wedge dw_{j_{k-l}},
\end{equation*}
where the sum is over all pairs $(I,J)$ with $I = (i_1,\ldots i_l),\, 1\leq i_1<\cdots < i_l \leq r$, and
$J = (j_1,\ldots <  j_{k-l}),\, 1\leq j_1<\cdots j_{k-l}\leq t$, 
and where the coefficients $B^{\omega}_{I,J}$ are germs of functions of class $\cQ$ at $a$, by Lemma \ref{lem:LogFittingWellDefined}. Then $\mathcal{F}_{n-k}(\Phi)_a$ is generated by the set of coefficients
$B_{I,J}^{\omega}(\pmb{u},\pmb{w})$, for all $\omega \in \Omega_{N,b}^k(-\log E)$ and all $(I,J)$.
\end{definition}

\begin{remark}\label{rem:dominant} A morphism $\Phi: (M,D) \to (N,E)$ is dominant (i.e., generically submersive)
if and only if the log Fitting ideal $\mathcal{F}_0(\Phi)$ of maximal minors is everywhere non-zero.
\end{remark}

\begin{theorem}\label{thm:LogFittingCharacterization}
A dominant $\cQ$-morphism $\Phi: (M,D) \to (N,E)$ is monomial at a point $a\in M$ if and only if the 
log Fitting ideal $\mathcal{F}_{0}(\Phi)_a$ is generated by a unit.
\end{theorem}

\begin{proof}
Let $b=\Phi(a)$. 
First assume that $\Phi$ is monomial at $a$. Consider $\Phi$-monomial coordinate systems $(\pmb{u},\pmb{v},\pmb{w})$ and $(\pmb{x},\pmb{y},\pmb{z})$ at $a$ and $b$, respectively (as in Definitions \ref{def:monom1}, \ref{def:monom2}). Since
$\Phi$ is dominant, $s'=s$ and
\begin{align*}
\frac{dx_j}{x_j} &= \frac{d\pmb{u}^{\pmb{\alpha}_j}}{\pmb{u}^{\pmb{\alpha}_j}} = \sum_{i=1}^{r} \alpha_{ji} \frac{d u_i}{u_i},\\
\frac{dy_k}{y_k}&= \frac{d\left(\pmb{u}^{\pmb{\beta}_k}(\xi_k + v_k)\right)}{\pmb{u}^{\pmb{\beta}_k}(\xi_k+v_k)} = \sum_{i=1}^{r} \beta_{ji} \frac{d u_i}{u_i} + \frac{dv_k}{\xi_k +v_k},\\
dz_l &=dv_l.
\end{align*}
It is easy to see that $\mathcal{F}_0(\Phi)_a$ is generated by a unit, using the conditions that the $\pmb{\alpha}_j$ are 
$\mathbb{Q}$-linearly independent, each $\pmb{\beta}_k$ is linearly dependent on $\{\pmb{\alpha}_j\}$ and each $\xi_k\neq 0$.

Secondly, assume that $\mathcal{F}_{0}(\Phi)_a$ is generated by a unit. Consider coordinate system $(\pmb{u},\pmb{w})$
at $a$ and $(\pmb{x},\pmb{z})$ at $b$, compatible with $D$ and $E$, as in \eqref{eq:logcoords1}, \eqref{eq:logcoords2}. 
By hypothesis, one of the coefficients of
\[
\Phi^*\left(\frac{dx_1}{x_1} \wedge \ldots \frac{dx_p}{x_p} \wedge dz_{1}\wedge \ldots \wedge d z_s\right)
\]
is a unit at $a$. In particular (using Lemma \ref{lem:LogFittingWellDefined}), one of the coefficients of
\[
\Phi^*\left(dz_{1}\wedge \ldots \wedge d z_{s}\right)
\]
is a unit. Write $z_k = f_k(\pmb{u},\pmb{w})$. Then
\[
dz_k = \sum_{i=1}^r \partial_{u_i}(f_k) u_i \frac{du_i}{u_i} + \sum_{j=1}^{t} \partial_{w_j}(f_k) dw_j;
\]
and the coefficients in the first sum are not units. By the rank theorem (submersion theorem), 
we can assume that:
\[
z_k = w_k, \quad k=1, \ldots, s.
\]

On the other hand, we can write
\[
x_j = \pmb{u}^{\pmb{\gamma}_j}\left(\xi_j + h_j(\pmb{u},\pmb{w}) \right),
\]
where $\xi_j \neq 0$ is a constant and $h_j(0,0) =0$, $j=1, \ldots, p$. After a change of coordinates, we can assume 
that $x_1 = \pmb{u}^{\pmb{\gamma}_1}$. We will prove the following assertion by induction on $l=1, \ldots, p$;
the case $l=p$ means that $\Phi$ is monomial at $a$.

\begin{claim}[$l$]\label{claim}
There exists $d_l$, $0\leq d_{l} < l$, a coordinate change in $(\bu, \bw)$ and a re-ordering of the variables $x$, such that, writing  $(x_1,\ldots,x_{l})= (x_1,\ldots,x_{l-d_{l}},y_1, \ldots, y_{d_{l}})$ we have
\[
\begin{aligned}
x_j &= \pmb{u}^{\pmb{\alpha}_j}, &\quad &j=1,\ldots, l-d_{l},\\
y_k &= \pmb{u}^{\pmb{\beta}_k}(\xi_k + v_k), &\quad& k=1,\ldots, d_{l},
\end{aligned}
\]
in coordinates $(\bu,\bw) = (u_1,\ldots,u_r,v_1,\ldots,v_{d_l},w_1,\ldots,w_{t-d_l})$ at $a$,
where the $\bal_j$ are $\mathbb{Q}$-linearly independent, each $\pmb{\beta}_k$ is $\mathbb{Q}$-linearly dependent
on $\{\pmb{\alpha}_j\}$, and each $\xi_k \neq 0$. 
\end{claim}
\medskip
\noindent
To prove Claim \ref{claim}($l+1$) assuming \ref{claim}($l$), 
consider the log differential form
\[
dz_{1}\wedge \ldots \wedge d z_{s} \wedge \frac{dx_1}{x_1} \wedge \ldots\wedge  \frac{dx_{l-d_{l}}}{x_{l-d_{l}}} 
\wedge \frac{dy_{1}}{y_1}\wedge \ldots \wedge \frac{dy_{d_{l}}}{y_{d_l}} \wedge \frac{dx_{l+1}}{x_{l+1}},
\]
which includes a unit as coefficient, and note that
\[
\frac{dx_{l+1}}{x_{l+1}} =  \frac{d\pmb{u}^{\pmb{\gamma}_{l+1}}}{\pmb{u}^{\pmb{\gamma}_{l+1}}}  
+ \frac{1}{\left(\xi_{l+1} + h_{l+1}(\pmb{u},\pmb{v})\right)} \left( \sum_i \partial_{u_i}(h_{l+1}) u_i \frac{d u_i}{u_i} 
+ \sum_k \partial_{v_k}(h_{l+1}) d v_k\right).
\]
We see that
\begin{enumerate}
\item either $\pmb{\gamma}_{l+1}$ is $\mathbb{Q}$-linearly independent of all $\pmb{\alpha}_j$, in which case there is
a coordinate change in $(\bu, \bv)$ preserving the previous forms, after which
$x_{l+1} = \pmb{u}^{\pmb{\gamma}_{l+1}}$,
and we take $d_{l+1} = d_l$, $\pmb{\alpha}_{l+1-d_{l}} = \pmb{\gamma}_{l+1}$;
\item or there exists an index $\kappa \geq s$ such that $\partial_{w_{\kappa}}(h_{l+1})$ is a unit. In this case, by
the implicit function theorem, we can change coordinates so that
\[
x_{l+1} = \pmb{u}^{\pmb{\gamma}_{l+1}}(\xi_{l+1}+w_{\kappa}),
\]
and we can take $d_{l+1} = d_l +1$, $\pmb{\beta}_{d_{l+1}} = \pmb{\gamma}_{l+1}$ and $v_{d_{l+1}} = w_{\kappa}$.
\end{enumerate}

\vspace{-\baselineskip}
\end{proof}

Theorem \ref{thm:LogFittingCharacterization} has two important corollaries.

\begin{corollary}[Monomiality is an open property in the source]\label{cor:MonOpen}
Let $\Phi: (M,D) \to (N,E)$ denote a $\cQ$-morphism. Suppose that $\Phi$ is monomial at $a \in M$. Then there is 
neighborhood $U$ of $a$ such that $\Phi|_{U}$ is everywhere monomial.
\end{corollary}

\begin{proof}
If $\Phi$ is dominant, the result is an immediate consequence of Theorem \ref{thm:LogFittingCharacterization}. 
In general, we can write $\Phi$ at $a$ as in Definitions \ref{def:monom1}, \ref{def:monom2}. Then it follows from
Theorem \ref{thm:LogFittingCharacterization} that $\mathcal{F}_{n-(p+q+s)}(\Phi)_a$ is generated by a unit
and then that $\Phi$ is monomial in a neighbourhood of $a$.
\end{proof}

\begin{corollary}[Monomiality is preserved by combinatorial blowings-up and power substitutions]\label{cor:MonCombinatorialBU}
Let $\Phi:(M,D) \to (N,E)$ be a monomial morphism.
\begin{enumerate}
\item If $\sigma: (\widetilde{M},\widetilde{D}) \to (M,D)$ is a combinatorial blowing-up, then the 
morphism $\widetilde{\Phi} = \Phi \circ \sigma$ is monomial.
\smallskip
\item If $\sigma: (\widetilde{U},\widetilde{D}) \to (U,D)$ is a power substitution (over a suitable coordinate chart $U$ of $M$), 
then the morphism $\widetilde{\Phi} = \Phi \circ \sigma$ is monomial.
\smallskip
\item If $\tau: (\widetilde{N},\widetilde{E}) \to (N,E)$ is a combinatorial blowing-up and there exists a morphism 
$\widetilde{\Phi}: (M,D) \to (\widetilde{N},\widetilde{E})$ such that $\tau \circ \widetilde{\Phi} = \Phi$, then 
$\widetilde{\Phi}$ is monomial.
\smallskip
\item If $\tau: (\widetilde{V},\widetilde{E}) \to (V,E)$ is a power substitution (over a coordinate chart $V$ of $N$) and
there exists a morphism $\widetilde{\Phi}: (M,D) \to (\widetilde{V},\widetilde{E})$ such that $\tau \circ \widetilde{\Phi} = \Phi$,
then $\widetilde{\Phi}$ is monomial.
\end{enumerate}
\end{corollary}

\begin{proof}
This follows again from Theorem \ref{thm:LogFittingCharacterization}.
\end{proof}

\section{Logarithmic derivatives tangent to a morphism}\label{sec:tang}

In this section, we develop two important ideas that have been introduced informally in \S\ref{subsec:deriv}
and Remark \ref{rem:idea}---derivations tangent to a morphism (see \S\ref{subsec:tang}) and pre-monomial form
(\S\ref{subsec:premonom}). 
We introduce two notions of log differential order 
relative to a monomial morphism---in \S\ref{subsec:ideal} for
an ideal, and in \S\ref{subsec:premonom} for a morphism that is ``partially monomial'' 
(already in monomial form except for its last component, in local coordinates).
To each notion, we associate a certain \emph{Weierstrass-Tschirnhausen
normal form}; these normal forms will play important technical parts in the inductive proofs of our main theorems
(Lemma \ref{lem:NormalFormPartialMonomial}, for example, in transforming a partially monomial
morphism to pre-monomial form, for the purpose of monomialization of a morphism, one component
at a time).

\medskip
Let $(M,D)$ denote manifold with SNC divisor $D$, of
quasianalytic class $\cQ$. We use $\cQ$ also to denote the structure sheaf of $M$. Let $\cD_D$ or $\Der_M(-\log D)$
denote the sheaf of $\cQ$-modules of logarithmic derivations (i.e., derivations which are tangent to $D$). Recall that,
in coordinates $(\bu,\bw) = (u_1,\ldots,u_r, w_1,\ldots,w_t)$
compatible with $D$ at a point $a\in M$ (where $D=\{u_1 \cdots u_r =0\}$), the local sections of $\cD_D$
are generated by
 \begin{equation*}
 u_1 \frac{\partial}{\partial u_1}, \ldots, u_r \frac{\partial}{\partial u_r} ,\, \frac{\partial}{\partial w_{1}} ,\ldots, \frac{\partial}{\partial w_{t}}.
 \end{equation*}

\subsection{Subsheaves of logarithmic derivatives}\label{subsec:subsheaf} Let $\De$ denote a sheaf of 
$\cQ$-submodules of $\cD_D$. Throughout this section, we assume that $\De$ is privileged (see Definition \ref{def:priv})
although, for the theorems in this article, it is enough to consider $\De$ of finite type. Given a privileged ideal (sheaf)
$\cI \subset \cQ$, let $\De(\cI)$ denote the privileged ideal whose stalk at each point $a\in M$ is
$$
\De(\cI)_a := \{X(f): f\in \cI_a,\, X \in \De_a \}.
$$
If $f\in \cQ(M)$, we also write $\De(f) := \De(\cI)$, where $\cI$ is the principal ideal generated by $f$.

\begin{definitions}\label{def:ChainAndClosure}\emph{Closure by repeated derivatives, and log differential order.}
Given a privileged ideal $\mathcal{I}$ and a privileged submodule $\De$ of $\cD_D$, we define a chain
of privileged ideals,
\begin{equation}\label{eq:chain}
\mathcal{I} = \mathcal{I}_0^{\De} \subset \mathcal{I}_1^{\De} \subset \mathcal{I}_2^{\De} \subset \cdots \subset \mathcal{I}_k^{\De} \subset \cdots,
\end{equation}
where $\mathcal{I}_{k+1}^{\De} = \mathcal{I}_{k}^{\De} + \De(\mathcal{I}_{k}^{\De})$,  $k=0,1,\ldots$, and we define
the \emph{closure}  of $\mathcal{I}$ by $\De$ as the (privileged) ideal
\[
\mathcal{I}_{\infty}^{\De} := \sum_{k=0}^{\infty} \mathcal{I}_k^{\De}.
\] 

Given $a\in M$, we define the \emph{log differential order} $\mu_a(\mathcal{I},\De)$ of $\cI$ \emph{relative to} $\De$
as the smallest $\mu \in \mathbb{N} \cup \{\infty\}$ such that
$\mathcal{I}_{\mu}^{\De} \cdot \mathcal{Q}_{a} = \mathcal{I}_{\infty}^{\De} \cdot \mathcal{Q}_{a}$.
(By convention, $\mu_a(\mathcal{I},\De):=\infty$ if $\cI = 0$.)
\end{definitions}

If the local rings $\mathcal{Q}_a$ are Noetherian (for example, if $\mathcal{Q}$ is the algebraic or analytic class), then 
$\mu_a(\mathcal{I},\De)$ is finite.

\begin{example}\label{ex:ToroidalHull}\emph{Toroidal hull.}
The \emph{toroidal hull} of an ideal $\cI$ (of finite type) is the closure of $\cI$ by $\cD_D$ (see \cite[Def.-Thm.\,17]{Ko},
\cite[\S\,2.2]{ATW}).
\end{example}

\begin{lemma}[Properties of closure by $\De$]\label{lem:BasicPropClosure}\hfill
Let $\mathcal{I}$, $\mathcal{J}$ denote privileged ideals, and $\De$, $\De_1$, $\De_2$ privileged submodules 
of $\cD_D$.
\begin{enumerate}
\item If $\mathcal{I} \subset \mathcal{J}$, then $\mathcal{I}_{\infty}^{\De} \subset \mathcal{J}_{\infty}^{\De}$.
\item If $\De_1 \subset \De_2$, then $\mathcal{I}_{\infty}^{\De_1} \subset \mathcal{I}_{\infty}^{\De_2}$.
\item Suppose that $\mathcal{I}_{\infty}^{\De}$ is principal and monomial (with respect to $D$) at $a\in M$. 
Then $\mu_{a}(\mathcal{I},\De)< \infty$;
moreover, if $\mu_{a}(\mathcal{I},\De)> 0$, then there exists a regular vector field $X \in \De_{a}$.
\end{enumerate}
\end{lemma}

\begin{remark}\label{rem:regvf} We recall that a vector field $X\in \cD_D$ is called \emph{regular} at $a$ if
$X$ does not vanish at $a$; i.e., there exists $f\in\cQ_a$ such that $X(f)(a)\neq 0$. A vector field $X$ is
\emph{singular} at $a$ if it is not regular at $a$.
\end{remark}

\begin{proof}
(1) and (2) are clear. To prove (3), suppose that $\mathcal{I}_{\infty}^{\De}$ is principal and monomial
at $a$, and consider a
coordinate system $(\pmb{u},\pmb{w})$ at $a$ compatible with $D$, as above. Then $\mathcal{I}^{\De}_{\infty,a}=(\pmb{u}^{\pmb{\gamma}})$, for some $\ga\in\IN^r$. Moreover, for all $k\in \mathbb{N}$, every $f\in \mathcal{I}_{k,a}^{\De}$
is divisible by $\bu^{\bga}$, since $\cI_k^\De \subset \cI_\infty^\De$. Therefore, the ideal $\mathcal{I}_{k,a}^{\De}$ has order
$\geq |\pmb{\gamma}|$, and $=|\pmb{\gamma}|$ only if $\mathcal{I}_{k,a}^{\De}= (\pmb{u}^{\pmb{\gamma}})$. It follows
that $ \mu_{a}(\mathcal{I},\De) < \infty$. Moreover, if every $X\in\De_{a}$ is singular, then the ideals $\mathcal{I}_{k,a}^{\De}$ 
all have the same order; this implies that $\mathcal{I}_a$ has order $|\pmb{\gamma}|$, so that
$\mu_{a}(\mathcal{I},\De)=0$. 
\end{proof}

The following two lemmas will be useful at several points of the proof of our main theorems (see,
for example, Lemma \ref{rk:BasisFormalIdeal} and \S\S\,\ref{subsec:I.A},\,\ref{subsec:I.B}). The notation $O(f)$ means a function (germ) that is divisible by $f$.

\begin{lemma}[Computation of the chain of ideals]\label{lem:ComputingInvariant}
Let $\mathcal{I}$ denote a privileged ideal sheaf, and $\De$ a privileged submodule of $\cD_D$. Let $a \in M$.
Let $F \in \mathcal{I}_{a}$, and let $X\in \De_{a}$ denote a singular vector field. Suppose that $F$ has the form
\[
F = \sum_{i=0}^d f_i + O(f),
\]
where $f$ and each $f_i\in\cQ_a$, $X(f)$ is divisible by $f$, and $X(f_i)=U_i f_i$, $i=0,\ldots,d$, where the $U_i\in\cQ_a$ are units
and $U_i(a)\neq U_j(a)$, $i\neq j$. Then
\[
f_i + O(f) \in \mathcal{I}_d^{\De} \cdot \mathcal{Q}_{a} \quad i=0,\ldots, d.
\]
\end{lemma}
\begin{proof}
Since $X$ is a singular derivation, it follows by repeated application of $X$ that, for every $i=0,\ldots,d$ and $k=0,\ldots,d$,
$X^k(f_i) = U_{ik}f_i$, where $U_{ik}\in\cQ_a$ and $U_{ik}(a)=U_i(a)^k$.
Therefore, for every $k=0,\ldots,d$,
\[
X^{k}(F) = \sum_{i=0}^d U_{ik} f_i + O(f);
\]
i.e., we have a system of equations
\[
\left(
\begin{matrix}
F\\
X(F)\\
\vdots\\
X^{d}(F)
\end{matrix}
\right)=
\left(
\begin{matrix}
1 & 1 & \cdots & 1\\
U_{01} & U_{11} & \cdots & U_{d1}\\
\vdots & \vdots & \ddots & \vdots\\
U_{0d} & U_{1d} & \cdots & U_{dd}\\
\end{matrix}
\right)
\left(
\begin{matrix}
f_0\\
f_1\\
\vdots\\
f_d
\end{matrix}
\right)  + O(f),
\]
where the matrix is invertible because it becomes a Vandermonde matrix when evaluated at $a$.
Since $X^k(F) \in \mathcal{I}_d^{\De} \cdot \mathcal{Q}_{a}$, $k=0,\ldots, d$, the result follows.
\end{proof}

\begin{remark}
\label{rk:ComputingInvariant}
Lemma \ref{lem:ComputingInvariant} will be used in the sequel usually in situations where $f=0$; in this case, it allows us to conclude that $(f_0,\ldots,f_d) \subset \mathcal{I}_d^{\De}\cdot \mathcal{Q}_{a}$. We will use this information to compute the log differential order $\mu_{a}(\mathcal{I},\De)$ of the ideal $\mathcal{I}_{a}$. For example, if $\mathcal{I}_a = (F)$ and $\Delta_a = (X)$, then we 
can conclude that $(f_0,\ldots,f_d) = \mathcal{I}_d^{\De}\cdot \mathcal{Q}_{a}$ and $\mu_{a}(\mathcal{I},\De) \leq d$ (see, for
example, Remark \ref{rem:Nak} below).

There are two places in the paper where we apply Lemma \ref{lem:ComputingInvariant} with $f \neq 0$ (see the proofs of the inductive assertions I.A.a$(m,n)$ and I.A.c$(m,n)$ in Section 7).
In both situations, use of the lemma is preceded by an argument to show that $f \in \mathcal{I}_d^{\De}\cdot \mathcal{Q}_{a}$. As a consequence, it allows us to conclude that $(f,f_0,\ldots,f_d) \subset \mathcal{I}_d^{\De}\cdot \mathcal{Q}_{a}$ and, once again, therefore, to control the log differential order.
\end{remark}

\begin{lemma}[Pullback of the ideal chain by blowings-up and power substitutions]
\label{lem:PreservingMu}
Let $\mathcal{I}\subset \cQ$ denote a privileged ideal and $\De$ be a privileged submodule of $\cD_D$. Let $\sigma : (\widetilde{M},\widetilde{D}) \to (M,D)$ denote the composite of a finite sequence of local blowings-up
and local power substitutions compatible with $D$. Suppose that 
$\widetilde{\De} = \sigma^{\ast}(\De)$ is a well-defined $\cQ$-submodule of $\cD_{\tD}$; i.e., there are no poles. Set
$\widetilde{\mathcal{I}}=\sigma^{\ast}(\mathcal{I})$. Then
\begin{enumerate}
\item $\sigma^{\ast}(\mathcal{I}_{k}^{\De}) = \widetilde{\mathcal{I}}_{k}^{\widetilde{\De}}$, for all $k\in \mathbb{N} \cup \{\infty\}$;
\item for every $\ta\in \widetilde{M}$,
$\mu_{\widetilde{a}}(\widetilde{\mathcal{I}},\widetilde{\De}) \leq \mu_{\s(\ta)}(\mathcal{I},\De)$.
\end{enumerate}
\end{lemma}

\begin{proof}
Clearly, $\widetilde{\De}$ is privileged.
Note that, if
$\sigma^{\ast}\left( \De(\mathcal{I})\right) = \widetilde{\De}(\widetilde{\mathcal{I}})$, then
$\sigma^{\ast}\left( \De^k(\mathcal{I}) \right) = \widetilde{\De}^k(\widetilde{\mathcal{I}})$, for all $k\in \mathbb{N}$.
The result follows from Definitions \ref{def:ChainAndClosure}.
\end{proof}

\subsection{Log derivatives tangent to a morphism}\label{subsec:tang}
Let $\Phi:(M,D) \to (N,E)$ denote a morphism of $\cQ$-manifolds with SNC divisors.

\begin{definition}\label{def:tang}
The sheaf of \emph{log derivatives tangent to} $\Phi$ (i.e., tangent to the fibres of $\Phi$) is the sheaf of
$\cQ$-submodules $\De^{\Phi}\subset \cD_D$ whose stalk at each $a\in M$ is
\[
\De^{\Phi}_{a} :=  \{X \in \cD_{D,a}:  X(f \circ \Phi)= 0 ,\, f\in \mathcal{Q}_{\Phi(a)}  \}.
\]
\end{definition}

\begin{remark}
The subsheaf $\De^{\Phi}$ is not necessary privileged for a quasianalytic class $\mathcal{Q}$, in general,
but it is of finite type wherever needed in this article; for example, if $\Phi$ is a monomial morphism (Lemma \ref{lem:BasisPhiDer} below).
\end{remark}

\begin{lemma}[Log derivatives tangent to a monomial morphism]\label{lem:BasisPhiDer}
Suppose that $\Phi:(M,D) \to (N,E)$ is a monomial morphism. Then $\De^{\Phi}$ is of finite type. If $a\in M$, and 
$\Phi$ is written in the form \eqref{eq:mon1} in ($\Phi$-monomial) coordinates 
$(\pmb{u},\pmb{v},\pmb{w})=(u_1,\ldots,u_r,v_{1},\ldots,v_{s},w_{1}, \ldots, w_t)$ at $a$ (cf.\ Definition \ref{def:monom2}),
then $\De^\Phi_a$ is generated by
\begin{equation}\label{eq:logder1}
Y^j := \sum_{i=1}^r \ga_{ji} u_i\frac{\p}{\p u_i}, \ \ j=1,\ldots, r-p,\quad \text{and}\quad Z^l := \frac{\p}{\p w_l}, \ \ l=1,\ldots, t,
\end{equation}
where $\bga_j = (\ga_{j1},\ldots,\ga_{jr})\in \IZ^r$, $j=1,\ldots,r-p$, form a basis of the
orthogonal complement of the $\IQ$-linear subspace spanned by $\bal_1,\ldots,\bal_p$
with respect to the standard scalar product $\langle \bga,\bal\rangle$.
\end{lemma}

\begin{remark}\label{rem:basis}
The set of vector fields $\{X^1,\ldots, X^{r-p+t}\} := \{Y^1,\ldots,Y^{r-p},Z^1,\ldots,Z^t\}$ will be called a \emph{monomial basis}
of $\De^{\Phi}_{a}$. 
\end{remark}

\begin{examples}
\begin{enumerate}
\item Let $\Phi: (\mathbb{C}^3,E) \to (\mathbb{C},D)$ be the monomial morphism given by
$z_1=v_1,\, E = \emptyset,\, D=\emptyset$. Then
$\Delta^{\Phi} = \{\p/\partial {w_1},\,\p/\partial {w_2}\}$, 
where $(v_1,w_1,w_2)$ are coordinates at $0$.
\smallskip
\item If $\Phi: (\mathbb{C}^3,E) \to (\mathbb{C},D)$ is given by
$x_1=u_1^{a_1}u_2^{a_2},\, E = \{u_1 u_2 =0\},\, D=\{x_1=0\}$, then
$\Delta^{\Phi} = \left\{\p/\partial {w},\, a_2u_1\p/\partial{u_1} - a_1u_2 \p/\partial{u_2}\right\}$,
where $(u_1,u_2,w)$ are $\Phi$-compatible coordinates at $0$.
\smallskip
\item If $\Phi: (\mathbb{C}^3,E) \to (\mathbb{C},D)$ is given by
$x_1=u_1^{a_1}u_2^{a_2}u_3^{a_3},\, E = \{u_1 u_2 u_3 =0\},\, D=\{x_1=0\}$, then
$$
\Delta^{\Phi} = \left\{a_2u_1\frac{\p}{\partial {u_1}} - a_1u_2 \frac{\p}{\partial{u_2}} , \
a_3u_1\frac{\p}{\partial {u_1}} - a_1u_3 \frac{\p}{\partial {u_3}}\right\}.
$$
\item If $\Phi: (\mathbb{C}^3,E) \to (\mathbb{C}^2,D)$ is given by
$x_1=u_1^{a_1}u_2^{a_2},\,z_1=v,\,
E = \{u_1 u_2 =0\},\, D=\{x_1=0\}$, then
$\Delta^{\Phi} = \left\{a_2u_1\p/\partial {u_1} - a_1u_2 \p/\partial{u_2}\right\}$.
\smallskip
\item If $\Phi: (\mathbb{C}^3,E) \to (\mathbb{C}^2,D)$ is given by
$x_1=u_1^{a_1}u_2^{a_2}u_3^{a_3},\,
x_2=u_1^{b_1}u_2^{b_2}u_3^{b_3},\,
E = \{u_1 u_2 u_3 =0\},\, D=\{x_1 x_2=0\}$, then
$$
\Delta^{\Phi} = \left\{c_1u_1\frac{\p}{\partial {u_1}} +c_2u_2\frac{\p}{\partial {u_2}} +c_3u_3\frac{\p}{\partial {u_3}}\right\},
$$
where $(c_1,c_2,c_3) = (a_1,a_2,a_3) \wedge (b_1,b_2,b_3)$.
\end{enumerate}
\end{examples}

\begin{proof}[Proof of Lemma \ref{lem:BasisPhiDer}]
Clearly, $\{Y^j, Z^l\} \subset \De^\Phi_a$. It is also clear we can assume that $s'=s$; i.e., that $\Phi$ is dominant.
We then argue by induction on $\dim N = n = p+s$. If $\dim N = 0$, then $\De^\Phi = \cD_D$, $r-p=r$, $(\bu,\bv,\bw)
= (\bu,\bw)$, and $Y^i = u_i \p/\p u_i$, $i=1,\ldots,r$, so the result is clear in this case.

By induction, assume the result holds when $\dim N < n$. We consider three cases, using the notation of
\eqref{eq:mon1}.
\begin{enumerate}
\item If $s>q$, we write $\Phi = (\Psi, \varphi)$, where $\vp = z_s = v_s$. By induction, $\De^{\Psi}_{a}$ is generated by
$X^1,\ldots, X^{r-p+t}$ and $X := \partial/\partial v_s$. Since $X^i(v_s) =0$, $i=1,\ldots,r-p+t$, and
$\De^{\Phi} \subset \De^{\Psi}$, it is clear that $\De^{\Phi}_a$ is generated by $X^1,\ldots, X^{r-p+t}$.
\smallskip
\item If $s=q\neq 0$, we write $\Phi = (\Phi', \varphi)$, where $\vp = y_q = \pmb{u}^{\bbe_q}(\xi_q + v_q)$. 
By induction, $\De^{\Phi'}_{a}$ is generated by
$X^1,\ldots, X^{r-p+t}$ and $X := \partial/\partial v_q$. Since $X^i(v_q) =0$, $i=1,\ldots,r-p+t$, and
$\De^{\Phi} \subset \De^{\Phi'}$, we again conclude that $\De^{\Phi}_a$ is generated by $X^1,\ldots, X^{r-p+t}$.
\smallskip
\item Finally, if $s=q=0$ (i.e., $n=p$), we write $\Phi = (\Phi', \varphi)$, where $\vp = x_p = \bu^{\bal_p}$. 
By induction, $\De^{\Phi'}_{a}$ has a monomial basis $\{X^1,\ldots,X^{r-p+t+1}\}$ analogous to \eqref{eq:logder1}.
Moreover, for all $i=1,\ldots,r-p+t+1$, $X^i(\pmb{u}^{\bal_p})= c_i \pmb{u}^{\bal_p}$, where $c_i\in \IZ$.
Since $\bal_1,\ldots,\bal_p$ are linearly independent over $\IQ$, $c_i \neq 0$ for some $i$, say $i=1$. Consider the
basis $\{Y^1,\ldots,Y^{r-p+t+1}\}$ of $\De^{\Phi'}_{a}$, where $Y^1=X^1$ and $Y^i = c_1X^i - c_iX^1$,
$i=2,\ldots,r-p+t+1$. Since $Y^i(\pmb{u}^{\pmb{\alpha}_{p}}) = 0$, $i=2,\ldots,r-p+t+1$, and $\De^{\Phi} \subset \De^{\Phi'}$, 
we conclude that $\{Y^2,\ldots, Y^{r-p+t+1}\}$ is a basis of $\De^{\Phi}$, and the result follows.
\end{enumerate}
\vspace{-\baselineskip}
\end{proof}

\begin{remark}\label{rk:FreeCoordinates}\emph{$\Phi$-free coordinate.}
It follows from Lemma \ref{lem:BasisPhiDer} that, if $\Phi:(M,D) \to (N,E)$ is a morphism that is monomial at a point $a\in M$, 
then there exists a regular vector field $X \in \De^{\Phi}_{a}$ if and only if there is coordinate system 
$(\pmb{u},\pmb{v},\pmb{w})=(u_1,\ldots,u_r,v_1,\ldots,v_s,\allowbreak w_1,\ldots,w_t)$ at $a$, where $t\geq 1$ and
$\Phi(\pmb{u},\pmb{v},\pmb{w}) =\Phi(\pmb{u},\pmb{v},0)$. We call $(w_1,\ldots,w_t)$ $\Phi$-\emph{free coordinates}.
\end{remark}

\begin{lemma}[Formal eigenvectors; cf.\,{\cite[\S7]{C1Fourier}}]\label{lem:FormalEigenVectors}
Let $\Phi:(M,D) \to (N,E)$ be a monomial morphism, and $a\in M$. In the notation of Lemma \ref{lem:BasisPhiDer}  and Remark \ref{rem:basis}, if $f\in\cQ_a$, then the formal
Taylor expansion $\hf_a(\bu,\bv,\bw)=T_af(\bu,\bv,\bw)$ of $f$ with respect to the coordinates $(\bu,\bv,\bw)$ 
(see \S\ref{subsec:alg} for the algebraic case) can be written
\begin{equation}\label{eq:eigen}
\widehat{f}_a(\pmb{u},\pmb{v},\pmb{w}) = \sum_{\pmb{\delta}\in \mathbb{N}^{t}} \pmb{w}^{\pmb{\delta}} \sum_{\pmb{\lambda} \in \mathbb{Z}^{r-p}} \widehat{f}_{\pmb{\delta}\pmb{\lambda}}(\pmb{u},\pmb{v}),
\end{equation}
where either $\widehat{f}_{\pmb{\delta},\pmb{\lambda}} = 0$, or the $\widehat{f}_{\pmb{\delta},\pmb{\lambda}}$ are eigenvectors of the monomial basis $\{X^1,\ldots, X^{r-p+t}\}$ of $\De^\Phi_a$; i.e.,
\[
X^i(\widehat{f}_{\pmb{\delta}\pmb{\lambda}}(\pmb{u},\pmb{w})) = \lambda_i \widehat{f}_{\pmb{\delta}\pmb{\lambda}}(\pmb{u},\pmb{v}), \quad i=1,\ldots, r-p+t
\]
(and $\la_i=0$ if $i>r-p$).
\end{lemma}

\begin{proof}
We can write
\[
\widehat{f}_a(\pmb{u},\pmb{v},\pmb{w}) = \sum_{\pmb{\delta}\in \mathbb{N}^{t}} \pmb{w}^{\pmb{\delta}} \sum_{\pmb{\varepsilon} \in \mathbb{N}^{r}}  \pmb{u}^{\pmb{\varepsilon}} \widehat{f}_{\pmb{\delta}\pmb{\varepsilon}}(\pmb{v}).
\]
For every $\pmb{\varepsilon} \in \mathbb{N}^r$ and $i=1,\ldots,r-p+t$, we have $X^{i}(\pmb{u}^{\pmb{\varepsilon}}) = \lambda_{\pmb{\varepsilon},i} \pmb{u}^{\pmb{\varepsilon}}$, where
$\lambda_{\pmb{\varepsilon},i} \in \mathbb{Z}$. For each $\pmb{\lambda} \in \mathbb{Z}^{r-p+t}$, let
$I_{\pmb{\lambda}} := \{\pmb{\varepsilon} \in \mathbb{N}^r: \pmb{\lambda}_{\pmb{\varepsilon}}=\la\}$; then we can take
\[
\widehat{f}_{\pmb{\delta}\pmb{\lambda}}(\pmb{u},\pmb{v}) = \sum_{\pmb{\varepsilon}\in I_{\pmb{\lambda}}} \pmb{u}^{\pmb{\varepsilon}}
 \widehat{f}_{\pmb{\delta}\pmb{\varepsilon}}(\pmb{v}).
\]

\vspace{-2\baselineskip}

\end{proof}
\vspace{\baselineskip}

\begin{lemma}[Formal generators of a $\De^\Phi$-closed ideal]\label{rk:BasisFormalIdeal}
Let $\Phi:(M,D) \to (N,E)$ denote a monomial morphism and let $\mathcal{I}$ be a privileged ideal sheaf which is closed by 
$\De^{\Phi}$; i.e., $\De^{\Phi}(\mathcal{I})\subset \mathcal{I}$. Let $a \in M$. In the notation of Lemma \ref{lem:BasisPhiDer}  and Remark \ref{rem:basis}, the ideal of formal power series generated by
Taylor expansions of elements of $\mathcal{I}_{a}$ (which can be identified with the formal completion $\widehat{\mathcal{I}}_{a}$
of $\mathcal{I}_{a}$) admits a finite system of generators $\{\widehat{f}_{\pmb{\lambda}}(\pmb{u},\pmb{v})\}$, where each $\widehat{f}_{\pmb{\lambda}}(\pmb{u},\pmb{v})$ is an eigenvector of  $\{X^1,\ldots, X^{r-p+t}\}$. In particular, $\widehat{\mathcal{I}}_{a}$ admits a system of generators that are independent of the $\Phi$-free variables $\pmb{w}$.
\end{lemma}

\begin{proof}
(See \cite[Lemma 3.7]{BdS} for a different proof.)
It is clear that $\widehat{\mathcal{I}}_{a}$ is also closed by $\De^{\Phi}_a$.
Let $\{{f}^\io\}_{\io \in I}$ denote a system of generators of $\mathcal{I}_{a}$. For each $\io\in I$,
consider the formal expansion \eqref{eq:eigen},
\[
\widehat{f}^{\iota}_a(\pmb{u},\pmb{v},\pmb{w}) = \sum_{\pmb{\delta}\in \mathbb{N}^{t}} \pmb{w}^{\pmb{\delta}} \sum_{\pmb{\lambda} \in \mathbb{Z}^{r-p}} \widehat{f}^{\iota}_{\pmb{\delta}\pmb{\lambda}}(\pmb{u},\pmb{v}).
\]
We will show that $\widehat{\cI}_{a}$ is the ideal $\widehat{J} = (\widehat{f}^{\iota}_{\pmb{\delta}\pmb{\lambda}}(\pmb{u},\pmb{v}))$  generated by all
$\widehat{f}^{\iota}_{\pmb{\delta}\pmb{\lambda}}$ (and, therefore, by finitely many of the latter, since the ring of formal
power series in Noetherian).

For each fixed $\iota \in I$, there are finitely many pairs 
$(\bde,\bla) = (\pmb{\delta}_k,\pmb{\lambda}_k)$, $k=1,\ldots,d=d_\io$, such that
\[
\widehat{f}^{\iota}_a(\pmb{u},\pmb{v},\pmb{w}) = \sum_{k=1}^d \pmb{w}^{\pmb{\delta}_k} \widehat{f}^{\iota}_{\pmb{\delta_k}\pmb{\lambda}_k}(\pmb{u},\pmb{v}) \widehat{U}_k(\pmb{u},\pmb{v},\pmb{w}),
\]
where every $\widehat{U}_k$ is a formal power series with $\widehat{U}_k(a)=1$. Therefore, 
$\widehat{\mathcal{I}}_{a} \subset \widehat{J}$.

Conversely (using the notation of Lemma \ref{lem:BasisPhiDer}), consider the vector fields
$W^l = w_l Z^l$, $l=1,\ldots,t$. Fix $\io\in I$. It is easy to find a $\mathbb{Z}$-linear combination 
$X = \sum_{j=1}^{r-p} a_j Y^j + \sum_{l=1}^t b_l W^l$ such that:
\[
X(\pmb{w}^{\pmb{\delta}_k} \widehat{f}^{\iota}_{\pmb{\delta_k}\pmb{\lambda}_k}(\pmb{u},\pmb{v})) =\rho_k \pmb{w}^{\pmb{\delta}_k} \widehat{f}^{\iota}_{\pmb{\delta_k}\pmb{\lambda}_k}(\pmb{u},\pmb{v}), \quad k=1,\ldots, d,
\]
where $\rho_k\neq \rho_{k'}$ whenever $k \neq k'$. By Lemma \ref{lem:ComputingInvariant} (applied with $F=\widehat{f}^{\iota}_a$, $f_k = \pmb{w}^{\pmb{\delta}_k} \widehat{f}^{\iota}_{\pmb{\delta_k}\pmb{\lambda}_k} \widehat{U}_k$ and $f= 0$), since 
$\De^{\Phi}_a(\widehat{\mathcal{I}}_{a})\subset \widehat{\mathcal{I}}_{a}$,
\[
\pmb{w}^{\pmb{\delta}_k} \widehat{f}^{\iota}_{\pmb{\delta_k}\pmb{\lambda}_k}(\pmb{u},\pmb{v}) \in 
(\widehat{\mathcal{I}}_a)^{\Phi}_{d} = \widehat{\mathcal{I}}_a, \quad k=1,\ldots, d.
\]
By applying the derivations $Z^l$, we get
$\widehat{f}^{\iota}_{\pmb{\delta_k}\pmb{\lambda}_k}(\pmb{u},\pmb{v}) \in \widehat{\mathcal{I}}_a$, $k=1,\ldots, d$.
Therefore, $\widehat{J}\subset \widehat{\mathcal{I}}_{a}$.
\end{proof}

\begin{remark}\label{rem:BasisFormalIdeal}
In the analytic or algebraic cases, the generators $\widehat{f}^{\iota}_{\pmb{\delta}\pmb{\lambda}}$ in the
preceding proof are convergent or algebraic power series, and it follows from Noetherianity that we get
a convergent or algebraic version of Lemma \ref{rk:BasisFormalIdeal}; in particular, the ideal $\cI_a$
of convergent or algebraic power series has a finite system of generators
$\{f_{\pmb{\lambda}}(\pmb{u},\pmb{v})\}$, where each $f_{\pmb{\lambda}}(\pmb{u},\pmb{v})$ is an eigenvector of  $\{X^1,\ldots, X^{r-p+t}\}$.
\end{remark} 

\begin{remarks}\label{rk:ToroidalHullResolution} 
(1) If $\Phi$ is a constant morphism, then $\De^\Phi$ is the full sheaf of log derivations $\cD_D = \Der_M(-\log D)$;  in this case,
Theorem \ref{thm:ideal} on resolution of singularities of a $\De^\Phi$-closed ideal sheaf has a simple proof. 
A $\cD_D$-closed (privileged) ideal sheaf $\mathcal{I}$ is formally generated at every point by monomials in the exceptional divisor $D$
(see Lemmas \ref{lem:FormalEigenVectors}, \ref{rk:BasisFormalIdeal}). It follows that $\cI$ is generated by monomials in $D$
(by the division axiom \ref{def:quasian}(1)). Therefore, $\cI$ admits a combinatorial resolution of singularities $\sigma : (\widetilde{M},\widetilde{D}) \to (M,D)$, 
and it follows that the pull-back of $\De^\Phi = \cD_D$ coincides with $\cD_{\tD}$.

\medskip\noindent
(2) \emph{Principalization of the toroidal hull}. In particular, if $\Phi$ is constant, then the  toroidal hull of $\mathcal{I}$ is $\De^\Phi$-closed
(see Example \ref{ex:ToroidalHull}), and can therefore be principalized by a sequence of combinatorial blowings-up 
$\sigma:(\widetilde{M},\widetilde{D}) \to (M,D)$. By Lemma \ref{lem:PreservingMu}(1), the closure of the pull-back $\widetilde{\mathcal{I}} = \sigma^{\ast}(\mathcal{I})$ by $\sigma^{\ast}(\De^{\Phi}) = \cD_{\tD}$ is principal monomial. By Lemma \ref{lem:BasicPropClosure}(3), 
therefore, for every point $\ta \in \widetilde{M}$, either $\widetilde{\mathcal{I}}$ is principal monomial at $\ta$, or the exceptional divisor 
$\widetilde{D}$ has at most $m-1$ components at $\ta$ (cf.\ \cite[Lemma 5.1]{BBAdvances}).
\end{remarks}

Control of the transform of $\De^{\Phi}$ as in Theorem \ref{thm:ideal} (and the preceding remarks) will play an important part in the proofs
of our monomialization theorems. Lemma \ref{lem:PreservingLogAdapted} following describes a general situation where we can provide such control.

\begin{lemma}\label{lem:PreservingLogAdapted}
Let $\Phi : (M,D) \to (N,E)$ denote a dominant monomial morphism, and let $\psi:M \to \mathbb{K}$ be a function. 
(In the algebraic case, $\IK$ should be replaced by $\IA^1_{\IK}$.)
Assume that $\Psi := (\Phi,\psi): M,D) \to (N \times \mathbb{K},F)$ is a morphism with respect to the divisor $F$ on
$N \times \mathbb{K}$ given either by $E \times \mathbb{K}$ or by $(E\times \mathbb{K}) \cup ( N \times \{0\} )$.
Suppose that, for all $a \in M$, there exists a set of generators $\{X^1,\ldots,X^{m-n}\}$ of $\De^{\Phi}_{a}$ such that:
$X^{i}(\psi) = 0$, $i=1,\ldots, m-n-1$.
Let $\sigma : (\widetilde{M},\widetilde{D}) \to (M,D)$ denote a sequence of local blowings-up and power substitutions,
and write $\widetilde{\Psi} = (\widetilde{\Phi},\widetilde{\psi}) := \Psi \circ \sigma = (\Phi \circ \sigma,\psi \circ \sigma)$. 
If $\sigma^{\ast}(\De^{\Phi}) = \De^{\widetilde{\Phi}}$, then $\sigma^{\ast}(\De^{\Psi}) = \De^{\widetilde{\Psi}}$.
\end{lemma}

\begin{proof}
Let $\ta \in \widetilde{M}$ and let $a = \sigma(\ta) \in M$. By hypothesis, there is a system of generators 
$\{X^1,\ldots,X^{m-n}\}$ of $\De^{\Phi}_{a}$ such that $X^{i}(\psi) = 0$ if $i\neq m-n$. If $X^{m-n}(\psi) = 0$, then
$\De^{\Phi}=\De^{\Psi}$ and the Lemma holds trivially. Suppose that $X^{m-n}(\psi) \neq 0$. Then
$\De^{\Psi} = \left\{ X^1, \ldots, X^{m-n-1} \right\}$. By hypothesis, all vector fields $Y^i=\sigma^{\ast}(X^i)$ are well-defined
and $\De^{\widetilde{\Phi}} = \{Y^1,\ldots, Y^{m-n} \}$. Thus $ \{Y^1,\ldots, Y^{m-n-1} \} \subset \De^{\widetilde{\Psi}}$ and $Y^{m-n}(\widetilde{\psi}) = \sigma^{\ast}\left(X^{m-n}(\psi)\right)\neq 0$. Since $\De^{\widetilde{\Psi}} \subset \De^{\widetilde{\Phi}}$, we conclude that $\De^{\widetilde{\Psi}} = \{Y^1, \ldots, Y^{m-n-1}\} = \sigma^{\ast}(\De^{\Psi})$.
\end{proof}

\begin{lemma}[Preserving derivatives tangent to a monomial morphism]\label{lem:BUPreservingMonomialDeriv}
Let $\Phi : (M,D) \to (N,E)$ denote a dominant monomial morphism, and let $\psi:M \to \mathbb{K}$ be a function. 
Assume that $\Psi := (\Phi,\psi): (M,D) \to (N \times \mathbb{K},F)$ is a monomial morphism with respect to the divisor $F$ on
$N \times \mathbb{K}$ given either by $E \times \mathbb{K}$ or by $(E\times \mathbb{K}) \cup ( N \times \{0\} )$.
Let $\sigma : (\widetilde{M},\widetilde{D}) \to (M,D)$ denote a sequence of local blowings-up and power substitutions,
and write $\widetilde{\Psi} = (\widetilde{\Phi},\widetilde{\psi}) := \Psi \circ \sigma = (\Phi \circ \sigma,\psi \circ \sigma)$. 
If $\sigma^{\ast}(\De^{\Phi}) = \De^{\widetilde{\Phi}}$, then $\sigma^{\ast}(\De^{\Psi}) = \De^{\widetilde{\Psi}}$.
\end{lemma}

\begin{proof}
Let $\{X^1,\ldots, X^{m-n}\}$ denote a monomial basis of $\De^\Phi$ at $a\in M$ (as given by 
Lemma \ref{lem:BasisPhiDer}), and let $(\pmb{u},\pmb{v},\pmb{w})$ denote a $\Psi$-monomial coordinate system 
for $M$ at $a$. We have to consider each of the following cases.
\begin{enumerate}
\item $\psi = 0$. In this case, $\De^{\Psi}=\De^{\Phi}$ and the result is trivial.
\medskip
\item $\psi = \pmb{u}^{\pmb{\alpha}}$ (at $a$). Then $X^i(\pmb{u}^{\pmb{\alpha}}) = C_i \pmb{u}^{\pmb{\alpha}}$, 
$C_i \in \IZ$,  $i=1,\ldots, m-n$, where $C_i \neq 0$ for some $i$; say $C_{m-n}\neq 0$. Consider the basis 
$\{Y^1,\ldots, Y^{m-n}\}$ of $\De^{\Phi}$ given by $Y^i := C_{m-n}X^i - C_i X^{m-n}$ for each $i\neq m-n$, and 
$Y^{m-n}:=X^{m-n}$. This basis satisfies the hypothesis of Lemma \ref{lem:PreservingLogAdapted}, so the result follows.
\medskip
\item  $\psi = v$ or $\psi = \pmb{u}^{\pmb{\beta}}(v+\xi)$, $\xi\neq 0$. In these cases, we can assume that 
$X^{m-n} = \partial / \partial v$, and the result follows from Lemma \ref{lem:PreservingLogAdapted}.
\end{enumerate}
\vspace{-\baselineskip}
\end{proof}

\begin{corollary}[Combinatorial blowings-up preserve derivatives tangent to a monomial morphism]\label{cor:DerCombinatorialBU}
Let $\Phi:(M,D) \to (N,E)$ denote a monomial morphism and let $\sigma: (\widetilde{M},\widetilde{D}) \to (M,D)$ be a combinatorial blowing-up. Set $\widetilde{\Phi}:=\Phi \circ \sigma$. Then  $\sigma^{\ast}(\De^{\Phi}) = \De^{\widetilde{\Phi}}$.
\end{corollary}

\begin{proof}
We argue by induction on $\dim N$. If $\dim N = 0$, then $\Phi$ is constant
and $\De^{\Phi} = \cD_D = \Der_{M}(-\log D)$, so the result follows from the standard transformation formulas for
log derivatives (see \cite[Lemma 3.1]{BMfunct}). For the inductive step, we can use Lemma \ref{lem:BUPreservingMonomialDeriv}.
\end{proof}

\subsection{Log differential order of an ideal relative to a monomial morphism}\label{subsec:ideal}

\begin{definition}\label{def:InvariantPairs} 
Let $\Phi:(M,D) \to (N,E)$ denote a monomial morphism and let $\mathcal{I}$ denote a privileged ideal sheaf on $M$. 
We define the \emph{log differential order} $\nu_{a}(\mathcal{I},\Phi)$ of $\cI$ \emph{relative to} $\Phi$ at $a \in M$
as the log differential order of $\cI$ at $a$ relative to the sheaf of log derivatives $\De^\Phi$ tangent to $\Phi$
(see Definitions \ref{def:ChainAndClosure}); i.e.,
$$
\nu_{a}(\mathcal{I},\Phi):=\mu_{a}(\mathcal{I},\De^{\Phi}). 
$$
We also define the \emph{log differential closure} of $\cI$ \emph{relative to} $\Phi$ as the ideal sheaf
$\mathcal{I}_{\infty}^{\Phi} :=\mathcal{I}_{\infty}^{\De^{\Phi}}$.
\end{definition}

\begin{example}
Let $\Phi: (\mathbb{C}^3,E) \to (\mathbb{C}^2,D)$ be the monomial morphism defined by
\[
x_1=u_1u_2,\, x_2=u_2u_3,\quad
E = \{u_1 u_2 u_3 =0\},\quad D=\{x_1 x_2=0\}.
\]
Then $\Delta^{\Phi}$ is generated by the vector field
\[
X = u_1 \frac{\p}{\partial {u_1}} - u_2 \frac{\p}{\partial {u_2}} +u_3 \frac{\p}{\partial{u_3}}.
\]
Consider the principal complex-analytic ideal sheaf $\mathcal{I} = (f)$ generated by
\[
f(\pmb{u}) = u_1u_2 + u_2u_3 + u_2u_3^4 + u_1u_2^6.
\]
We compute the chain of ideals $\cI^\Phi_k = \cI^{\De^\Phi}_k$ associated to $\cI = \mathcal{I}_0^{\Phi}$.
Note first that
$$
X(f) =   3 u_2u_3^4 -5 u_1u_2^6,\quad X^2(f) = 9 u_2u_3^4 +25 u_1u_2^6, 
$$
so that 
$$
u_2u_3^4 = \frac{1}{24}\left( 5X(f)+ X^2(f)\right),\quad
u_1u_2^6 = \frac{1}{40}\left(- 3X(f)+ X^2(f)\right).
$$
It follows that
\begin{align*}
\mathcal{I}_0^{\Phi}&= (u_1u_2 + u_2u_3 + u_2u_3^4 + u_1u_2^6),\\
\mathcal{I}_1^{\Phi} &= (u_1u_2 + u_2u_3 + u_2u_3^4 + u_1u_2^6, \, 3 u_2u_3^4 -5 u_1u_2^6),\\
\mathcal{I}_2^{\Phi} &= (u_1u_2 + u_2u_3,\, u_2u_3^4,\, u_1u_2^6).
\end{align*}
It is not difficult to see that $\mathcal{I}_2^{\Phi}$ is invariant by $\Delta^{\Phi}$; therefore,
$\mathcal{I}^{\Phi}_{\infty} = \mathcal{I}_2^{\Phi}$. Note that $\mathcal{I}^{\Phi}_{\infty}$ is not generated by monomials, 
so cannot be principalized by combinatorial blowings-up.
\end{example}

\begin{lemma}[Weierstrass-Tschirnhausen normal form I]\label{lem:NormalFormIdeal}
Let $\Phi:(M,D) \to (N,E)$ denote a monomial morphism and let $\mathcal{I}$ be a privileged ideal sheaf on $M$. Let $a\in M$.
\begin{enumerate}
\item The ideal $\mathcal{I}_a$ is principal and monomial if and only if 
$\cI^\Phi_{\infty,a} := \mathcal{I}_{\infty}^{\Phi} \cdot \mathcal{Q}_{a}$ is principal and monomial, and $\nu_{a}(\mathcal{I},\Phi)=0$.
\medskip
\item If $\mathcal{I}_{\infty,a}^{\Phi}$ is principal and monomial, then $d:=\nu_{a}(\mathcal{I},\Phi) <\infty$. If, moreover,
$d>0$, then there exists a $\Phi$-monomial coordinate system $(\pmb{u},\pmb{v},\pmb{w})=(\pmb{u},\pmb{v},w_1,\widehat{\pmb{w}}_1)$ at $a$, such that $\mathcal{I}_{a}$ has a set of generators $\{F_\io\}_{\io\in I}$ of the form
\[
F_\io(\pmb{u},\pmb{v},\pmb{w}) = \pmb{u}^{\pmb{\gamma}}\left\{ \widetilde{F_\io}(\pmb{u},\pmb{v},\pmb{w}) w_1^d + \sum_{j=0}^{d-1} f_{\io j}(\pmb{u},\pmb{v},\widehat{\pmb{w}}_1)w^j_1  \right\},
\]
where $\mathcal{I}_{\infty,a} = (\pmb{u}^{\pmb{\gamma}})$ and there exists $\iota\in I$ such that $\widetilde{F_{\iota}}(0)\neq 0$, $f_{\iota,d-1} = 0$ and $f_{\io j}(0) =0$, $j=0,\ldots,d-2$.
\end{enumerate}
\end{lemma}

\begin{proof}
(1) holds because, on the one hand, if $\cI_a$ is principal and monomial, then $\cI^\Phi_{\infty,a} = \cI_a$, by
Lemma \ref{lem:BasisPhiDer}, and, on the other, $\nu_{a}(\mathcal{I},\Phi)=0$ if and only if 
$\cI^\Phi_{\infty,a} = \cI_a$, according to Definitions \ref{def:ChainAndClosure}.

(2) Suppose that $\mathcal{I}_{\infty,a}^{\Phi}$ is principal and monomial. Then $d:=\nu_{a}(\mathcal{I},\Phi) <\infty$,
 by Lemma \ref{lem:BasicPropClosure}(3). Suppose that $d>0$.  Again by Lemma \ref{lem:BasicPropClosure}(3), 
 there exists a regular vector field germ $X \in \De_{a}^{\Phi}$. Therefore there is a monomial basis
 $\{X^1,\ldots, X^{m-n}\}$ of $\De^\Phi_a$, as in Lemma \ref{lem:BasisPhiDer}, with respect to
 a $\Phi$-monomial coordinate system $(\pmb{u},\pmb{v},\pmb{w})$ with at least one $\Phi$-free $\pmb{w}$-variable
(Remark \ref{rk:FreeCoordinates}). By Lemma \ref{lem:FormalEigenVectors}, every $F \in \mathcal{I}_{a}$ 
has a formal Taylor expansion of the form
\[
\widehat{F} = \sum_{\pmb{\delta} \in \mathbb{N}^{t}} \pmb{w}^{\pmb{\delta}} \sum_{\pmb{\lambda} \in \mathbb{N}^{m-n-t}} \widehat{f}_{\pmb{\delta},\pmb{\lambda}}(\pmb{u},\pmb{v}),
\]
where every $\widehat{f}_{\pmb{\delta},\pmb{\lambda}}(\pmb{u},\pmb{v})$ is an eigenvector of $\{X^1,\ldots, X^{m-n}\}$. 

Since $\mathcal{I}_{\infty}^{\Phi}=(\pmb{u}^{\pmb{\gamma}})$, we see that $\pmb{u}^{\pmb{\gamma}}$ is an eigenvector.
It follows from Lemma \ref{rk:BasisFormalIdeal} that there exists $F \in \mathcal{I}$ such that
$F = \pmb{u}^{\pmb{\gamma}} \left( \widetilde{F}(\pmb{w})  + O(\pmb{u},\pmb{v}) \right)$,
where $\widetilde{F}(\pmb{w})$ has order $d$.
By a linear change of coordinates in $\pmb{w}$, we can assume that
\[
F = \pmb{u}^{\pmb{\gamma}} \left(\widetilde{F}(\pmb{u},\pmb{v},\pmb{w}) w_1^d + \sum_{j=0}^{d-1} f_j(\pmb{u},\pmb{v},\hat{\pmb{w}}_1)w^j_1 \right).
\]
Finally, we can use the implicit function theorem (Definition \ref{def:quasian}(2)) to transform the coordinates so that $f_{d-1}=0$
(essentially a Tschirnhausen transformation);
the new coordinate system is still $\Phi$-monomial because $w_1$ is $\Phi$-free. The result follows.
\end{proof}

\subsection{Pre-monomial morphism}\label{subsec:premonom}
The idea of pre-monomial form (Remark \ref{rem:idea}) will be made precise in Definition \ref{def:PreMonomial} below. 

\begin{lemma}[Algebraic dependence on a monomial morphism]\label{lem:AlgebraicDependentFunctions}
Let $\Phi:(M,D) \to (N,E)$ denote a morphism which is monomial at a point $a\in M$. 
Suppose that $\Phi$ is written in the form \eqref{eq:mon1} 
with respect to ($\Phi$-monomial) coordinates 
$(\pmb{u},\pmb{v},\pmb{w})$ at $a$ and $(\bx,\by,\bz)$ at $b=\Phi(a)$(cf.\ Definition \ref{def:monom2}). Let $G(\bu,\bv,\bw)
=\sum g_{\bga\bde\pmb{\varepsilon}} \bu^{\bga} \bv^{\bde} \bw^{\pmb{\varepsilon}}$
denote a formal power series at $a$. Then the following conditions are equivalent.
\begin{enumerate}
\item $dG \wedge d\Phi := dG \wedge dx_1 \wedge \ldots \wedge dx_p \wedge dy_1 \wedge \ldots \wedge dy_q \wedge dz_{q+1} \ldots \wedge dz_s = 0$.
\smallskip
\item For every $X \in \De^{\Phi}_{a}$, $X(G) = 0$.
\smallskip
\item $g_{\pmb{\gamma}\pmb{\delta}\pmb{\varepsilon}} =0$ whenever $\pmb{\varepsilon}\neq \pmb{0}$; moreover,
if $g_{\pmb{\gamma}\pmb{\delta}\pmb{0}} \neq 0$, then $\pmb{u}^{\pmb{\gamma}} = \pmb{x}^{\pmb{\rho}}$,
with $\pmb{\rho} \in \mathbb{Q}^{p}$.
\end{enumerate}
\end{lemma}
\begin{proof}
Direct computation.
\end{proof}

\begin{definition}\label{def:PreMonomial}\emph{Pre-monomial morphism.}
Let $\Psi: (M,D)\to(P,F)$ denote a morphism of $\cQ$-manifolds with SNC divisors ($\dim M =m$,
$\dim P = n+1$). We say that $\Psi$ is \emph{pre-monomial} at a point $a\in M$ if there are coordinate
systems $(\bu,\bv,\bw)$ and $(\bx,\by,\bz,t)$ at $a$ and $\Psi(a)$, respectively ($t$ is a single variable),
in which 
\begin{equation}\label{eq:ep}
D = \{u_1\cdots u_r = 0\} \quad \text{and}\quad
F = \{x_1\cdots x_p y_1\cdots y_q t^{\varepsilon} =0\}, \text{ with } \varepsilon =0 \text{ or } 1,
\end{equation}
(notation of \eqref{eq:coords1}) and
\begin{equation*}
\Psi = (\Phi, \psi) = (\Phi_1,\ldots,\Phi_n, \psi),
\end{equation*}
where
\begin{enumerate}
\item $\Phi$ is a dominant monomial morphism with respect to $D$ and $E:=\{x_1\cdots x_p \cdot y_1\cdots y_q =0\}$,
in the form \eqref{eq:mon1}, i.e., $(\bu,\bv,\bw)$ and $(\bx,\by,\bz)$ are $\Phi$-monomial
coordinate systems as in Definition \ref{def:monom2};
\smallskip
\item $\psi = g + \vp$, where $g=g(u,v)$ is algebraically dependent on (the components of) $\Phi$, i.e.,
$$
dg \wedge d\Phi := dg\wedge d\Phi_1 \wedge d\Phi_n =0,
$$
and $t=\vp(\bu,\bv,\bw)$ is a function of one of the following four normal forms (we use the notation of \eqref{eq:mon1}):
\begin{align*}
\qquad\quad t &= 0,  &&\\
t &= \bu^{\bal},&&\bal, \bal_1,\ldots,\bal_p\, \text{ $\IQ$-linearly independent},\\
t &= \bu^{\bga}(\eta + w_1),&& \bga,\bal_1,\ldots,\bal_p\, \text{ $\IQ$-linearly dependent and } \bga\neq 0,\\
t &= w_1.&&
\end{align*}
\end{enumerate}

We call $g$ a \emph{$\Phi$-dependent} or \emph{remainder term}. (The decomposition $\psi=g+\vp$ is not necessarily 
unique.) Coordinate systems as above will be called \emph{$\Psi$-pre-monomial}.
\end{definition}

Although the monomial part $\Phi$ of a pre-monomial morphism $\Psi = (\Phi, \psi)$ is assumed dominant
in Definition \ref{def:PreMonomial}, the morphism $\Psi$ itself need not be dominant.

\begin{example}\label{ex:PreMonomial} 
The morphism $\Psi:(\mathbb{C},D) \to (\mathbb{C}^2,F)$ given by $\Psi(u) = (u^2,u^3)$, where $D=(u=0)$ and $F=(x=0)$ 
(in particular, $\varepsilon =0$), is pre-monomial but not dominant. Indeed, we can write $\Psi = (\Phi,\psi)$, 
where $\Phi:(\mathbb{C},D) \to (\mathbb{C},E)$ is given by $\Phi(u) = u^2$ and $E=(x=0)$, and $\psi(u)=g(u) + \varphi(u)$, 
where $g(u)=u^3$ and $\varphi(u)\equiv 0$. Note that the morphism $\Psi^g:=(\Phi,\varphi) = (u^2,0)$ is monomial (but not dominant); 
cf.\ Remark \ref{rem:premonom1} below.
\end{example}

\begin{remarks}\label{rem:premonom1}
Recall that, if $\Phi: (M,D) \to (N,E)$ is a $\cQ$-morphism, then the pre-image $\Phi^{-1}(E)$ is SNC as a space and its support 
lies in that of $D$, by definition (see Section 1). We will use this remark implicitly in the following.

In the preceding formulas for $\vp(\bu,\bv,\bw)$, $w_1$ is a $\Phi$-free variable. In the final case, necessarily $\varepsilon=0$ in \eqref{eq:ep} and $\Psi$ is monomial at $a$ (in the form \eqref{eq:mon1} after a change
of variable $w_1' = g(\bu,\bv) + w_1$ in the source). 
In the third case, $\varepsilon=1$ and $\eta\neq 0$.
If $n=0$, then a pre-monomial morphism $\Psi=(\Phi,\psi)$ is monomial
since a remainder $g$ satisfies $dg=0$.

Condition (2) of the definition means that 
$\Psi^g := (\Phi,\vp)$ defines a monomial morphism $\Psi^g: (M,D^g)\to (P,F^g)$ (locally), where
\begin{description}
\item[$(P,F^g)$] is induced from $(P,F)$ by the codimension one blowing-up of $P$ with centre
$\{t=0\}$ in the second and third cases of (2) above, if $\varepsilon=0$ (and $F^g = F$ in all other cases);
\smallskip
\item[$(M,D^g)$] is induced from $(M,D)$ by the codimension one blowing-up of $M$ with centre
$\{w_1 = 0\}$ in the third case of (2), if $\eta=0$ (so also $\varepsilon=0$), and $D^g=D$ in all other cases.
\end{description}
The divisors $D^g$ and $F^g$ are not necessarily uniquely determined by $\Psi$.
Note that, to write $\Psi^g$ in the monomial form \eqref{eq:mon1}
at $a$, we may have to permute the source or target variables.
The codimension one blowings-up above will come in the final step of the proof of our main theorems, in
transforming a pre-monomial to a monomial morphism (\S\ref{subsec:II}). These blowings-up can be compared
with the final blowing-up in algorithms for principalization of an ideal, where the strict transform is replaced by a 
component of the divisor.

If $\Psi$ is a pre-monomial at $a$, then $\De^\Psi_a = \De^{\Psi^g}_a$.
\end{remarks}

\begin{remarks}\label{rem:premonom2}
It follows from Definition \ref{def:PreMonomial} that (at least in the analytic and real quasianalytic cases)
we can write $P = N\times \IK$, locally at $b=\Psi(a)$, where $N = \{t=0\}$, and $\Psi = (\Phi, \psi): M \to N\times \IK$,
where $\Phi: (M,D) \to (N,E)$ is a dominant morphism that is monomial at $a$ and $E = \{x_1\cdots x_p y_1\cdots y_q =0\}$. 
In this notation, $F$ is either the \emph{induced divisor} $E\times \IK$ (in the case $\varepsilon = 0$ in \eqref{eq:ep}) or the \emph{extended divisor} $(E\times \IK) \cup (N\times \{0\})$
(in the case $\varepsilon=1$). (Compare with the notation following this remark.)

It will be convenient to use this notation, but it is important to note that, in the algebraic case, the
local product structure of $P$ only makes sense \emph{\'etale locally}, and $\IK$
should be replaced by $\IA^1_{\IL}$, where $\IL$ is the residue field $\IK_a$ of $a$ (see \S\ref{subsec:alg}).
The notation above should always be understood in this way in the algebraic case, though we may not
say so explicitly.
\end{remarks}

In the remainder of this subsection,
$\Phi : (M,D) \to (N,E)$ denotes a dominant monomial morphism, and $\psi:M \to \mathbb{K}$ denotes a function. We assume that $\Psi := (\Phi,\psi): (M,D) \to (N \times \mathbb{K},F)$ is a morphism, where $F$
denotes either the induced divisor $E \times \mathbb{K}$ or the extended divisor 
$(E\times \mathbb{K}) \cup ( N \times \{0\} )$ on $N \times \mathbb{K}$.
(Such $\Psi$ will sometimes be called a ``partially monomial'' morphism.)

\begin{definition}\label{def:InvariantPartMon}\emph{Log differential invariant for a partially monomial morphism.}
Let $\mathcal{J}_1^{\Psi}:= \De^{\Phi}(\psi)$, the sheaf of ideals (of finite type) 
given by all local sections of $\De^\Phi$ applied to the ideal $(\psi)$ generated by $\psi$.
We define the  \emph{log differential order} $\nu_a(\Psi)$ of $\Psi=(\Phi,\psi)$ at $a \in M$ as
$$
\nu_{a}(\Psi)=\mu_{a}(\mathcal{J}_1^{\Psi},\De^{\Phi})+1.
$$
We also define the \emph{log differential closure} $\cJ_\infty^\Psi$ of $\cJ_1^\Psi$ as the log differential closure of the
latter relative to $\Phi$; i.e.,
$$
\cJ_\infty^\Psi := \left(\cJ_1^\Psi\right)_\infty^{\De^\Phi}.
$$
\end{definition}

\begin{remarks}\label{rem:InvariantPartMon}
(1) Consider the chain of ideals associated to $\mathcal{J}_1^{\Psi}$,
\[
\De^{\Phi}(\psi) = \mathcal{J}_1^{\Psi} \subset \mathcal{J}_2^{\Psi} \subset \cdots \mathcal{J}_k^{\Psi} \subset \cdots,
\]
where $\mathcal{J}_{k+1}^{\Psi} = \mathcal{J}_{k}^{\Psi} + \De^{\Phi}(\mathcal{J}_{k}^{\Psi})$, $k=1,2,\ldots$. 
Then $\nu_{a}(\Psi)$ is the minimal $\nu \in \mathbb{N}\cup \{\infty\}$ such that
$\mathcal{J}_{\nu,a}^{\Psi} = \mathcal{J}_{\infty,a}^{\Psi}$ (i.e., $\nu_{a}(\Psi)$ corresponds to 
the number of times we have to derive $\psi$). 

\medskip
\noindent
(2) Note the difference between the preceding chain of ideals and that given by Definitions \ref{def:ChainAndClosure}
for the ideal $\cJ = (\psi)$ (see also Definition \ref{def:InvariantPairs}). 
The index one term in the chain given by Definitions \ref{def:ChainAndClosure} 
is $\cJ_1^{\De^\Phi} = (\psi) + \De^\Phi(\psi)$, whereas here we have $\cJ_1^\Psi = \De^\Phi(\psi)$ 
(the latter does not necessarily contain $\psi$). 
The comparison nevertheless explains the shift by $1$ in the definition of $\nu_a(\Psi)$ above.

There should be no confusion between $\cJ_\infty^\Psi$ above and the notation of Definition \ref{def:InvariantPairs}, because, 
in the latter, $\Phi$ is a monomial morphism, whereas
$\Psi$ is pre-monomial in Definition \ref{def:InvariantPartMon}, and $\Psi$ will be used to denote a pre-monomial morphism
throughout the rest of the paper.

\end{remarks}

\begin{lemma}[Weierstrass-Tschirnhausen normal form II]\label{lem:NormalFormPartialMonomial}
Following the notation above, consider a partially monomial 
morphism $\Psi= (\Phi,\psi):(M,D) \to (N\times \mathbb{K},F)$ 
and a point $a\in M$. 
\begin{enumerate}
\item The morphism $\Psi$ is pre-monomial at $a$ if and only if $\mathcal{J}_{\infty,a}^{\Psi}$ is either identically zero 
or principal and monomial, and $\nu_{a}(\Psi)=1$.
\medskip
\item If $\mathcal{J}_{\infty,a}^{\Psi}$ is principal and monomial, then $d:= \nu_{a}(\Psi) <\infty$. If, moreover, $d>1$, 
then there exists a $\Phi$-monomial coordinate system $(\pmb{u},\pmb{v},\pmb{w})=(\pmb{u},\pmb{v},w_1,\widehat{\pmb{w}}_1)$
at $a$ such that
\[
\psi(\pmb{u},\pmb{v},\pmb{w}) = g(\pmb{u},\pmb{v}) + \pmb{u}^{\pmb{\gamma}}\left( \widetilde{H}(\pmb{u},\pmb{v},\pmb{w}) w_1^d + \sum_{j=0}^{d-2} h_{j}(\pmb{u},\pmb{v},\widehat{\pmb{w}}_1)w^j_1  \right),
\]
where $\mathcal{J}_{\infty,a} = (\pmb{u}^{\pmb{\gamma}})$, $\widetilde{H}(0)\neq 0$, each $h_j(0)=0$,
and $dg \wedge d\Phi:= dg\wedge  d\varphi_1 \wedge \cdots \wedge d\varphi_n = 0$.
\end{enumerate}
\end{lemma}

\begin{proof}
The ``only if'' direction in (1) follows from Lemma \ref{lem:BasisPhiDer} and the normal forms for a pre-monomial morphism 
given in Definition \ref{def:PreMonomial}. 

Suppose that $\mathcal{J}_{\infty,a}^{\Phi}$ is either identically zero or principal and monomial. Then
$d := \nu_{a}(\Psi)< \infty$, by Lemma \ref{lem:BasicPropClosure}(3).
Let $\{X^1,\ldots, X^{m-n}\}$ denote a monomial basis of $\De^\Phi_a$, with respect to
a $\Phi$-monomial coordinate system $(\pmb{u},\pmb{v},\pmb{w})$, as given by Lemma \ref{lem:BasisPhiDer}.

First, if $\mathcal{J}_{\infty,a}^{\Psi}=(0)$, then $\mathcal{J}_{1,a}^{\Psi}=(0)$ and $d=1$. In this case, 
$\psi(\pmb{u},\pmb{v})$ is an eigenvector of $\{X^1,\ldots, X^{m-n}\}$ with corresponding eigenvalue $0$;
therefore, $\Psi$ is pre-monomial.

Secondly, suppose that $\mathcal{J}_{\infty}^{\Psi}=(\pmb{u}^{\pmb{\gamma}})$. Then $\pmb{u}^{\pmb{\gamma}}$ 
is an eigenvector of $\{X^1,\ldots, X^{m-n}\}$, and it follows from Lemma \ref{rk:BasisFormalIdeal} that $\psi$
has formal expansion at $a$ of the form
\[
\widehat{\psi} =\widehat{\psi}_{\pmb{0} \pmb{0}}(\pmb{u},\pmb{v}) + \pmb{u}^{\pmb{\gamma}} R(\pmb{u},\pmb{v},\pmb{w}), 
\]
where
\[
R(\pmb{u},\pmb{v},\pmb{w})=\sum_{\pmb{\delta} \in \mathbb{N}^t} \pmb{w}^{\pmb{\delta}} \sum_{\pmb{\lambda} \in \mathbb{N}^{m-n-t}} \widehat{\psi}_{\pmb{\delta} \pmb{\lambda}}(\pmb{u},\pmb{v}),
\]
$\widehat{\psi}_{\pmb{0} \pmb{0}}(\pmb{u},\pmb{v})$ is an eigenvector with eigenvalue $0$, and 
each $\pmb{u}^{\pmb{\gamma}}\widehat{\psi}_{\pmb{\delta} \pmb{\lambda}}(\pmb{u},\pmb{v})$ is an eigenvector with 
eigenvalue nonzero. 
Moreover, by Lemma \ref{lem:FormalEigenVectors},
\begin{equation}\label{eq:ComputationOfJ}
\mathcal{J}^{\Phi}_{\infty} = \pmb{u}^{\pmb{\gamma}} \left(\widehat{\psi}_{\pmb{\delta} \pmb{\lambda}}(\pmb{u},\pmb{v})\right) = (\pmb{u}^{\pmb{\gamma}}),
\end{equation}
so that $\widehat{\psi}_{\pmb{\delta} \pmb{\lambda}}$ is a unit, for some $\pmb{\delta},\pmb{\lambda}$. 
We consider two cases.

\medskip
\noindent
(a) If $R(\pmb{u},\pmb{v},\pmb{w})$ is a unit, then $\pmb{u}^{\pmb{\gamma}}$ is an eigenvector with nonzero eigenvalue. 
It follows that $\pmb{\gamma}$ is linearly independent of all  $\pmb{\alpha}_i$, and that $d=1$. Then, after a change of coordinates in the source, the morphism $\Psi$ is in pre-monomial form.

\medskip
\noindent
(b) If $R(\pmb{u},\pmb{v},\pmb{w})$ is not a unit, then, by \eqref{eq:ComputationOfJ}, there exists a regular vector field 
germ $X \in \De_{a}^{\Phi}$, and, therefore, a $\Phi$-free $\pmb{w}$-variable. Therefore,
\[
\widehat{\psi} =\widehat{\psi}_{\pmb{0} \pmb{0}}(\pmb{u},\pmb{v}) + \pmb{u}^{\pmb{\gamma}} \left(\widetilde{\psi}(\pmb{w})  + O(\pmb{u},\pmb{v}) \right),
\]
where $\widetilde{\psi}(\pmb{w})$ is a formal power series of order $d$, and 
$\widehat{\psi}_{\pmb{0} \pmb{0}}(\pmb{u},\pmb{v})$ is an eigenvector with eigenvalue $0$. After a linear change of 
coordinates in $\pmb{w}$, we can distinguish one of the $\Phi$-free variables and write
\[
\widehat{\psi} = \widehat{\psi}_{\pmb{0} \pmb{0}}(\pmb{u},\pmb{v}) + \pmb{u}^{\pmb{\gamma}} \left(S(\pmb{u},\pmb{v},\pmb{w}) w^d_1 + \sum_{j=0}^{d-1} s_j(\pmb{u},\pmb{v},\widehat{\pmb{w}}_1)w^j_1 \right),
\]
where $S$ is a unit. It follows from the division axiom \ref{def:quasian}(1) that there is a function $g(\pmb{u},\pmb{v})$ of class $\cQ$ 
such that
\[
\psi = g(\pmb{u},\pmb{v}) + O(\pmb{u}^{\pmb{\gamma}}),
\]
where the
formal expansion of $g(\pmb{u},\pmb{v})$ at $a$ coincides with $\widehat{\psi}_{\pmb{0},\pmb{0}}(\pmb{u},\pmb{v})$ 
modulo the ideal $(\pmb{u}^{\pmb{\gamma}})$. Therefore,
\[
\psi = g(\pmb{u},\pmb{v}) + \pmb{u}^{\pmb{\gamma}} \left(\widetilde{H}(\pmb{u},\pmb{v},\pmb{w}) w^d_1 + \sum_{j=0}^{d-1} h_j(\pmb{u},\pmb{v},\widehat{\pmb{w}}_1)w^j_1 \right),
\]
where $\tH$ is a unit.
As in the proof of Lemma \ref{lem:NormalFormIdeal}, we can use the implicit function theorem (axiom \ref{def:quasian}(2)) to change  
coordinates so that $f_{d-1}=0$;
the new coordinate system is still $\Phi$-monomial because $w_1$ is $\Phi$-free. If $d=1$, we get pre-monomial form after absorbing the unit $\widetilde{H}$ into $w_1$. This completes the proof of (1) and (2). 
\end{proof}

\begin{example}\label{ex:Closure}
Let $\Psi : (\mathbb{C}^3,D) \to (\mathbb{C}^3,E)$ denote the morphism given by
\[
x_1= u^2_1 u_2,\quad
x_2 = u^4_1 u_2 u_3,\quad
z = u^4_1 u_2 u_3 \left( (u_2- u_3)^2 + (u_2 + u_3)^3\right),
\]
where $D=\{u_1 u_2  u_3 =0\}$ and $E = \{x_1 x_2=0\}$. The log Fitting ideals at $0$ are
\[
\mathcal{F}_{2}(\Psi) = (1), \quad \mathcal{F}_{1}(\Psi) = (1), \quad \mathcal{F}_{0}(\Psi)=u^4_1 u_2 u_3 \left(2(u_2- u_3)^2 + 3(u_2 + u_3)^3\right);
\]
in particular, the morphism is not yet monomial. Write $\Psi = (\Phi,\varphi)$, where $\Phi$ corresponds to the monomial part
of $\Psi$ (i.e., the first two components). Then $\Delta^{\Phi}$ is generated by the vector-field
\[
X = -u_1 \frac{\p}{\partial {u_1}} +2u_2\frac{\p}{\partial {u_2}}+2u_3\frac{\p}{\partial {u_3}},
\]
and we calculate
\[
\mathcal{J}_1^{\Psi} =u^4_1 u_2 u_3 \left(2(u_2- u_3)^2 + 3(u_2 + u_3)^3\right) = \mathcal{F}_{0}(\Psi).
\]
Consider the toroidal hull $\mathcal{H}(\mathcal{J}_1^{\Psi})$ and the $\Delta^{\Phi}$-closure 
$\mathcal{J}_{\infty}^{\Psi}$ of $\mathcal{J}_1^{\Psi}$;
\[
\mathcal{H}(\mathcal{J}_1^{\Psi}) = u^4_1 u_2 u_3 \left(u_2^2, u_2u_3,u_3^2\right),\quad
\mathcal{J}_{\infty}^{\Psi}=u^4_1 u_2 u_3 \left((u_2- u_3)^2,(u_2 + u_3)^3\right).
\]
The combinatorial blowing-up $\sigma_1$ with center $(u_2,u_3)$ principalizes $\mathcal{H}(\mathcal{J}_1^{\Psi})$, 
but does not principalize $\mathcal{J}_{\infty}^{\Psi}$. Indeed, at the point of $\s_1^{-1}(0)$ where
$\s_1$ is given in local coordinates $(u_1,\tu_2,v)$ by 
\[
u_2 = \widetilde{u}_2,\quad u_3 = \widetilde{u}_2(v+1),
\]
the morphism $\Psi_1:= \Psi \circ \sigma_1$ is given by
\[
x_1= u^2_1 \tu_2,\quad
x_2 = u^4_1 \tu_2^2 (v+1),\quad
z = u^4_1 \tu_2^4 (v+1) \left( v^2 + \tu_2(2 + v)^3\right),
\]
and
\[
\sigma^{\ast}\Delta^{\Phi} = \left(-u_1\frac{\p}{\partial {u_1}} - 2\tu_2\frac{\p}{\partial {\tu_2}}\right) = \Delta^{\Phi_1},\quad
\text{where }\, \Phi_1 := \Phi \circ \sigma_1.
\]
The morphism $\Psi_1$ is not yet monomial, nor does there yet exist a $\Phi$-free coordinate, as can be seen
from the $\Delta^{\Phi_1}$-closure
$\sigma^{\ast}(\mathcal{J}_{\infty}^{\Psi})= \mathcal{J}_{\infty}^{\Psi_1} = u^4_1 \tu_2^4 \left(v^2,\tu_2\right)$,
but not from the toroidal hull:
$\sigma^{\ast}(\mathcal{H}(\mathcal{J}_1^{\Psi})) = (u^4_1 \tu_2^4)$. 
After principalization of  $\mathcal{J}_{\infty}^{\Psi}$,
we guaranteed the existence of a $\Phi$-free variable $w$. Details of the computation are left
to the reader.
\end{example}

\begin{remark}\label{rem:moninvt}\emph{An invariant for pre-monomialization.}
Let $\Phi : (M,D) \to (N,E)$ be a dominant morphism. Consider the log Fitting ideals $\cF_k(\Phi)$, 
$k = 0,\ldots, n-1$. Given $a\in M$, set
\[
k(a):= \min\{k: \mathcal{F}_k(\Phi) \cdot \mathcal{Q}_a = \mathcal{Q}_a\}
\]
(with $k(a) := n$ if $\cF_{n-1}(\Phi)_a$ is not generated by a unit). Then $k(a)$ is an upper-semicontinuous
invariant.
By Theorem \ref{thm:LogFittingCharacterization}, $k(a)=0$ if and only if $\Phi$ is monomial at $a$.
It follows that, in general, there are coordinates $\bt = (t_1,\ldots,t_n)$ at $b=\Phi(a)$, compatible with $E$,
such that the morphism $\Phi^{k(a)}_{\bt} := (t_1,\ldots,t_{n-k(a)})$ is monomial at $a$.

The morphisms $\Phi^{k(a)}_{\pmb{t},j} = (t_1,\ldots,t_{n-k(a)},t_j)$, $j=n-k(a)+1,\ldots, n$,
are thus partially monomial. Let
\[
\nu_{a}^{k(a)}(\Phi,\pmb{t}) := \min\left\{\nu_a\left(\Phi^{k(a)}_{\pmb{t},j}\right):  j=n-k(a)+1,\ldots, n\right\}.
\]
We can define an invariant $\nu_{a}^{k(a)}(\Phi)$ as the minimum of $\nu_{a}^{k(a)}(\Phi,\pmb{t})$ over
all local coordinate systems $\bt = (t_1,\ldots,t_n)$ at $b=\Phi(a)$, as above (\'etale local, in
the algebraic case). Then the pair
$\left(k(a),\, \nu_{a}^{k(a)}(\Phi)\right)$ (ordered lexicographically) is upper-semicontinuous as a function of $a\in M$.

By Lemma \ref{lem:NormalFormPartialMonomial}, if $k(a)>0$, then there is a coordinate system $\bt$ such that
$\Psi:=\Phi^{k(a)}_{\pmb{t},n-k(a)+1}$ is pre-monomial at $a$ if and only if $\nu_{a}^{k(a)}(\Phi)=1$ and $\cJ^\Psi_{\infty,a}$
is a  principal monomial ideal. We will prove that a partially monomial morphism can be transformed to pre-monomial
essentially by decreasing the invariant $\left(k(a),\, \nu_{a}^{k(a)}(\Phi)\right)$. This invariant will not be formally
needed in the paper, however, since all of our arguments are local (cf.\ Remarks \ref{rem:idea} and \ref{rem:relorder2}).
\end{remark}

\section{Pre-monomial morphisms}\label{sec:premon}
In this section, we develop techniques that will be used to transform a pre-monomial morphism
to monomial, in the inductive proof of our main theorems. This is the only place in the article where
there is a dichotomy between the algebraic/analytic and general real quasianalytic cases, and
where we will need power substitutions in the latter case, precisely because of
the extension problem in real quasianalytic classes (see Examples \ref{ex:halfline}).

Consider a morphism $\Psi$ which is pre-monomial at a point $a$; i.e., $\Psi=(\Phi,\psi)$ in a
neighbourhood of $a$, where $\Phi$ is monomial and $\psi(\bu,\bv)=g(\bu,\bv)+\vp(\bu,\bv)$ as
in Definition \ref{def:PreMonomial}; in particular, $g$ is a remainder term---algebraically dependent
on $\Phi$ in the sense of Lemma \ref{lem:AlgebraicDependentFunctions}. We would like to transform
$g$ into a composite $g=h\circ\Phi$, where $h(\bx,\by,\bz)$ is a function of class $\cQ$, because,
in this case, $\Psi$ is a monomial morphism, in the form \eqref{eq:mon1} after
 a coordinate change $\widetilde{t} = t-h(\bx,\by,\bz)$
in the target (in fact, extra care is needed if $\{t=0\}$ is a component of the divisor $F$---see
Remark \ref{rem:cleaning} and \S\ref{subsec:II}).

Using combinatorial techniques developed in \S\ref{subsec:comb}, we will
reduce this problem to the case that the formal Taylor expansion $G(\bu,\bv)$ of $g$ at $a$
is the composite of $H(\bx^{1/d},\bz)$ (where $H(\bx,\bz)$ is a power series and $\bx^{1/d}:=
(x_1^{1/d},\ldots,x_p^{1/d})$) with the formal expansion of $\Phi$ at $a$. The
challenge is to make further transformations to eliminate the fractional exponents in $H$.

This is simple if we use power substitutions. (We are forced to take this
route in the general real quasianalytic case; see below). On the other hand, while the composite function
problem (i.e., given a formal solution $H$ of the equation $g=h\circ\Phi$, can we solve for $h$
of class $\cQ$?) admits a solution in the analytic or algebraic cases, there is no
solution, in general, in real quasianalytic classes, because of the extension problem. We will 
overcome this obstacle in \S\ref{subsec:realquasian}.

In the analytic (likewise, algebraic) case, we can treat the problem of fractional exponents in $H(\bx^{1/d},\bz)$
by introducing a polynomial relation (with analytic coefficients) satisfied by the latter. (This is the
strategy of Cutkosky \cite{C1,C3}.) By \cite{Gab}, $H(\bx,\bz)$ is the formal expansion of an
analytic function $h(\bx,\bz)$. (Assuming, for simplicity, that $\bx$ is a single variable $x$), the product
$$
R(x,\bz,t) := \prod_{i=1}^d \left(t-h(\varepsilon^ix^{1/d}, \bz)\right),
$$
where $\varepsilon$ denotes the primitive $d$th root of unity,
provides such a relation. In \S\ref{subsec:rel}, we will develop techniques to find sequences
of blowings-up in the source and target, after which there is a relation of degree $<d$; thus
eventually of degree $1$, so that $h(x^{1/d},\bz)$ is itself transformed to an analytic function.

This approach does not adapt to the real quasianalytic case: Consider also
$$
S(\bxi,\bz,t) := \prod_{i=1}^d \left(t-h(\xi_i,\bz)\right),
$$
where $\xi_1,\ldots,\xi_d$ are independent variables. Then $S$ is symmetric in $\xi_1,\ldots,\xi_d$,
so that $S$ is a polynomial in $t$ with coefficients $b_j$ that are functions of $\bz$
and the elementary symmetric polynomials $\s_1,\ldots,\s_d$ in $(\xi_1,\ldots,\xi_d)$. The coefficients
$a_j(x,\bz)$ of $R$ are obtained by substituting $\xi_i = \varepsilon^i x^{1/d}$.

\begin{remark}\label{rem:symm}
If $h$ is, in fact, a power series over a field $\IK$, then the coefficients $a_j$ are also power series over $\IK$
 because $\s_d(x^{1/d}, \varepsilon x^{1/d},\ldots, \varepsilon^{d-1}x^{1/d}) = (-1)^{d-1}x$. If $H$ is convergent (or algebraic), we
 likewise get convergent (or algebraic) coefficients over the original field $\IK$.
 \end{remark}
 
 In the real quasianalytic case, however, the coefficients $b_j$ above are well-defined only on the image
 of the mapping $\s =(\s_1,\ldots,\s_d): \IR^d\to\IR^d$. When we substitute $\xi_i = \varepsilon^i x^{1/d}$, for some
 $x\in\IR$, we do get a real value for $\s$, but this value is not, in general, in $\s(\IR^d)$, so that the $b_j$ cannot
 necessarily be defined at these points, because of the extension problem.
 
\subsection{Combinatorial lemmas}\label{subsec:comb}

\begin{lemma}\label{lem:ModificationTargetExtension}
Let $\Phi: (V,D) \to (W,E)$ denote a morphism.
\begin{enumerate}
\item If $\tau: (\widetilde{W},\widetilde{E}) \to (W,E)$ is a combinatorial blowing-up, then there is a commutative diagram
\begin{equation}\label{eq:mainCommutativeRelationI}
\begin{tikzcd}
(\widetilde{V},\widetilde{D}) \arrow{d}{\widetilde{\Phi}} \arrow{r}{\s} & (V,D) \arrow{d}{\Phi}\\
(\widetilde{W},\widetilde{E}) \arrow{r}{\tau} & (W,E)
\end{tikzcd}
\end{equation}
where $\s$ is a composite of finitely many smooth combinatorial blowings-up.

\medskip
\item Moreover, if $\Phi$ is a monomial morphism, then $\widetilde{\Phi}$ is a monomial morphism.
\end{enumerate}
\end{lemma}

\begin{proof}
(1) The pull-back of the ideal of the centre of $\tau$ is a monomial ideal, and can therefore be principalized by
a sequence of combinatorial blowings-up $\s: \tV\to V$. The morphism $\tPhi$ exists by the universal property
of blowing up. 

(2) Moreover, if $\Phi$ is a monomial morphism, then so is  $\tPhi$, by Corollary \ref{cor:MonCombinatorialBU}.
\end{proof}

In the remainder of this subsection, $\Phi: V\to W$ denotes a monomial morphism given by \eqref{eq:mon1}
with respect to $\Phi$-monomial coordinate systems at points $a\in V$ and $b=\Phi(a)$,
compatible with divisors $D,\,E$, as in 
Definition \ref{def:monom2}. Here $V$ and $W$ are coordinate charts of class $\cQ$, or \'etale coordinate
charts, and the constants $\xi_k$ belong to a field $\IK$ (perhaps a finite extension of the base field).

Every monomial $\bx^{\bga}$, where $\bga \in \IQ^p$, pulls back to a monomial $\bu^{\bal(\bga)}$,
where $\bal(\bga) = \sum_{j=1}^p \ga_j\bal_j$. 
We write $\bu^{\bal(\bga)}=\bx^{\bga}$. Of course, $\bu^{\bal(\bga)}\in \cQ_{U,a}$ if and only if $\bal(\bga)\in\IN^s$.

\begin{lemma}[Algebraically dependent monomials]\label{lem:RelationTrick}
Let 
\[
\cS = \{ \pmb{x}^{\pmb{\gamma}}: \pmb{\gamma}\in \mathbb{Z}^{l},\ \Phi^{\ast} (\pmb{x}^{\pmb{\gamma}}) \in \cQ_{U,a}\}.
\]
Then $\cS$ is a finitely generated multiplicative semigroup (with identity $\pmb{x}^{\pmb{0}}=1$),
and there is a commutative diagram \eqref{eq:mainCommutativeRelationI}
such that
\begin{enumerate}
\item $\s$ and $\tau$ are composites of finitely many smooth combinatorial blowings-up (in particular, $\sigma^{\ast}(\De^{\Phi})= \De^{\widetilde{\Phi}} $);
\smallskip
\item $\widetilde{\Phi}$ is a monomial morphism;
\smallskip
\item $\tau^{\ast}(\cS) \subset \mathcal{Q}_{\widetilde{W},\widetilde{b}}$, for all\, $\widetilde{b} \in \widetilde{\Phi}(\sigma^{-1}(a))$.
\end{enumerate}
\end{lemma}

\begin{proof}
Let $I \subset \mathcal{Q}_{V,a}$ denote the ideal generated by all $\bu^{\bal} \in \mathbb{N}^r$
such that $\bal = \sum \ga_j\bal_j$, where
$\bga=(\ga_1,\ldots,\ga_p)\in \IZ^p$ with at least one negative component $\ga_j$.
Let $\bu^{\pmb{\varepsilon}_1},\ldots,\bu^{\pmb{\varepsilon}_m}$ denote a finite subcollection of the $\bu^{\bal}$
which generates $I$. We claim that,
for every element $\bu^{\bal}$ of $I$, there exist $\la_1,\ldots,\la_m \in \IN$, such that
\begin{equation}\label{eq:ualpha}
\bu^{\bal} = \prod_{l=1}^m (\bu^{\pmb{\varepsilon}_l})^{\la_l}\cdot \bu^{\bbe},\ \text{ where } 
\bu^{\bbe} = \prod_{j=1}^p \bx_j^{\ga_j},\ \ga_j\geq 0.
\end{equation}
Indeed, each $\bu^{\bal} = \bu^{\pmb{\varepsilon}_l}\bu^{\bbe}$, for some $l$. If $\bu^{\bbe} = \prod x_j^{\ga_j}$ with all
$\ga_j\geq 0$, then we are done. Otherwise, $\bu^{\bbe} \in I$ and $|\bbe|<|\bal|$, so the claim follows by induction.

Each $\bu^{\pmb{\varepsilon}_l} = \bx^{\bga_l} = \bx^{\bga_l^+}/\bx^{\bga_l^-}$, where $\bga_l^+,\, \bga_l^- \in \IN^p$,
and $\cS$ is the multiplicative semi-group generated by $\bx^{\bga_1},\ldots,\bx^{\bga_m}$
(cf.\ Gordon's Lemma \cite[p.\,12]{Ful}).

Consider the monomial ideal 
\[
\mathcal{J} := \prod_{l=1}^{m} (\pmb{x}^{\pmb{\gamma}_l^{+}},\pmb{x}^{\pmb{\gamma}_l^{-}}).
\] 
Let $\tau: (\widetilde{W},\widetilde{E})\to (W,E)$ be a sequence of combinatorial blowings-up which principilizes $\mathcal{J}$. 
By Lemma \ref{lem:ModificationTargetExtension}, there is a sequence of combinatorial blowings-up $\sigma$ and a monomial morphism $\widetilde{\Phi}$ such that the diagram \eqref{eq:mainCommutativeRelationI} commutes. In particular, 
$(\tau\circ\widetilde{\Phi})^{\ast}(\cS) = (\Phi\circ\s)^{\ast}(\cS) \subset \mathcal{Q}_{\widetilde{U},\widetilde{a}} $ for every $\widetilde{a} \in \sigma^{-1}(a)$. 

Consider $\widetilde{a} \in \sigma^{-1}(a)$ and $\widetilde{b} = \widetilde{\Phi}(\widetilde{a})$. Since $\tau$ principalizes
$\cJ$, for each $l=1,\ldots, m$, either $\tau^*(\pmb{x}^{\pmb{\gamma}_l^{-}})$ divides $\tau^*(\pmb{x}^{\pmb{\gamma}_l^{+}})$
at $\tb$, 
so that $\tau^{\ast}(\pmb{x}^{\pmb{\gamma}_l}) \in\mathcal{Q}_{\widetilde{W},\widetilde{b}}$, as desired, or 
$\tau^*(\pmb{x}^{\pmb{\gamma}_l^{+}})$ strictly divides $\tau^*(\pmb{x}^{\pmb{\gamma}_l^{-}})$, so that 
$\tau^{\ast}(\pmb{x}^{\pmb{\gamma}_l})$ has a pole at $\widetilde{b}$, yet pulls back to a well-defined function 
at $\widetilde{a}$ (a contradiction).
\end{proof}

\begin{lemma}[Algebraically dependent monomials II]\label{lem:TranslationTrick}
Let $e(b)$ denote the number of components of the divisor $E$ at $b=\Phi(a)$ (i.e., $e(b)=p+q$), and
let $i(a)$ denote the number of components of $E$ at $b$ that are independent at $a$ (i.e., $i(a)=p$). Then there is 
a commutative diagram $\eqref{eq:mainCommutativeRelationI}$ such that
\begin{enumerate}
\item $\s$ and $\tau$ are composites of finitely many smooth combinatorial blowings-up (in particular, $\sigma^{\ast}(\De^{\Phi})= \De^{\widetilde{\Phi}} $);
\smallskip
\item $\widetilde{\Phi}$ is a monomial morphism;
\smallskip
\item for every $\widetilde{a} \in \sigma^{-1}(a)$, the number of components $e(\tb)$ of $\tE$ at 
$\widetilde{b}=\widetilde{\Phi}(\widetilde{a})$ is at most $p$; i.e.,
$e(\widetilde{b}) \leq i(a) = p \leq p+q=e(b)$.
\end{enumerate}
\end{lemma}

\begin{proof}
It will be convenient to keep track of the exceptional coordinates in the target without distinguishing between $\pmb{x}$ 
and $\pmb{y}$, so we will write $(\pmb{x},\pmb{y})=\pmb{X}  = (X_1,\ldots,X_{p+q})$. 
Since $y_k = \pmb{u}^{\pmb{\beta}_k}(\xi_k + v_k)$ ($k=1,\ldots, q$), where $\pmb{\beta}_k$ is $\mathbb{Q}$-linearly dependent 
on the exponents $\{\pmb{\alpha}_j\}$ of the monomials $x_j=\pmb{u}^{\pmb{\alpha}_j}$, there is (a smallest) 
$d_k\in \mathbb{N}$, and $\pmb{\gamma}^{+}_k,\, \pmb{\gamma}^{-}_k \in \mathbb{N}^{p}$ such that:
\begin{equation}\label{eq:LemmaBeta}
y^{d_k}_k \cdot \frac{ \pmb{x}^{\pmb{\gamma}^{+}_k}}{\pmb{x}^{\pmb{\gamma}^{-}_k}} = (\xi_k + v_k)^{d_k},
\qquad k=1,\ldots, q.
\end{equation}

Consider  the monomial ideal 
\[
\mathcal{K} = \prod_{k=1}^{q} (y^{d_k}_k \pmb{x}^{\pmb{\gamma}^{+}_k}, \pmb{x}^{\pmb{\gamma}^{-}}_k).
\]
Let $\tau: (\widetilde{W},\widetilde{E})\to (W,E)$ be a sequence of combinatorial blowings-up which principilizes $\mathcal{K}$. 
By Lemma \ref{lem:ModificationTargetExtension}, there is sequence of combinatorial blowings-up $\sigma$ and a 
monomial morphism $\widetilde{\Phi}$ such that \eqref{eq:mainCommutativeRelationI} commutes. 
Since all blowings-up in $\s$ (and in $\tau$) are combinatorial, the coordinates $(\pmb{v},\pmb{w})$ in the source are preserved
(as are the coordinates $\bz$ in the target). Therefore, at each point $\widetilde{a} \in \sigma^{-1}(a)$,
 there exists a $\widetilde{\Phi}$-compatible coordinate system of the form $(\widetilde{\pmb{u}},\widetilde{\pmb{v}},\widetilde{\pmb{w}},\pmb{v},\pmb{w})$ at $\widetilde{a}$ (i.e., $\widetilde{\pmb{v}}$ and $\widetilde{\pmb{w}}$ 
 represent coordinates of the type ``$\bv$'' and ``$\bw$'', respectively). Note that
\begin{equation}\label{eq:Trick1}
(\tau\circ \widetilde{\Phi})^{\ast} \left(y^{d_k}_k \cdot \frac{ \pmb{x}^{\pmb{\gamma}_k^{+}}}{\pmb{x}^{\pmb{\gamma}_k^{-}}} \right)= (\xi_k + v_k)^{d_k}, \qquad k=1,\ldots, q.
\end{equation}
There is also a coordinate system of the form $(\widetilde{\pmb{x}},\widetilde{\pmb{y}},\widetilde{\pmb{z}},\pmb{z})= (\widetilde{\pmb{X}},\widetilde{\pmb{z}},\pmb{z})$
at $\widetilde{b} := \widetilde{\Phi}(\widetilde{a})$, such that
\begin{equation}\label{eq:Trick2}
X_j = \widetilde{\pmb{X}}^{\pmb{\varepsilon}_j}(\pmb{\rho}+ \widetilde{\pmb{z}})^{\pmb{\delta}_j}, \qquad j=1,\ldots,p+q,
\end{equation}
where the $(\pmb{\varepsilon}_j, \pmb{\delta}_j)$ are integer vectors such that the matrix $E$
with rows $(\pmb{\varepsilon}_j,\pmb{\delta}_j)$, $j=1,\ldots,p+q$, has determinant $1$, and $\pmb{\rho}$ is a vector with
entries all nonzero. It follows from  \eqref{eq:Trick1}, \eqref{eq:Trick2} that, for every $k=1,\ldots, q$, there exists a 
vector $\pmb{\lambda}_k$ such that
\[
(\pmb{\rho}+\widetilde{\pmb{z}} )^{\pmb{\lambda}_k} = (\xi_k+ v_k )^{d_k}
\]
and the matrix $L= (\pmb{\lambda}_k)$ has maximal rank. By the implicit function theorem, we can solve for 
$v_k = F_k(\widetilde{\pmb{z}})$, where the 
matrix $(\nabla F_k)$ has maximal rank. Therefore, after a change of the coordinates $\widetilde{\pmb{z}}$, we can assume that $\widetilde{z}_k = v_k$. The result follows.
\end{proof}

Let us write also $e(a):=e(\Phi(a))$. Note that $e(a)=0$ if and only if $i(a)=0$. We will use Lemma \ref{lem:TranslationTrick}
as the basis for induction on $(e(a),i(a))$ (where $\leq$ means that each component is $\leq$). Condition (3) of the lemma
implies that, for every $\ta\in\s^{-1}(a)$, either $(e(\ta),i(\ta)) < (e(a),i(a))$, or $q = e(\cdot) - i(\cdot) = 0$ both at $a$ and 
at $\ta$. The following is a simple application.

\begin{lemma}[Principal ideals in the target]\label{lem:PrincipalIdealTrick}
Suppose that $\Phi$ is dominant. Let $h$ denote a function of class $\cQ$ on $W$, such that $h\circ\Phi$ is a monomial
$\bu^{\pmb{\varepsilon}}$ times a unit, in some neighbourhood of $a$.
Then (after shrinking $V$ and $W$ to suitable neighbourhoods of $a$ and $b$)
there is a finite number of commutative diagrams 
\begin{equation}\label{eq:PrincipalIdealTrick}
\begin{tikzcd}
(V_\la,D_{\la}) \arrow{d}{\Phi_{\la}} \arrow{r}{\s_\la} & (V,D) \arrow{d}{\Phi}\\
(W_\la,E_{\la}) \arrow{r}{\tau_\la} & (W,E)
\end{tikzcd}
\end{equation}
where
\begin{enumerate}
\item each $\s_\la$ and $\tau_\la$ is a composite of finitely many smooth local combinatorial blowings-up (in particular, $\sigma^{\ast}_{\la}(\De^{\Phi})= \De^{\Phi_{\la}} $);
\smallskip
\item the families of morphisms $\{\s_\la\}$ and $\{\tau_\la\}$ cover $V$ and $W$, respectively, and (in the analytic or
real quasianalytic cases) there are compact subsets $K_\la \subset V_\la$, $L_\la \subset W_\la$, for all $\la$, such that
$\bigcup \s_\la(K_\la)$ and $\bigcup \tau_\la(L_\la)$ are (compact) neighbourhoods of $a$ and $b$,
respectively;
\smallskip
\item each $\Phi_\la$ is a monomial morphism;
\smallskip
\item for each $\la$ and $\ta \in \s_\la^{-1}(a)$, $h_{\lambda} := h\circ \tau_{\lambda}$ generates a (principal) monomial ideal 
in a neighbourhood of $\tb = \Phi_\la(\ta)$.
\end{enumerate}
\end{lemma}

\begin{proof}
We argue by induction on $(e(a),i(a))$. If $e(a)=0$ (equivalently, $i(a)=0$), then $h$ is a unit,
and the result is trivial. So we assume the result for $(e(\cdot),i(\cdot))<(e(a),i(a)))$. We have one
of two possibilities.

\smallskip
(1) $q = e(a)- i(a)) > 0$. In this case, we apply Lemma \ref{lem:TranslationTrick}, to get sequences of combinatorial 
blowings-up $\sigma$ and $\tau$ such that, for all\, $\widetilde{a} \in \sigma^{-1}(a)$, $e(\ta) < e(a)$.
Since $h\circ \tau \circ \widetilde{\Phi}= h\circ \Phi \circ \sigma$ is a monomial times a unit at $\ta$, the assertion is true at $\ta$, by induction. 

\smallskip
(2) $q = e(a)-i(a) = 0$. In this case,
consider the formal Taylor expansions $H$ of $h$ at $b$, and $G$ of $h\circ\Phi$ at $a$.
Since $q=0$ at $a$, we can write
\[
H = H(\bx,\bz) = \sum_{\pmb{\ga} \in \mathbb{N}^{p}} \sum_{\pmb{\de} \in \mathbb{N}^{s}} h_{\pmb{\ga} \pmb{\de}} 
\pmb{x}^{\pmb{\ga}} \pmb{z}^{\pmb{\de}},
\]
and $z_l = v_l$, $l=1,\ldots,s$. On the other hand,
\[
G = G(\bu,\bv) = \pmb{u}^{\pmb{\varepsilon}} \left( \sum_{\pmb{\al} \in \mathbb{N}^{r}} \sum_{\pmb{\de} \in \mathbb{N}^{s}} g_{\pmb{\al}\pmb{\de}} \pmb{u}^{\pmb{\al}} \pmb{v}^{\pmb{\de}} \right),
\]
where $g_{\bzero \bzero} \neq 0$.
Since $\bga \mapsto \bal(\bga):=\sum_{j=1}^p \ga_j\bal_j$ is injective, and $G$ is the pull-back of $H$,
we can write $\bu^{\pmb{\varepsilon}} = \bx^{\bla}$ (i.e., $\pmb{\varepsilon} = \bal(\bla)$), for unique $\bla \in \IN^p$, and (for each
$(\bal,\bde)$ such that $g_{\bal \bde}\neq 0$)
$\bu^{\bal} = \bx^{\pmb{\ga}'}$ (i.e., $\bal = \bal(\pmb{\ga}')$), for unique $\pmb{\ga}' \in \IZ^p$.

Moreover, in this case, we can apply Lemma \ref{lem:RelationTrick}
to get sequences of combinatorial blowings-up $\sigma$ and $\tau$ such that, for all\, $\widetilde{a} \in \sigma^{-1}(a)$, 
\begin{itemize}
\item[either] $(e(\ta),i(\ta))<(e(a),i(a))$, so the result is true at $\ta$, as in (1);
\smallskip
\item[or] $q = e(\ta) -i(\ta) =0$. Since the pullback by $\tau$ to $\tb=\tPhi(\ta)$ of each monomial $\pmb{\ga'}$ above
is a well-defined element of $\mathcal{Q}_{\widetilde{W},\widetilde{b}}$ (by Lemma \ref{lem:RelationTrick}), 
it follows in this case that the pullback of $H$ is a 
monomial $\bx^{\bla}$ times
an invertible power series in the monomial coordinates $(\bx,\bz)$ at $\tb$. Therefore
(by the quasianalyticity and division axioms), $h\circ\tau$ is locally a monomial times a unit.
\end{itemize}

The result follows since $\s$ is proper and property $(4)$ is open in the source.
\end{proof}

\begin{remark}[Simplifying the normal form of a monomial morphism]\label{rk:ReducingNormalForm}
Consider a monomial morphism $\Phi: (M,D) \to (N,E)$ and fix a point $a\in M$. As a consequence of Lemma \ref{lem:TranslationTrick}, 
we can show there is a family of commutative diagrams $\eqref{eq:PrincipalIdealTrick}$ satisfying conditions $(1)$, $(2)$ and $(3)$
of Lemma \ref{lem:PrincipalIdealTrick}, such that, for each $\lambda$,
the open sets $V_{\lambda}$ and $W_{\lambda}$ have $\Phi_{\lambda}$-monomial coordinate systems $(\pmb{u},\pmb{v},\pmb{w})$ and $(\pmb{x},\pmb{z})$ (respectively) with respect to which $\Phi_{\lambda}$ has the form \eqref{eq:mon2}. Indeed, the 
normal form \eqref{eq:mon2}
corresponds to the case that $e(a)-i(a) =0$. We
can now argue by induction 
on $e(a)-i(a)$, using Lemma \ref{lem:TranslationTrick}, just as in the proof of Lemma \ref{lem:PrincipalIdealTrick}.
\end{remark}

In the following subsections, we will use Lemmas \ref{lem:RelationTrick} and \ref{lem:TranslationTrick}
in a way that is similar to that of the preceding proof, together with Lemma \ref{lem:denombound} following.

\begin{lemma}[Bound on the denominators of rational exponents]\label{lem:denombound}
Let $G(\bu,\bv)$ denote a formal power series at $a$ that is algebraically dependent on the morphism $\Phi$;
i.e., $dG \wedge d\Phi = 0$ (see Lemma \ref{lem:AlgebraicDependentFunctions}). Then:
\begin{enumerate}
\item There is a positive integer $d$ and a formal Laurent series
$$
H(\bx,\bv) = \sum_{\bga\in\IZ^p}\sum_{\bde\in\IN^s} h_{\bga \bde}\bx^{\bga}\bv^{\bde}
$$
such that $G(\bu,\bv) = H(\bx^{1/d},\bv)$,
where $\bx^{1/d} := (x_1^{1/d},\ldots,x_p^{1/d})$; i.e.,
$$
G(\bu,\bv) = \sum_{\bga\in\IZ^p}\sum_{\bde\in\IN^s} h_{\bga \bde}\bu^{\bal(\bga)/d}\bv^{\bde}.
$$
(In particular, if $q=0$, then $G(\bu,\bv) = H(\bx^{1/d},\bz)$.)
\smallskip
\item There is a commutative diagram \eqref{eq:mainCommutativeRelationI}, where
$\s$ and $\tau$ are composites of finitely many smooth combinatorial blowings-up,
$\widetilde{\Phi}$ is a monomial morphism, and,
for every\, $\ta\in\s^{-1}(a)$, there are $\tPhi$-monomial coordinate systems $(\tbu,\tbv,\tbw)$ at $\ta$
and $(\tbx,\tby,\tbz)$ at $\tb:=\tPhi(\ta)$, and a formal power series $\tH(\tbx,\tbv)$, such that
$$
\tG(\tbu,\tbv) = \tH(\tbx^{1/d},\tbv),\quad \text{where}\quad \tG := \s^*_{\ta}(G).
$$
\end{enumerate}
\end{lemma}

\begin{proof}
Consider $\Gamma := \{\bga\in\IQ^p: \bal(\bga):=\sum\ga_j\bal_j \in \IN^r\}$. We claim there is
a positive integer $d$ such that every $\bga\in\Gamma$ can be written $\bga = \pmb{\ga'}/d$, where
$\pmb{\ga'}\in\IZ^p$. Indeed, for each $\bga=(\ga_1,\ldots,\ga_p)\in\Gamma$, write
$\ga_j= n_j + p_j/q_j$ with smallest $q_j$, where $n_j,p_j,q_j\in \IZ$, $q_j>0$ and $|p_j|<q_j$, $j=1,\ldots,p$.
Then
$$
\sum \ga_j\bal_j - \sum n_j\bal_j = \sum (p_j/q_j)\bal_j .
$$
Since the left-hand side lies in the integral lattice, and the right-hand side is bounded
(over all $\bga$), there are only finitely many possible denominators $q_j$, for all $\bga$.
Let $d=\lcm\{q_1,\ldots,q_p\}$.

The assertion (1) is a consequence of Lemma \ref{lem:AlgebraicDependentFunctions} and the preceding claim,
and (2) then follows from Lemma \ref{lem:RelationTrick}.
\end{proof}

\subsection{The real quasianalytic case}\label{subsec:realquasian}
In this subsection, $\cQ$ denotes a real quasianalytic class. (The results apply to quasianalytic classes
in general, but we will give stronger results in the algebraic or analytic cases, so, for simplicity of exposition,
we use real notation here.) The following is the main result of this subsection.

\begin{proposition}[Real quasianalytic relations]\label{prop:cleaning}
Let $\Phi: (M,D)\to (N,E)$ denote a dominant morphism of class $\cQ$, which is monomial at a point $a\in M$,
and let $g$ denote a function of class $\cQ$ in a neighbourhood of $a$, which is algebraically dependent on $\Phi$; i.e.,
$$
dg\wedge d\Phi := dg\wedge d\Phi_1\wedge\cdots\wedge d\Phi_n = 0
$$
(see Lemma \ref{lem:AlgebraicDependentFunctions}). Then there are neighbourhoods
$V$ of $a$, $W$ of $b=\Phi(a)$, and a finite number of commutative diagrams 
\begin{equation}\label{eq:diag1}
\begin{tikzcd}
(V_\la,D_{\la}) \arrow{d}{\Phi_{\la}} \arrow{r}{\s_\la} & (V,D) \arrow{d}{\Phi}\\
(W_\la,E_{\la}) \arrow{r}{\tau_\la} & (W,E)
\end{tikzcd}
\end{equation}
such that
\begin{enumerate}
\item each $\s_\la$ and $\tau_\la$ is a composite of finitely many smooth local combinatorial blowings-up and power substitutions (in particular, $\sigma^{\ast}_{\la}(\De^{\Phi})= \De^{\Phi_{\la}} $);
\smallskip
\item the families of morphisms $\{\s_\la\}$ and $\{\tau_\la\}$ cover $V$ and $W$, and there are compact subsets 
$K_\la \subset V_\la$, $L_\la \subset W_\la$, for all $\la$, such that
$\bigcup \s_\la(K_\la)$ and $\bigcup \tau_\la(L_\la)$ are (compact) neighbourhoods of $a$ and $b$, respectively;
\smallskip
\item each $\Phi_\la$ is a monomial morphism;
\smallskip
\item for every $\la$, there exists $h_\la \in \cQ(W_\la)$ such that $h_\la \circ \Phi_\la = g\circ \s_\la$.
\end{enumerate}
\end{proposition}

\begin{remark}\label{rem:cleaning}
Let $\Psi$ denote a morphism that is pre-monomial at a point $a\in M$, as in Definition \ref{def:PreMonomial}.
Write $\Psi=(\Phi,\psi)$, $\psi = g + \vp$, following the notation of the latter. We can
also write $\Psi$ locally as in Remark \ref{rem:premonom2}. Applying Proposition \ref{prop:cleaning} to
$\Phi$ and $g$, we get a family of commutative diagrams
\begin{equation*}
\begin{tikzcd}
(V_{\la},D_{\la}) \arrow{d}{\Psi_{\la}} \arrow{r}{\s_\la} & (V,D) \arrow{d}{\Psi}\\
(W_\la\times\IR,F_{\la}) \arrow{r}{} & (W\times\IR,F)
\end{tikzcd}
\end{equation*}
where $\Psi_{\la}=(\Phi_\la,\psi_\la)$, $\psi_\la=\psi\circ\s_\la$, and the bottom arrow is
$\tau_\la\times\mathrm{identity}$. We can assume that, for each $\la$, $\Psi_\la$ is
a pre-monomial morphism, and $\psi_\la = g_\la + \vp_\la$, where $g_\la = g\circ\s_\la=h_\la\circ\Phi_\la$ and
$\vp_\la = \vp\circ\s_\la$.

If $F = E\times\IR$ (so that each $F_\la = E_\la\times\IR$, then we can make a coordinate change
\begin{equation}\label{eq:coordchange}
t' = t - h_\la
\end{equation} 
in the target of $\Psi_\la$ (for every $\la$); therefore, $\Psi_\la$ is a monomial morphism with
respect to the divisors $D_\la^{g_\la}$ and $F_\la^{g_\la}$ (i.e., after codimension one
blowings-up in the target and source, if necessary, according to Remarks \ref{rem:premonom1}). 
This will provide one
of the principal steps in the inductive proof our main theorems
(see Step II in Remark \ref{rem:idea}, or II($m,n$) in Section \ref{sec:LocalMon}), in the case that $\{t=0\}$ is 
not a component of the divisor $F$. If $\{t=0\}$ is a component of $F$, however, \eqref{eq:coordchange}
is not a valid coordinate change, and an additional argument will be needed (see \S\ref{subsec:II}).
\end{remark}

We will use the techniques of \S\ref{subsec:comb} together with the \emph{quasianalytic continuation theorem}
\cite[Thm.\,1.3]{BBC} and the following result to prove Proposition
\ref{prop:cleaning}.

\begin{proposition}[Regularity of a real quasianalytic relation]\label{prop:QuasianalyticExtensionReplacement}
Let $\Phi: (M,D)\to (N,E)$ denote a dominant morphism of class $\cQ$, which is monomial at a point $a\in M$.
Let $K$ be a compact neighbourhood of $a$, and let $g$ denote a function of class $\cQ$ in a neighbourhood of $K$.
Assume that $g|_K = h\circ\Phi$, where $h: \Phi(K)\to\IR$ is a function. 
Then there are neighbourhoods
$V$ of $a$, $W$ of $b=\Phi(a)$, and a finite number of commutative diagrams \eqref{eq:diag1}
such that
\begin{enumerate}
\item each $\s_\la$ and $\tau_\la$ is a composite of finitely many smooth local combinatorial blowings-up and power substitutions (in particular, $\sigma^{\ast}_{\la}(\De^{\Phi})= \De^{\Phi_{\la}} $);
\smallskip
\item the families of morphisms $\{\s_\la\}$ and $\{\tau_\la\}$ cover $V$ and $W$, and there are compact subsets 
$K_\la \subset V_\la$, $L_\la \subset W_\la$, for all $\la$, such that
$\bigcup \s_\la(K_\la)$ and $\bigcup \tau_\la(L_\la)$ are (compact) neighbourhoods of $a$ and $b$, respectively;
\smallskip
\item each $\Phi_\la$ is a monomial morphism;
\smallskip
\item $h\circ\tau_\la \in \cQ(W_\la)$, for all $\la$.
\end{enumerate}
\end{proposition}

\begin{proof}[Proof of Proposition \ref{prop:cleaning}]
We argue by induction on $(e(a),i(a))$ (see Lemma \ref{lem:TranslationTrick}). 
If $e(a)=0$ (equivalently, $i(a)=0$), then $\Phi$ is the identity morphism, so we can take $h=g$.
So we can assume the result for $(e(\cdot),i(\cdot))<(e(a),i(a)))$. We have one of two possibilities.

\smallskip
(1) $q = e(a)- i(a)) > 0$. In this case, we apply Lemma \ref{lem:TranslationTrick}, to get sequences of combinatorial 
blowings-up $\sigma$ and $\tau$ such that, for all $\widetilde{a} \in \sigma^{-1}(a)$, $e(\ta) < e(a)$. Since $\tPhi$
is monomial and $\tg = g\circ\s$ is algebraically dependent on $\tPhi$ at $\ta$, $\ta \in \sigma^{-1}(a)$, 
the result follows by induction.

\smallskip
(2) $q = e(a)- i(a)) =0$. In this case, consider the formal Taylor expansion $G(\bu,\bv)$ of $g$ at $a$. By 
Lemma \ref{lem:denombound}(1), there is a positive integer $d$ such that $G(\bu,\bv) = H(\bx^{1/d},\bz)$, where
$H(\bx,\bz)$ is a formal Laurent series. Moreover, we can apply Lemma \ref{lem:denombound}(2) 
to get sequences of combinatorial blowings-up $\sigma$ and $\tau$ such that , for all $\widetilde{a} \in \sigma^{-1}(a)$, 
\begin{itemize}
\item[either] $(e(\ta),i(\ta))<(e(a),i(a))$, so the result is true at $\ta$, by induction;
\smallskip
\item[or] $q = e(\ta) -i(\ta) =0$, and there is a formal power series $\tH(\tx,\tz)$ such that
$G(\tbu,\tbv) = H(\tbx^{1/d},\tbz)$, where $\tG$ is the formal expansion of $g\circ\s$ at $\ta$. Since $\s$ is
proper, it is enough to prove the assertion of the proposition at $\ta$. Now, by a power substitution with
respect to the coordinates $(\tu_1,\ldots, \tu_r)$ in the source, and a power substitution with
respect to $(\tx_1,\ldots,\tx_p)$ in the target, we can reduce to the case that $d=1$. So we have
reduced the proposition to the case that $G = H\circ\hPhi_a$, where $H$ is a power series and
$\hPhi_a$ is the formal expansion of $\Phi$ at $a$.
\end{itemize}

In the latter case, by the quasianalytic continuation theorem \cite[Thm.\,1.3]{BBC}, there is a compact
neighbourhood $K$ of $a$ such that, if $a'\in K$ and $\hg_{a'}$ denotes the formal expansion of $g$
at $a'$, then $\hg_{a'} = H_{\Phi(a')}\circ\hPhi_{a'}$, where $H_{\Phi(a')}$ is a formal power series
at $\Phi(a')$ (depending only on the image point). In particular, there is a function $h$ on $\Phi(K)$
such that $g=h\circ\Phi$ on $K$. (Of course, $h$ is continuous, but we do not need to use this.)
The result now follows from Proposition \ref{prop:QuasianalyticExtensionReplacement}.
\end{proof}

\begin{remark}\label{rem:cleaning1}
In the case that $\cQ$ is the class of real-analytic functions, we do not need Proposition
\ref{prop:QuasianalyticExtensionReplacement} because the condition $G = H\circ\hPhi_a$,
where $G$ is convergent and $\Phi$ is generically of maximal rank, implies
that $H$ is convergent; this result was first proved by Gabrielov \cite{Gab}
and three different proofs can be found in \cite{BBC,EH,Malg}.
\end{remark}

It remains to prove Proposition \ref{prop:QuasianalyticExtensionReplacement}, which addresses
the problem of extension of a quasianalytic function defined on $\Phi(K)$ to a neighbourhood of $b=\Phi(a)$.
Our proof of Proposition \ref{prop:QuasianalyticExtensionReplacement} depends on explicit combinatorial
methods involving power substitutions and combinatorial blowings-up, for which it will be more convenient
to use cubes rather than arbitrary compact sets. We begin with two lemmas and a more careful
coordinate description of combinatorial blowings-up and power substitutions.

Let $\Phi: (V,D) \to (W,E)$ denote a monomial morphism given by \eqref{eq:mon1}
with respect to $\Phi$-monomial coordinate systems at points $a\in V$ and $b=\Phi(a)$
(where $V$ and $W$ are coordinate charts of class $\cQ$),
compatible with divisors $D,\,E$, as in 
Definition \ref{def:monom2}. Assume that $\Phi$ is dominant, so that $s'=s$. We will also
ignore the free coordinates $(w_1,\ldots,w_t)$ since they have no effect on the arguments
following. Thus $m=r+s$, $n=p+s$, and we can assume that $V=\IR^{r+s}$, $W=\IR^{p+s}$.
The mapping $\Phi: \IR^{r+s}\to \IR^{p+s}$ is a submersion on $O\times\IR^s$, for any
open orthant $O$ of $\IR^r$.

For any $l \in \IN$, write 
\begin{equation*}
C^l \,= [-1,1]^l \subset \IR^{l},\qquad
C^l_+ = [0,1]^l \subset \IR^{l}.
\end{equation*}

\begin{lemma}[Image of a real monomial mapping]\label{lem:cubes}
Consider the finite family of monomial mappings $\Phi_\mu: \IR^{p+s}\to \IR^{p+s}$ that
we obtain by setting $r-p$ of the coordinates $u_i$ each $= \pm 1$ in \eqref{eq:mon1}, and
where we take only the mappings of this kind which have generic rank $p+s$. Then
$$
\Phi(C^{r+s}) = \bigcup_\mu \Phi_\mu(C^{p+s}).
$$
\end{lemma}

\begin{proof} It is enough to prove that 
\begin{equation}\label{eq:cubes}
\Phi(C^r_+\times C^s) = \bigcup_\mu \Phi_\mu(C^p_+\times C^s).
\end{equation}
We will prove this assertion by induction on $r-p$. (It is trivial if $r=p$.) For each
$i=1,\ldots,r$, set $H_i^1 := \{u_i=1\}$ and $\Phi^{(i)} := \Phi|_{H_i^1}$.

Suppose that $r>p$. By induction, we can assume that $\Phi^{(i)}$ satisfies \eqref{eq:cubes},
for each $i$ such that $\Phi^{(i)}$ has generic rank $p+s$. (Clearly, $\Phi^{(i)}$ has no
component $x_j=\bu^{\bzero} = 1$ in this case.) Therefore, it is enough to prove that
\begin{equation}\label{eq:codim1cubes}
\Phi(C^r_+\times C^s) = \bigcup_{i=1}^r \Phi^{(i)}(C^{r-1}_+\times C^s)
\end{equation}
(and we can include in the preceding union only those $i$ such that $\Phi^{(i)}$ has 
generic rank $p+s$). The right-hand side of \eqref{eq:codim1cubes} is the image by $\Phi$ 
of the closure of $\left((\bdry C^r_+)\times C^s\right) \cap (O\times\IR^s)$,
where $O$ is the open positive orthant of $\IR^r$; therefore, \eqref{eq:codim1cubes} is a consequence of the
following claim.

\smallskip
\noindent
Claim. Let $(a,b) \in (0,1)^r\times [-1,1]^s$. Then $\Phi^{-1}(\Phi(a,b))$ intersects
$(\bdry C^r_+)\times C^s$ at points of $O\times \IR^s$.
\smallskip

To establish this claim, we will show that, after a permutation of the variables $(u_1,\ldots,u_r)$, the
system of equations $\bu^{\bal_j} = c_j$, $j=1,\ldots,p$ (for any $c_j>0$), can be rewritten as 
as $u_i^{m_i} = e_i u_{i+1}^{\ga_{i,i+1}}\cdots u_r^{\ga_{ir}}$ in the positive orthant, where,
for all $i$, $e_i> 0$, $m_i$ is a positive integer, and the exponents $\ga_{il}$ are
integers (perhaps $\leq 0$). It follows that the fibre $\Phi^{-1}(c,d)$ (in $O\times\IR^s$), where
$c = (c_1,\ldots,c_p)$ and $d \in \IR^s$, 
is an unbounded connected set (smooth, of dimension $r-p$), and the claim is an immediate consequence.

Of course, the first equation $\bu^{\bal_1} = c_1$ can be rewritten $u_1^{\al_{11}} =
c_1 u_2^{-\al_{12}}\cdots u_r^{-\al_{1r}}$, after a permutation of the variables.
Suppose that the first $h$ equations can be rewritten as asserted. Then, after
raising the $(h+1)$st equation to a suitable power and substituting
$u_i^{m_i} = e_i u_{i+1}^{\ga_{i,i+1}}\cdots u_r^{\ga_{ir}}$,
$i=1,\ldots,h$, the $(h+1)$st equation can be rewritten as required, after a 
permutation of the variables $(u_{h+1},\ldots, u_r)$. 
\end{proof} 

Consider $\IR^{r+s}$ with coordinates $(\bu,\bv) = (u_1,\ldots,u_r, v_1,\ldots,v_s)$ as above, and exceptional
divisor $D := \{u_1\cdots u_r = 0\}$.

\begin{definition}\label{def:combblup}\emph{Standard coordinate charts of a combinatorial blowing-up.}
Let $\rho: (\tV,\tD)\to (V,D)$ denote a combinatorial blowing-up (see Definitions \ref{def:compat}).
The centre of $\rho$ is of the form $\{u_1=\cdots = u_t = 0\}$, for some $t\leq r$,
after a permutation of the $u$-coordinates. In this case, $\tV$ is covered by $t$
coordinate charts $\tV_l,\, l=1,\ldots,t$, where $\tV_l$ has coordinates $(\tu_1,\ldots,\tu_r, v_1,\ldots,v_s)$
in which $\rho$ is given by $u_i = \tu_l\tu_i$ if $i\in \{1,\ldots,t\}\backslash\{l\}$, and $u_i=\tu_i$ otherwise.
(The transform $\tD$ of $D$ by $\rho$ is given in $\tV_l$ by $\tD = \{\tu_1\cdots \tu_r = 0\}$; for each $i\neq l$, $\{\tu_i = 0\}$ is
the strict transform of $\{u_i=0\}$, and $\{\tu_l=0\}$ is the new component of $\tD$.) We call the charts $V_l$ the
\emph{standard coordinate charts} of $\rho$. Note that the cubes $C^{r+s}$  in $\tV_l=\IR^{r+s}$, for all $l$, cover the cube 
$C^{r+s}$ in $V=\IR^{r+s}$.

If $\rho: \tV\to V=\IR^{r+s}$ is a composite of combinatorial blowings-up, then the \emph{standard coordinate charts}
of $\tV$ are defined in an evident way, by induction. The cubes $C^{r+s}$ in the standard coordinate charts of $\tV$
cover the cube $C^{r+s}$ in $V=\IR^{r+s}$.
\end{definition}

\begin{definition}\label{def:powersubst}\emph{Standard coordinate charts of a power substitution.}
Consider a power substitution $\rho: \IR^{r+s} \to \IR^{r+s}$ relative to $D$; say 
$u_i = \tu_i^{k_i},\, v_k=\tv_k$, where each $k_i$ is a positive integer.
(The transform $\tD$ is $\tD := \{\tu_1\cdots\tu_r = 0\}$.) 
We also consider $\rho_{\pmb{\varepsilon}}:  \IR^{r+s} \to \IR^{r+s}$, 
given by
$u_i = \varepsilon_i \tu_i^{k_i}$ (and $v_k=\tv_k$), where $\varepsilon_i = \pm 1$ if $k_i$ is even, and
$\varepsilon_i = 1$ if $k_i$ is odd
(recall Definitions \ref{def:defs}).

Following Definitions \ref{def:defs},
we extend $\rho$ to a more general \emph{power substitution}
$$
P: \coprod_{\pmb{\varepsilon}} \IR^{r+s}_{\pmb{\varepsilon}} \to \IR^{r+s},
$$
where $\coprod$ denotes disjoint union, $\pmb{\varepsilon}= (\varepsilon_1,\ldots,\varepsilon_r)$
(with $\varepsilon_i $ as above), each $\IR^{r+s}_{\pmb{\varepsilon}}$ is 
a copy of $\IR^{r+s}$, and $P|_{\IR^{r+s}_{\pmb{\varepsilon}}} = \rho_{\pmb{\varepsilon}}$.
We call the
$\IR^{r+s}_{\pmb{\varepsilon}}$ the \emph{standard coordinate charts} of (the source of) $P$. We define the \emph{transform}
$\tD$ of $D$ by $P$ as above. If
$C^{r+s}_{\pmb{\varepsilon}}$ denotes the cube $C^{r+s}$ in $\IR^{r+s}_{\pmb{\varepsilon}}$, then 
$\bigcup_{\pmb{\varepsilon}} \rho_{\pmb{\varepsilon}}(C^{r+s}_{\pmb{\varepsilon}}) =  C_{r+s} \subset \IR^{r+s}$.
\end{definition}

\begin{remark}\label{rem:powersubst}\emph{Lifting of a power substitution.}
Consider a power substitution $\tau: x_j = \tx_j^{t_j}$, $j=1,\ldots,p$ (with respect to $E$) in the target 
of $\Phi$ (other variables unchanged). Then there is a commutative diagram
\begin{equation}\label{eq:powersubst1}
\begin{tikzcd}
\IR^{r+s} \arrow{d}{\tPhi} \arrow{r}{\rho} & \IR^{r+s} \arrow{d}{\Phi}\\
\IR^{p+s} \arrow{r}{\tau} & \IR^{p+s}
\end{tikzcd}
\end{equation}
where $\rho$ is a power substitution  $u_i = \tu_i^{k_i}$, $i=1,\ldots,r$, and $\tPhi$ is a monomial morphism in 
standard form \eqref{eq:mon1}. We say that $\rho$ is a \emph{lifting} of $\tau$ by $\Phi$. Clearly,
there is a lifting $\rho: u_i = \tu_i^{t}$, where $t=\lcm\{t_1,\ldots,t_p\}$.

Let $Q: \coprod_{\pmb{\delta}} \IR^{p+s}_{\pmb{\delta}} \to \IR^{p+s}$, $\pmb{\delta} \in \{-1,1\}^p$ as above, denote
the generalized power substitution induced by $\tau$ (in particular, $Q|_{\IR^{p+s}_{\pmb{\delta}}} = \tau_{\pmb{\delta}}$, 
for each $\pmb{\delta}$, as above). Then \eqref{eq:powersubst1} induces a commutative diagram
\begin{equation}\label{eq:powersubst2}
\begin{tikzcd}
\coprod_{\pmb{\varepsilon}} \IR^{r+s}_{\pmb{\varepsilon}} \arrow{d}{\tPhi} \arrow{r}{P} & \IR^{r+s} \arrow{d}{\Phi}\\
\coprod_{\bde} \IR^{p+s}_{\bde(\pmb{\varepsilon})} \arrow{r}{Q} & \IR^{p+s}
\end{tikzcd}
\end{equation}
where $P$ is induced by $\rho$ and, for each $\pmb{\varepsilon}$, $\tPhi_{\pmb{\varepsilon}} := \tPhi |_{\IR^{r+s}_{\pmb{\varepsilon}}}:
\IR^{r+s}_{\pmb{\varepsilon}} \to \IR^{p+s}_{\bde}$, for a unique $\bde = \bde(\pmb{\varepsilon})$, and $\tPhi_{\pmb{\varepsilon}}$ is given
simply by the morphism $\tPhi$ of \eqref{eq:powersubst1}. (Not all $\IR^{p+s}_{\bde}$ are necessarily covered; i.e.,
not all $\bde$ necessarily $= \bde(\pmb{\varepsilon})$, for some $\pmb{\varepsilon}$.) We call $P$ a \emph{lifting} of $Q$.

We can replace \eqref{eq:powersubst2} by the collection of relations $\Phi\circ \rho_{\pmb{\varepsilon}} 
= \tau_{\bde(\pmb{\varepsilon})} \circ \tPhi_{\pmb{\varepsilon}}$ corresponding to the restriction to standard
coordinate charts of the power substitution $P$ (as we do in Lemma \ref{lem:transftoid} and the proof of Proposition
\ref{prop:QuasianalyticExtensionReplacement} below, although $Q$ above is needed to be able to say that 
we have a covering of a neighbourhood in the target of $\Phi$).
\end{remark}

\begin{lemma}[Factorization by blowings-up and power substitutions]\label{lem:transftoid}
Suppose that $r=p$, so that $\Phi: \IR^{p+s} \to \IR^{p+s}$. Then there is a finite number
of commutative diagrams
\begin{equation}\label{eq:transftoid}
\begin{tikzcd}
\IR^{p+s}_\nu \arrow{d}{\Phi_\nu} \arrow{r}{\s_\nu} & \IR^{p+s} \arrow{d}{\Phi}\\
\IR^{p+s}_\nu\arrow{r}{\tau_\nu} & \IR^{p+s}
\end{tikzcd}
\end{equation}
where
\begin{enumerate}
\item each $\s_\nu$ and $\tau_\nu$ is (the restriction to a coordinate chart of) the composite
of a finite sequence of local combinatorial blowings-up and power substitutions over standard
coordinate charts (thus $\s_\nu^*(\De^{\Phi}) = \De^{\Phi_\nu}$);
\smallskip
\item each $\Phi_\nu =$ identity;
\smallskip
\item $\bigcup_\nu\s_\nu(C^{p+s}_\nu)$ is the cube $C^{p+s}$ in the
source of $\Phi$.
\end{enumerate}
\end{lemma}

\begin{proof} First assume that $s=0$. Let $I_1 := \{j \in \{1,\ldots,p\}: \al_{j1}\neq 0\}$. We  reduce
to the case that $\al_{j1} = 1$, for all $j\in I_1$, by making a power substitution $\rho_{\pmb{\varepsilon}}: u_1 = 
\varepsilon_1 \tu_1,
u_i = \varepsilon_i \tu_i^q$,
$i\geq 2$, (with respect to each standard chart) in the source, where $q = \lcm\{\al_{j1}: j\in I_1\}$, to make each $\al_{ji}$ a multiple of $\al_{j1}$ ($j\in I_1$) (and $\pmb{\varepsilon}_1 = 1$,
$\pmb{\varepsilon}_i = \pm 1$, $q$ even, or $\pmb{\varepsilon}_i =1$, $q$ odd, $i\geq 2$), and then making a
corresponding power substitution $x_j = \pmb{\varepsilon}^{\bal_j}\tx_j^{\al_{j1}}$
(where $\pmb{\varepsilon}=(\varepsilon_1,\ldots,\varepsilon_p)$), $j=1,\ldots,p$, in the target.

Now, we make blowings-up in the source (combinatorial with respect to $\{u_2\cdots u_p \allowbreak = 0\}$)
to order all the exponents $\bal_j$ (while keeping all $\al_{j1} = 1$). Then, by combinatorial
blowings-up in the target followed by a reordering of the variables, we reduce to the case
that $\al_{11}=1$ and $\al_{j1}=0$, for all $j>1$. 

By repeating the preceding argument using $x_2,\ldots,x_p$, then using $x_3,\ldots,x_p$, etc.,
we transform the matrix $(\al_{ji})$ to an upper-triangular matrix with diagonal entries $=1$.

By further combinatorial blowings-up in the target, we can make all $a_{jp}=0$, $j=1,\ldots,p-1$.
Finally we can repeat this argument for the $(p-1)$st column, $(p-2)$nd column, etc., to transform $(\al_{ji})$
to the identity matrix, as required for the case $s=0$.

It follows that, given $\Phi: \IR^{p+s}\to \IR^{p+s}$, there is a finite number of 
commutative diagrams \eqref{eq:transftoid}, where (1), (3) are satisfied, and each $\Phi_\nu$ has
the form \eqref{eq:mon1} with $x_j = u_j$, $j=1,\ldots,p$. The lemma follows by further combinatorial
blowings-up over coordinate charts in the target, for each $\nu$.
\end{proof}

\begin{proof}[Proof of Proposition \ref{prop:QuasianalyticExtensionReplacement}]
We can assume that $\Phi$ is the morphism $\Phi: (V=\IR^{r+s},D) \allowbreak\to (W=\IR^{p+s},E)$ above,
and that $K$ is the cube $C^{r+s}$ in $V$.
Consider the finite family of morphisms $\Phi_\mu: \IR^{p+s} \to \IR^{p+s}$ given by Lemma \ref{lem:cubes}.
For each $\mu$, consider the finite family of commutative diagrams (indexed by $\nu \in \La(\mu)$, say) given
by Lemma \ref{lem:transftoid}; we write $\s_{\mu\nu},\,\tau_{\mu\nu},\,\Phi_{\mu\nu}$ for the morphisms
involved, so that $\Phi_\mu\circ\s_{\mu\nu} = \tau_{\mu\nu}\circ\Phi_{\mu\nu}$ and $\Phi_{\mu\nu}=$ identity.
Since each $\Phi_{\mu\nu}=$ identity, if follows that, for each $\mu$,
$$
\bigcup_{\nu\in\La(\mu)} \tau_{\mu\nu}(C^{p+s}_{\mu\nu}) = \Phi_\mu(C^{p+s})
$$
and each $h\circ \tau_{\mu\nu}$ extends to a function of class $\cQ$ in a neighbourhood of $C^{p+s}_{\mu\nu}$.
Moreover, by Lemma \ref{lem:cubes},
$$
\bigcup_{\mu,\nu} \tau_{\mu\nu}(C^{p+s}_{\mu\nu}) = \Phi(C^{r+s}).
$$

For each $\mu$, we apply Lemma \ref{lem:ModificationTargetExtension} or Remark \ref{rem:powersubst}
successively to the local combinatorial blowings-up or 
power substitutions
comprising $\tau_{\mu\nu}$, to lift $\tau_{\mu\nu}$ to the source $V$ of $\Phi$. 
We obtain a finite family of commutative diagrams
\[
\begin{tikzcd}
V_{\mu\nu} \arrow{d}{\Psi_{\mu\nu}} \arrow{r}{\rho_{\mu\nu}} & V \arrow{d}{\Phi}\\
W_{\mu\nu} \arrow{r}{\tau_{\mu\nu}} & W
\end{tikzcd}
\]
where $\Psi_{\mu\nu}$ is a monomial morphism, and the result follows easily.
\end{proof}

\subsection{Relations for a pre-monomial morphism}\label{subsec:rel}

\begin{definition}\label{def:PreMonomial-Relation}
Let $\Psi=(\Phi,\psi): (M,D)\to (N\times \mathbb{K},F)$ denote a morphism which is pre-monomial at
a point $a\in M$ (see Definition \ref{def:PreMonomial} and Remarks \ref{rem:premonom2}).
Let $g$ denote a remainder term for $\Psi$ at $a$. (We can assume that $\psi(a)=0=g(a)$;
in particular, $\Psi(a)= (b,0)$, where $b=\Phi(a)$.) A \emph{relation} for $g$ is a nonzero germ
of a function $R$ of class $\cQ$ on $N\times\IK$ at $(b,0)$ (i.e., a function $R(x,y,z,t)$ in coordinates
in the notation of Definition \ref{def:PreMonomial}), such that
$$
R(\Phi, g) = 0.
$$
\end{definition}

\begin{remarks}\label{rem:rel} (1) The existence of a remainder and associated relation is a
condition that is open in the source of $\Psi$.

\medskip\noindent
(2) In the real quasianalytic case, Proposition \ref{prop:cleaning} shows that, after suitable
sequences of local combinatorial blowings-up and power substitutions in the source and target
of $\Phi$, we can assume there is a relation of the form 
\begin{equation}\label{eq:realrel}
R(x,y,z,t)=t-h(x,y,z).
\end{equation}

\noindent
(3) If $\Psi$ has a remainder $g$ at $a$, and a relation $R$ which is monomial with respect
to the divisor $F$, then $g=0$. In this case, after codimension one blowings-up of the target and
source if necessary (according to Remarks \ref{rem:premonom1}),
the morphism $\Psi$ is monomial at $a$. Note that neither of these codimension one blowings-up
change the sheaf of derivations $\De^\Psi$. (This is clear for the blowing-up of the target. For the 
blowing-up with centre $\{w_1=0\}$ in the source, as in Remarks \ref{rem:premonom1}, it is true because
$X(w_1)=0$, for every $X\in \De^\Psi$, since $X$ annihilates $\psi,\, g$ and $\bu^{\bga}$.)
\end{remarks}

\begin{definition}[Order of a relation]\label{def:PreMonomial-OrderRelation}
Given a relation $R$ as in Defintion \ref{def:PreMonomial-Relation}, we define the \emph{order}
$\rho_a(R)$ as
$$
\rho_a(R) := \nu_{(b,0)}(\cI_R, \pi) = \mu_{(b,0)}(\cI_R, \De^\pi),
$$
where $\cI_R$ is the ideal generated by $R$, and $\pi: N\times\IK \to N$ denotes the projection
(see Definition \ref{def:InvariantPairs}).
\end{definition}

\begin{remarks}\label{rem:relorder1} (1) The sheaf $\De^\pi$ of log derivatives tangent to $\pi$
is generated either by $\p /\p t$ or by $t\,\p/\p t$, according as $F$ is the induced divisor $E\times\IK$
or the extended divisor $(E\times\IK) \cup (N\times\{0\})$.

\medskip\noindent
(2) If $g=0$, then $R = t$ is relation, and $\rho_a(R) =1$ or $0$, according as $F$ is the induced
or extended divisor. On the other hand, if $\rho_a(R)=0$, then $\Psi$ is monomial at $a$ and
$F$ is necessarily the extended divisor (i.e., $\{t=0\} \subset F$).

\medskip\noindent
(3) In \eqref{eq:realrel}, if $h\neq 0$ (i.e., $g\neq 0$), then $\rho_a(R) = 1$ (for either case of $F$).
\end{remarks}

\begin{remark}\label{rem:relorder2}
We can define an invariant $\rho_a(\Psi)$ of $\Psi$ at $a$ as the minimum of $\rho_a(R)$ over
all remainders $g$ and associated relations $R$. (We take $\rho_a(\Psi) = \infty$ if there is
no $R$ for any $g$.) Although we do not formally need $\rho_a(\Psi)$ for the local constructions
in this article, we will prove that a pre-monomial morphism can be transformed to monomial
essentially by decreasing this invariant (cf.\ Remarks \ref{rem:idea} and \ref{rem:moninvt}).
\end{remark}

\begin{example}\label{ex:Relation}
Let $\Psi=(\Phi,\psi): (\mathbb{C}^4,E) \to (\mathbb{C}^3,D)$ denote the pre-monomial morphism 
\[
x_1=u_1^2u_2^2,\quad
x_2=u_2^2u_3^2,\quad
t= u_1u_2^2u_3 + u_1^3u_2^3 + u_1u_2^{5}(w-1) = g(\bu) +  u_1u_2^{5}(w-1),
\]
where $E = \{u_1 u_2 u_3 =0\}$, $D=\{x_1 x_2 t=0\}$, and $g(\bu) =  u_1u_2^2u_3 + u_1^3u_2^3$
is a remainder.
Then $g(\pmb{u}) = h \circ \Phi$, where $h(\pmb{x}) = (x_1x_2)^{1/2} + x_1^{3/2}$, and
\[
R(\pmb{x},t) = t^4  - x_1(x_1^2+x_2)t^2 + x_1^2(x_1^2-x_2)^2 
\]
is a relation for $g$. 
Note that $\Delta^{ \pi } = (t\, \p/\partial t)$, where $\pi$ is the projection $\pi(\pmb{x},t)=t$. It is easy
to calculate $\rho_{\pmb{0}}(R) = 2$.
\end{example}

The following proposition isolates the only part of the proofs of our main theorems in the
analytic or algebraic cases, where we explicitly require that our morphism be analytic or algebraic.

\begin{proposition}[Relation in the analytic or algebraic case]\label{prop:PreparationPreMonomial2}
Let $\Psi=(\Phi,\psi): (M,D)\to (N\times \mathbb{K},F)$ denote an analytic or algebraic morphism, 
which is pre-monomial at a point $a \in M$. Then there exist open neighbourhoods
$V$ of $a$ in $M$ and $W$ of $b:=\Phi(a)$ in $N$, and a finite number of commutative diagrams 
\begin{equation}\label{eq:transftorel}
\begin{tikzcd}
(V_\la,D_{\la}) \arrow{d}{\Phi_{\la}} \arrow{r}{\s_\la} & (V,D) \arrow{d}{\Phi}\\
(W_\la,E_{\la}) \arrow{r}{\tau_\la} & (W,E)
\end{tikzcd}
\end{equation}
such that 
\begin{enumerate}
\item each $\s_\la$ and $\tau_\la$ is a composite of finitely many smooth local combinatorial blowings-up 
(in particular, $\sigma^{\ast}_{\la}(\De^{\Phi})= \De^{\Phi_{\la}} $);
\smallskip
\item the families of morphisms $\{\s_\la\}$ and $\{\tau_\la\}$ cover $V$ and $W$, respectively, and (in the
analytic case) there are compact subsets $K_\la \subset V_\la$, $L_\la \subset W_\la$, for all $\la$, such that
$\bigcup \s_\la(K_\la)$ and $\bigcup \tau_\la(L_\la)$ are (compact) neighbourhoods of $a$ and $b$, respectively;
\smallskip
\item each $\Phi_\la$ is a monomial morphism;
\smallskip
\item for all $\la$ and for every $\widetilde{a} \in \sigma_\la^{-1}(a)$, the morphism
$\Psi_\la = (\Phi_\la,\psi_\la)$, where $\psi_\la:=\psi \circ \sigma_\la$, is pre-monomial at $\ta$
(with respect to the transform $F_\la$ of $F$ induced by \eqref{eq:transftorel}), 
and there exist a remainder $\widetilde{g}$ and relation $\widetilde{R}$ for $\tg$ at $\ta$, such that
$\rho_{\ta}(R)< \infty$.
\end{enumerate}
\end{proposition}

\begin{proof}
Given a commutative diagram \eqref{eq:transftorel}, where $\s_\la,\,\tau_\la$ are as in (1),
$\tPhi_\la$ is monomial and $\tPsi_\la$ pre-monomial at every $\ta\in\s_\la^{-1}(a)$, by
Corollary \ref{cor:MonCombinatorialBU} and Definition \ref{def:PreMonomial} (or Lemma
\ref{lem:NormalFormPartialMonomial}(1)); this observation will be used implicitly in the proof.

We argue by induction on $(e(a),i(a))$. If $e(a)=0$ (equivalently, $i(a)=0$), then $\Phi$ has
the form $\Phi(\bv,\bw) = \bv$ (i.e., $\Phi$ is the morphism $\bz=\bv$, in the notation of 
Definition \ref{def:PreMonomial}), so that $R(\bz,t) := t-g(\bz)$ is a relation, where $g(\bv)$ is a
remainder for $\Psi$ at $a$. By induction, therefore, we can assume the result for $(e(\cdot),i(\cdot))< (e(a),i(a))$.
We consider two cases.

\medskip\noindent
(a) $q = e(a)-i(a)>0$. By Lemma \ref{lem:TranslationTrick}, there is a commutative diagram 
\eqref{eq:mainCommutativeRelationI}, where $\s$ and $\tau$ are sequences of combinatorial
blowings-up, such that,  for all $\widetilde{a} \in \sigma^{-1}(a)$, $(e(\ta),i(\ta))<(e(a),i(a))$. Since
$\widetilde{\Psi}=(\widetilde{\Phi},\,\widetilde{\psi}=\psi \circ \sigma)$ is pre-monomial on $\s^{-1}(a)$, the result follows 
by induction and the properness of $\sigma$. 

\medskip\noindent
(b) $q = e(a)-i(a)=0$. Then there are pre-monomial coordinate systems $(\bu,\bv,\bw)$ and $(\bx,\by,\bz,t)=(\bx,\bz,t)$ 
for $\Psi$ at $a$ and $(b,0)$, where $b=\Phi(a)$ (respectively). Let $g(\bu,\bv)$ denote a remainder for $\Psi$ at $a$.
By Lemma \ref{lem:denombound}(1),
$$
G(\bu,\bv) = H(\bx^{1/d}, \bz),
$$
where $G$ is the formal expansion of $g$ at $a$, $H$ is a formal Laurent series, and $d$ is a positive integer.

Then, by Lemma \ref{lem:RelationTrick}, there is a commutative diagram 
\eqref{eq:mainCommutativeRelationI}, where $\s$ and $\tau$ are sequences of combinatorial
blowings-up, such that, for all $\widetilde{a} \in \sigma^{-1}(a)$, 
\begin{itemize}
\item[either] $(e(\ta),i(\ta))<(e(a),i(a))$, and we can conclude by induction at $\ta$, as in (a);

\smallskip
\item[or] $(e(\ta),i(\ta))=(e(a),i(a))$ and, in this case, we can assume that $H$ is a formal
power series (rather than just a Laurent series).
\end{itemize}

In the latter case, let $\varepsilon = e^{2\pi i/d}$ and set
\begin{align*}
R(\bx,\bz,t) :&= \prod_{i_1,\ldots,i_p =1}^d \left(t - H(\varepsilon^{i_1}x_1^{1/d},\ldots, \varepsilon^{i_p}x_p^{1/d},\bz)\right)\\
                    &= t^{pd} + \sum_{i=0}^{pd-1} A_i(\bx,\bz)t^i.
\end{align*}
Then the $A_i(\bx,\bz)$ and $R(\bx,\bz,t)$ are formal power series over $\IK$, and are convergent
(or algebraic) in the analytic (or algebraic) cases (see Remark \ref{rem:symm}). Clearly,
$R(\Phi,g)=0$ and $\rho_a(R)<\infty$, as required.
\end{proof}

\subsection{Control of the order of a relation}\label{subsec:relorder}

 The two lemmas of this subsection allow us to control the the order $\rho_a(R)$ of a relation $R$ after transformation 
 by blowings-up that are compatible with the log derivations $\De^\Phi$. These lemmas play important
 parts in the proof of our main theorems in Section \ref{sec:LocalMon}, in the step on transforming
 a pre-monomial to a monomial morphism by decreasing $\rho_a(R)$. These lemmas apply to
 a quasianalytic class $\cQ$, in general, and will be needed
 even in the case of a relation of the form \eqref{eq:realrel} (e.g., in the real quasianalytic case)
 when $\{t=0\}$ is a component of the divisor $F$.
 
 \begin{lemma}[Control of a relation I: blowings-up independent of $t$]\label{lem:PreMonomial-ContollRelations}
Let $\Psi = (\Phi,\psi): (M,D) \to (N\times \mathbb{K},F)$ denote a morphism of class $\cQ$ which is pre-monomial
at a point $a\in M$. Assume we have neighbourhoods $V$ of $a$ in $M$ and $W$ of $b=\Phi(a)$ in $N$,
and a finite number of commutative diagrams \eqref{eq:transftorel} satisfying the conditions:
\begin{enumerate}
\item each $\s_\la$ and $\tau_\la$ is a composite of finitely many smooth local blowings-up 
compatible with $D,\,E$, and\,
$\sigma^{\ast}_{\la}(\De^{\Phi})= \De^{\Phi_{\la}} $;
\smallskip
\item[(2),\,(3)] as in Proposition \ref{prop:PreparationPreMonomial2}.
\end{enumerate}
For each $\la$, consider the morphism $\Psi_{\la}=(\Phi_{\la},\psi_\la): (M_\la,D_{\la}) \to (N_{\la}\times \mathbb{K},F_{\la})$, where $\psi_\la := \psi \circ \sigma_{\la}$ and $F_{\la}$ is the transform of $F$ induced by \eqref{eq:transftorel}. Then, 
at every point $\ta \in \s_\la^{-1}(a)$, for each $\la$,
\begin{enumerate}
\item[(a)] $\sigma_{\la}^{\ast}(\De^{\Psi}) = \De^{\Psi_{\la}}$;
\smallskip
\item[(b)] $\Psi_{\la}$ is pre-monomial;
\smallskip
\item[(c)] if $g$ is a remainder for $\Psi$ at $a$ and $R$ is a corresponding relation,
then $g_{\la}:=g\circ \sigma_{\la}$ is a remainder for $\Psi_{\la}$ at $\ta$ and 
$R_\la := (\tau_{\la}\times \mathrm{Id})^{\ast}R$ is a relation; moreover,
$\rho_{\ta}(R_{\la}) \leq \rho_{a}(R)$. 
\end{enumerate}
\end{lemma}

\begin{proof} 
(a) follows from Lemma \ref{lem:PreservingLogAdapted} and the normal forms in Definition
\ref{def:PreMonomial} (the proof is analogous to that of Lemma \ref{lem:BUPreservingMonomialDeriv}).

\medskip\noindent
(b) Let $\cJ_1^\Psi := \De^\Phi(\psi)$ (see Definition \ref{def:InvariantPartMon}). By Lemma
\ref{lem:NormalFormPartialMonomial}(1), either $\mathcal{J}^{\Psi}_{1} = (\pmb{u}^{\pmb{\gamma}})$
or $\mathcal{J}^{\Psi}_{1} =(0)$. (We use the notation of Definition \ref{def:PreMonomial}.) Given
$\ta\in\s_\la^{-1}(a)$, for some $\la$, let $(\tbu,\tbv,\tbw)$ denote a $\Phi_\la$-monomial coordinate system
at $\ta$. By (1) and Lemma \ref{lem:PreservingMu}, either
$\mathcal{J}^{\Psi_{\la}}_{1} = \sigma^{\ast}_{\la}(\mathcal{J}^{\Psi}_{1}) =  (\widetilde{\pmb{u}}^{\widetilde{\pmb{\gamma}}})$
or $\mathcal{J}^{\Psi_{\la}}_{1}=(0)$. By Lemma \ref{lem:NormalFormPartialMonomial}(1), therefore,
$\Psi_{\la}$ is pre-monomial at $\ta$. 

\medskip\noindent
(c) First, consider the case $\mathcal{J}^{\Psi}_{1} =(0)$. Then $\psi(\bu,\bv,\bw) = g(\bu,\bv)$
(from Definition \ref{def:PreMonomial} and Lemma \ref{lem:BasisPhiDer}). Therefore,
$\psi_\la = \psi\circ\s_\la = g_\la$ and
\begin{equation}\label{eq:PreMonomial-RelationZeroComputation}
R_{\la}(\Phi_{\la},g_{\la})= R(\tau_{\la} \circ \Phi_{\la}, g_{\la})=R(\Phi\circ \sigma_{\la},g\circ \sigma_{\la}) = R(\Phi,g) \circ \sigma_{\la} = 0,
\end{equation}
so that $R_\la$ is a relation for $g_\la$. 

\medskip
Secondly, suppose that $\mathcal{J}^{\Psi}_{1} =(\pmb{u}^{\pmb{\gamma}})$. Then
$\psi \circ \sigma_{\la} = g \circ \sigma_{\la} + \widetilde{\pmb{u}}^{\widetilde{\pmb{\gamma}}} U(\widetilde{\pmb{u}},\widetilde{\pmb{v}},\widetilde{\pmb{w}})$,
where $U$ is a unit. Note that, if $g_\la$ is a remainder for $\Psi_\la$, then it follows again from
\eqref{eq:PreMonomial-RelationZeroComputation} that $R_\la$ is a relation. To show that $g_\la$ is
a remainder, we consider two cases.
\begin{enumerate}
\item[(i)] If $\widetilde{\pmb{\gamma}}$ is $\mathbb{Q}$-linearly independent of the exponents 
$\widetilde{\pmb{\alpha}}_i$ of the monomial morphism $\Phi_{\la}$, then, after 
a change of coordinates, we can assume that
$\psi \circ \sigma_{\la} = g \circ \sigma_{\la} + \widetilde{\pmb{u}}^{\widetilde{\pmb{\gamma}}}$;
therefore, $g_{\lambda}$ is a remainder.
\smallskip
\item[(ii)] If $\widetilde{\pmb{\gamma}}$ is linearly dependent on the
$\widetilde{\pmb{\alpha}}_i$, then, without loss of generality, we can assume that
\[
\psi \circ \sigma_{\la} = g \circ \sigma_{\la} + \widetilde{\pmb{u}}^{\widetilde{\pmb{\gamma}}} \left(\xi + \tw_1\widetilde{U}(\widetilde{\pmb{u}},\widetilde{\pmb{v}},\widetilde{\pmb{w}}) 
+ a_0(\widetilde{\pmb{u}},\widetilde{\pmb{v}},\widehat{\widetilde{\pmb{w}}})\right)
\]
where $\widetilde{U}$ is a unit and $\widehat{\widetilde{\pmb{w}}}$ denotes $\widetilde{\pmb{w}}$ with the first
entry $\tw_1$ removed (using Lemma \ref{lem:BasisPhiDer} and the fact that $\De^{\Phi_\la}(\psi_\la) = (\tbu^{\tbga})$).
After a change of coordinates in the source, 
\[
\psi \circ \sigma_{\la} = g \circ \sigma_{\la} + \widetilde{\pmb{u}}^{\widetilde{\pmb{\gamma}}} (\xi + w_1),
\]
so that $g_{\lambda}$ is a remainder.
\end{enumerate}

\smallskip
Finally, in either case $\mathcal{J}^{\Psi}_{1} = (0)$ or $(\pmb{u}^{\pmb{\gamma}})$, let
$\pi_{\la}: N_{\la}\times \mathbb{K} \to N_{\la}$ denote the projection. Then 
$\pi_{\la} = \pi \circ (\tau_{\la}\times \text{Id})$, so that 
$(\tau_{\la}\times\text{Id})^{\ast}(\De^{\pi}) = \De^{\pi_{\la}}$. By Lemma \ref{lem:PreservingMu},
$\rho_{\ta}(R_{\la}) \leq \rho_{a}(R)$.
\end{proof}

\begin{lemma}[Control of a relation II: reduction from extended to induced divisor]\label{lem:PreMonomial-ContollRelations2}
Let $\Psi = (\Phi,\psi): (M,D) \to (N\times \mathbb{K},F)$ denote a morphism of class $\cQ$ which is pre-monomial
at a point $a\in M$. Let $g$ denote a remainder for $\Psi$ at $a$, and $R$ a relation for $g$ such that
$0<\rho_a(R)<\infty$. (We can assume that $\psi(a)=0=g(a)$.) Assume that
\begin{equation}\label{eq:control2}
\begin{aligned}
g(\bu,\bv) &= \bu^{\bde} V(\bu.\bv),\\
R(\bx.\by,\bz,t) &= \sum_{i=0}^d \bx^{\pmb{\varepsilon}_i} \by^{\bzeta_i} t^{k_i} \tU_i(\bx.\by,\bz,t)
\end{aligned}
\end{equation}
(in the notation of Definition \ref{def:PreMonomial}), where $k_0 < k_1 < \cdots < k_d$, $V(a) \neq 0$,
and $\tU_i(b,0) \neq 0$, $i=0,\ldots,d$, where $b=\Phi(a)$. If $F$ is the extended divisor
$(E\times \mathbb{K}) \cup (N\times\{0\})$ at $(b,0)$,
then there exist open neighbourhoods
$V$ of $a$ in $M$ and $\tW$ of $(b,0)$ in $N\times\IK$, and a finite number of commutative diagrams 
\begin{equation}\label{eq:transftorel2}
\begin{tikzcd}
(V_\la,D_{\la}) \arrow{d}{\Psi_{\la}} \arrow{r}{\s_\la} & (V,D) \arrow{d}{\Psi}\\
(\tW_\la,F_{\la}) \arrow{r}{\tau_\la} & (\tW,F)
\end{tikzcd}
\end{equation}
such that 
\begin{enumerate}
\item each $\s_\la$ and $\tau_\la$ is a composite of finitely many smooth local combinatorial blowings-up 
(in particular, $\sigma^{\ast}_{\la}(\De^{\Psi})= \De^{\Psi_{\la}} $);
\smallskip
\item the families of morphisms $\{\s_\la\}$ and $\{\tau_\la\}$ cover $V$ and $\tW$, respectively, and (in the
analytic case) there are compact subsets $K_\la \subset V_\la$, $L_\la \subset \tW_\la$, for all $\la$, such that
$\bigcup \s_\la(K_\la)$ and $\bigcup \tau_\la(L_\la)$ are (compact) neighbourhoods of $a$ and $(b,0)$, respectively;
\smallskip
\item each $\Psi_\la$ is a pre-monomial morphism;
\smallskip
\item for every $\la$ and $\ta \in \s_\la^{-1}(a)$, there is a remainder $g_\la$ for $\Psi_\la$ at $\ta$, with
a relation $R_\la$, such that $\rho_\ta(R_\la) \leq \rho_a(R)$ and, if\, $\rho_\ta(R_\la) = \rho_a(R)$,
then $F_\la$ is the induced divisor $E_\la \times \IK$ at $\ta$.
\end{enumerate}
\end{lemma} 

\begin{remark}\label{rem:Nak}
Assume that $F$ is the extended divisor at $(b,0)$ (so that $\De^\pi$ is generated by $t\,\p/\p t$) and that $R$ has the form
\begin{equation}\label{eq:Nak}
R(\bx.\by,\bz,t) = \sum_{i=0}^d a_i(\bx.\by,\bz)\, t^{k_i} \tU_i(\bx.\by,\bz,t),
\end{equation}
where $k_0 < k_1 < \cdots < k_d$ and $\tU_i(b,0) \neq 0$, $i=0,\ldots,d$. Then each term
$a_i(\bx.\by,\bz) t^{k_i} \in \left(\cI_R\right)^{\De^\pi}_d$ (notation of Definitions \ref{def:ChainAndClosure}),
by Lemma \ref{lem:ComputingInvariant} applied with 
$F= R$, $f_i =a_it^{k_i}\tilde{U}_i$, $f= 0$ and $X= t\,\p/\p t$, and it follows that $\rho_a(R) \leq d$. Moreover, $R$ has
an expression of the form \eqref{eq:Nak} with $d=\rho_a(R)$, essentially by Nakayama's lemma.

We recall that, in the analytic and algebraic cases, the existence of a relation of finite order comes from
Proposition \ref{prop:PreparationPreMonomial2}; then the form \eqref{eq:Nak} is a consequence essentially of
Noetherianity. On the other hand, in the general quasianalytic case, where we use power substitutions, Remark
\ref{rem:rel}(2) provides a relation in the form  \eqref{eq:Nak}. See also \S\ref{subsec:II}.
\end{remark}

\begin{proof}[Proof of Lemma \ref{lem:PreMonomial-ContollRelations2}]
We argue by induction on $(e(a),i(a))$. If $e(a)=0$ (or $i(a)=0$), then $\Phi$ has
the form $\Phi(\bv,\bw) = \bv$, and $F$ is the induced divisor at $a$ (according to 
the normal forms in Definition \ref{def:PreMonomial}(2)). 
By induction, therefore, we can assume the result for $(e(\cdot),i(\cdot))< (e(a),i(a))$.
We consider two cases.

\medskip\noindent
(a) $q = e(a)-i(a)>0$. By Lemma \ref{lem:TranslationTrick}, there is a commutative diagram 
\eqref{eq:mainCommutativeRelationI} (with $V=M,\,, W=N$), where $\s$ and $\tau$ are sequences of combinatorial
blowings-up, such that,  for all $\widetilde{a} \in \sigma^{-1}(a)$, $\tPhi$ is monomial at $\ta$ and $(e(\ta),i(\ta))<(e(a),i(a))$.
Clearly, for all $\widetilde{a} \in \sigma^{-1}(a)$, $\tPsi = (\tPhi, \widetilde{\psi}=\psi\circ\s)$ is pre-monomial at $\ta$; moreover,
by  Lemma \ref{lem:PreMonomial-ContollRelations}, $\tg = g\circ\s$ is a remainder for $\tPsi$ at $\ta$, and 
$\tR = (\tau\times\text{Id})^*R$ is a relation with $\rho_{\ta}(\tR) \leq \rho_a(R)$. The result, therefore, follows by induction.

\medskip\noindent
(b) $q = e(a)-i(a)=0$. (Using the notation of Definition \ref{def:PreMonomial}), $\psi$ can be written in one of the
following three forms:
\begin{equation}\label{eq:PreMonomial-Eq1}
\psi = g + \varphi = \begin{cases}
\bu^{\pmb{\delta}} V(\bu,\bv)\\
 \bu^{\pmb{\delta}} V(\bu,\bv) + \pmb{u}^{\pmb{\gamma}} \\
 \bu^{\pmb{\delta}} V(\bu,\bv) + \pmb{u}^{\pmb{\gamma}} (\eta + w_1)
 \end{cases}
 \end{equation}
 where $\bga$ is linearly independent of, or linearly dependent on the $\bal_j$, in the second or third cases, respectively
 (and $\bde$ is dependent on the $\bal_j$). In the second or third cases, after a sequence of combinatorial
 blowings-up in the source, we can assume that either $\bde<\bga$ or $\bga\leq \bde$,
 respectively. Suppose that  $\bga\leq \bde$. Then, in the second case, $\bga <\bde$, and we can
 absorb the remainder into $\vp$ to get a monomial morphism. In the third case, $\psi = \bu^{\bga}(\eta + \bu^{\bde-\bga} V + w_1)$,
 and $\eta + \bu^{\bde-\bga} V$ is a unit, even if $\bga=\bde$ (otherwise, $\Psi$ is not a morphism, because the divisor $F$
 includes $\{t=0\}$); thus $\psi = \bu^{\bga}(\teta + \tw_1)$, where $\teta\neq 0$), after a change of the coordinate $w_1$,
 and $\Psi$ is again a monomial morphism. Therefore, we can assume that $\bde  < \bga$.
 
 Let $\cK$ denote the ideal sheaf $(\bx^{\pmb{\varepsilon}_i} t^{k_i})$ (there are no $\by$'s because $q=0$.
 Let $\tau: (\tN, \tF) \to (N\times \IK, F)$ be a sequence of combinatorial blowings-up such that $\tau^*\cK$
 is principal and monomial. By Lemma \ref{lem:ModificationTargetExtension}, there is a commutative diagram
\begin{equation}\label{eq:diag}
\begin{tikzcd}
(\tM,\tD) \arrow{d}{\tPsi} \arrow{r}{\s} \arrow{dr}{\Psi'} & (M,D) \arrow{d}{\Psi}\\
(\tN,\tF) \arrow{r}{\tau} & (N\times\IK,F)
\end{tikzcd}
\end{equation}
where $\s$ is a sequence of combinatorial blowings-up. Clearly, the diagonal morphism $\Psi' = (\Phi',\psi')$ 
is pre-monomial at every $\ta \in \s^{-1}(a)$, with remainder $g'=g\circ\s$ and relation $R'=R$; moreover,
$\rho_{\ta}(R') \leq \rho_a(R)$, by Lemma \ref{lem:PreMonomial-ContollRelations}.

Let $\ta \in \s^{-1}(a)$. If $(e(\ta),i(\ta))<(e(a),i(a))$, then we can finish by induction.

Therefore, we can suppose that $(e(\ta),i(\ta))=(e(a),i(a))$. Then we can write $x_j = \tbu^{\tbal_j}$, $j=1,\ldots,p$ (with
respect to pre-monomial coordinates at $\ta$, say $\tbu = (\tu_1,\ldots,\tu_{r'})$), 
where $\tbal_1,\ldots,\tbal_p$ are $\IQ$-linearly independent.

Since $\tau$ is a sequence of combinatorial blowings-up, $\tau^{-1}((b,0))$ is covered by charts 
with coordinates $(\bar{\pmb{x}},\pmb{z})$ in which
\[
x_j = \bar{\pmb{x}}^{\pmb{\eta}_j},\  j=1,\ldots, p, \quad \text{and}\quad  t = \bar{\pmb{x}}^{\pmb{\eta}_0},
\]
where the matrix with rows $\pmb{\eta}_0,\ldots,\pmb{\eta}_p$ has determinant $\pm 1$. Let $\tb = \tPsi(\ta)$,
so that $\tau( \tb ) = (b,0)$. It follows that there are coordinates $(\widetilde{\pmb{x}},\widetilde{\pmb{z}},\pmb{z})$ 
at $\widetilde{c}$ in which
\[
x_j = \widetilde{\pmb{x}}^{\pmb{\lambda}_j}(\pmb{\xi} + \widetilde{\pmb{z}})^{\pmb{\rho}_j}, \ j=1,\ldots, p,\quad
\text{and}\quad t = \widetilde{\pmb{x}}^{\pmb{\lambda}_0}(\pmb{\xi} +\widetilde{\pmb{z}})^{\pmb{\rho}_0},
\]
where every entry of $\pmb{\xi}$ is nonzero. Say $\tbx = (\tx_1,\ldots, \tx_{p'})$; thus $p'\leq p+1$. 

Let $L$ denote the $p \times p'$ matrix with rows $\bla_1,\ldots,\bla_p$. We can write 
$\widetilde{x}_j = \tbu^{\pmb{\gamma}_j} W_j$, $j=1,\ldots,p'$, where the $W_j$ are units.
Let $A$ and $C$ denote the $p\times r'$ and $p'\times r'$ matrices
with rows $\tbal_1,\ldots,\tbal_p$ and $\bga_1,\ldots,\bga_{p'}$, respectively. 
By \eqref{eq:diag}, $C\cdot L = A$. It follows that $L$ has rank $p$; therefore, $p' \geq p$ and 
$\bla_1,\ldots,\bla_p$ are linearly independent. 

\medskip\noindent
We claim that  $\bla_0$ is $\IQ$-linearly dependent on $\bla_1,\ldots,\bla_p$; this will be
proved below. Assuming the claim, then $p'=p$ and, after a coordinate
change in the target of $\tPsi$, we can write
\begin{equation}\label{eq:ControlingRelations1b}
x_j = \widetilde{\pmb{x}}^{\pmb{\lambda}_j}, \ j=1,\ldots, p,\quad
\text{and}\quad t = \widetilde{\pmb{x}}^{\pmb{\lambda}_0}(\xi_0 +\tz_0).
\end{equation}
After a coordinate change in $\tbu$ (absorbing units), we can also write $\widetilde{x}_j = \tbu^{\pmb{\gamma}_j}$, $j=1,\ldots,p$.
Therefore, $\widetilde{\Phi} := (\widetilde{x}_1,\ldots, \widetilde{x}_p)$ is a monomial morphism. Moreover,
let $\widetilde{\pi}$ denote the projection $\widetilde{\pi}(\widetilde{\pmb{x}},\widetilde{z}_0,\pmb{z})=\widetilde{z}_0$. Since
\[
\tau^{\ast}\left( t\,\frac{\p}{\p t}\right) = \left(\xi_0 + \widetilde{z}_0\right) \frac{\p}{\p \widetilde{z}_0}\,,
\]
it follows that $\tau^{\ast}(\De^{\pi}) = \De^{\widetilde{\pi}}$, and, since $\xi_0 + \widetilde{z}_0$ is unit, the divisor $\widetilde{F}$
is of induced type $\tE\times\IK$  at $\tb$. 

Write $\bu^{\bde} = \tbu^{\tbde}$, $\bu^{\bga} = \tbu^{\tbga}$ (so that $\tbde < \tbga$). Then
$A \cdot L^{-1} \cdot \pmb{\lambda_0}   = \tbde$ and, by \eqref{eq:PreMonomial-Eq1}, \eqref{eq:ControlingRelations1b},
\[
\widetilde{z_0} = \begin{cases}
 V(\tbu,\bv) -\xi_0\\
 V(\tbu,\bv)-\xi_0 + \tbu^{\tbga-\tbde} \\
  V(\tbu,\bv)-\xi_0 + \tbu^{\tbga-\tbde} (\eta + w_1)
 \end{cases}
\]
where $V(\ta) -\xi_0 = 0$ (since the coordinate systems are centred at $\ta$ and $\tb$). 
In particular, $\widetilde{g} := V(\tbu,\bv) -\xi_0$\, is a remainder for $\tPsi$ at $\ta$.
Moreover, if $\widetilde{R} := R \circ \tau$, then 
\[
\widetilde{R}(\widetilde{\Phi},\widetilde{g})= R(\Phi,g) = 0,
\]
by \eqref{eq:ControlingRelations1b}, so that $\tR$ is a relation. 
Since $\tau^*(\De^{\pi}) = \De^{\widetilde{\pi}}$, it then follows from Lemma \ref{lem:PreservingMu}(2) that
$\rho_{\ta}(\tR) \leq \rho_a(R)$, which completes the proof.

\medskip\noindent
Finally, assume that $\bla_0, \bla_1,\ldots,\bla_p$ are linearly independent; we have to show that this assumption leads to
a contradiction. Indeed, instead of \eqref{eq:ControlingRelations1b}, we have
\begin{equation}\label{eq:ControlingRelations2a}
x_j = \widetilde{\pmb{x}}^{\pmb{\lambda}_j}, \ j=1,\ldots, p,\quad
\text{and}\quad t = \widetilde{\pmb{x}}^{\pmb{\lambda}_0}.
\end{equation}
Therefore, 
$\tau^*(t^{k_i}\pmb{x}^{\pmb{\varepsilon}_i}) 
= \widetilde{\pmb{x}}^{k_i \pmb{\lambda}_0 + \sum_{j=1}^{p}\varepsilon_{ij}\pmb{\lambda}_i}$,
for all $i$, and the exponents in the right-hand sides of these expressions are all distinct 
(because the $k_i$ are distinct). Since $\tau$ principilizes the ideal $\mathcal{K}$, it follows that 
$\widetilde{R} := R \circ \tau$ is principal and monomial at $\tb$; i.e.,
\begin{equation}\label{eq:PreMonomial-2}
\widetilde{R} = \widetilde{\pmb{x}}^{\pmb{\varepsilon}} \widetilde{U}(\widetilde{\pmb{x}},\pmb{z}),
\end{equation}
where $\widetilde{U}$ is a unit.

Consider the morphism $\Psi_g:=(\Phi,g)$. We claim that there exists a morphism $\widetilde{\Psi}_g$ such that 
$\Psi_g\circ\s = \tau \circ \widetilde{\Psi}_g $ and $\widetilde{\Psi}_g (\ta)=\widetilde{b}$. Indeed,
according to \eqref{eq:ControlingRelations2a}, the morphism $\tPsi$ is given by 
$\Psi\circ\s$ followed by the monomial mapping with
rational exponents determined by the inverse of the $(p+1) \times (p+1)$ matrix with
rows $\bla_j$, $j=0,\ldots,p$. Since the last component of $\Psi$ is
                $ t = \bu^{\bde} (V + \bu^{\bga-\bde} X)$, for
some expression $X$ (from \eqref{eq:PreMonomial-Eq1}), and $\bde < \bga$, this morphism will still be well-defined if we
replace $X$ by zero.

Therefore, $\widetilde{R}(\widetilde{\Psi}_g) = R(\Psi_g) = 0$, so that, by \eqref{eq:PreMonomial-2},
$\widetilde{\Psi}_g^{\ast}(\widetilde{x}_j) = 0$, for some $j=1,\ldots,p+1$. It follows from
\eqref{eq:ControlingRelations2a} that
\begin{itemize}
\item[either] $\Psi_g^{\ast}(x_j) = 0$, for some $j$, which means that $\Phi$ is not dominant,
\smallskip
\item[or] $\Psi_g^{\ast}(t) = 0$, which means that $g=0$.
\end{itemize}
In either case, we have a contradiction.
\end{proof}

\section{Proof of monomialization}\label{sec:LocalMon}

\subsection{The inductive scheme}\label{subsec:indshceme}
Let $\Phi:(M,D) \to (N,E)$ denote a dominant monomial morphism, of quasianalytic class $\cQ$.
Set $m=\dim M$ and $n=\dim N$. In this section, we will prove our main 
results---Theorems \ref{thm:mainR2}, \ref{thm:mainA2} on monomialization, and Theorem
\ref{thm:ideal} on desingularization of an ideal relative to a monomial 
morphism---within a common inductive scheme. Induction is with respect to pairs $(m,n)$,
ordered lexicographically. Recall that $\De^\Phi$ denotes the subsheaf of
$\cD_D = \Der_M(-\log D)$ of logarithmic derivatives tangent to $\Phi$ (Definition
\ref{def:tang}).

We consider the following list of assertions or claims. (The full inductive scheme below
will be needed for general quasianalytic classes, but can be considerably simplified in the
analytic or algebraic cases; see Remarks \ref{rem:Noeth1}, \ref{rem:Noeth2}.)

\medskip
\noindent
I.A.a($m,n$). Let $\mathcal{I}$ denote a privileged ideal (sheaf) on $M$, and let $a \in M$. 
There exist open neighbourhoods
$V$ of $a$ in $M$ and $W$ of $b=\Phi(a)$ in $N$, and a finite number of commutative diagrams
\begin{equation}\label{eq:mainCommutative}
\begin{tikzcd}
(V_\la,D_{\la}) \arrow{d}{\Phi_{\la}} \arrow{r}{\s_\la} & (V,D) \arrow{d}{\Phi}\\
(W_\la,E_{\la}) \arrow{r}{\tau_\la} & (W,E)
\end{tikzcd}
\end{equation}
where
\begin{enumerate}
\item each $\s_\la$ and $\tau_\la$ is a composite of finitely many smooth local blowings-up 
(and local power substitutions);
\smallskip
\item the families of morphisms $\{\s_\la\}$ and $\{\tau_\la\}$ cover $V$ and $W$, 
(and there are compact subsets $K_\la \subset V_\la$, $L_\la \subset W_\la$, for each $\la$, 
such that
$\bigcup \s_\la(K_\la)$ and $\bigcup \tau_\la(L_\la)$ are (compact) neighbourhoods of $a$ and $b$,
respectively);
\smallskip
\item each $\Phi_\la$ is a monomial morphism;
\smallskip
\item the pull-back $\sigma^{\ast}_{\la}(\De^{\Phi}) = \De^{\Phi_{\la}}$;
\smallskip
\item for every $\widetilde{a} \in V_{\la}$, either
\smallskip
\begin{enumerate}
\item[(a)] the ideal $\widetilde{\mathcal{I}} := \sigma^{\ast}_{\la}(\mathcal{I})$ is principal and
monomial at $\widetilde{a}$,
\smallskip
\item[or (b)] there exists a regular vector field (germ) $X \in \De^{\Phi_{\la}}_{\ta}$.
\end{enumerate}
\end{enumerate}

\begin{remark}\label{rem:indcases}
In condition (1) above, ``(and local power substitutions)'' should be understood to mean
that local power substitutions are needed \emph{except} in the analytic or algebraic cases,
where only local blowings-up are needed. Similarly, in (2), the phrase in brackets does not apply
to the algebraic case.

Power substitutions are needed in the real quasianalytic case only for Proposition \ref{prop:cleaning}
(see also Remark \ref{rem:cleaning}). Proposition \ref{prop:cleaning} is used in the inductive step
in this section (see \S\ref{subsec:II}), but power substitutions are not otherwise involved in 
the arguments below.
\end{remark}

\smallskip
\noindent
I.A.b($m,n$). Let $\mathcal{I}$ be a privileged ideal which is $\De^{\Phi}$-closed, i.e.,
$\De^{\Phi}(\mathcal{I})\subset \mathcal{I}$, and let $a \in M$. There exist open neighbourhoods
$V$ of $a$ and $W$ of $b=\Phi(a)$, and a finite number of commutative diagrams \eqref{eq:mainCommutative}, 
satisfying conditions (1)--(4) above, together with

\begin{enumerate}
\item[(5)] for all $\widetilde{a} \in V_{\la}$, the ideal $\widetilde{\mathcal{I}} := \sigma^{\ast}_{\la}(\mathcal{I})$ is principal and monomial at $\widetilde{a}$.
\end{enumerate}

\medskip
\noindent
I.A.c($m,n$). Let $\mathcal{I}$ be a privileged ideal and let $a \in M$. There exist open neighbourhoods
$V$ of $a$ and $W$ of $b=\Phi(a)$, and a finite number of commutative diagrams \eqref{eq:mainCommutative}, satisfying conditions (1)--(3) above, together with

\begin{enumerate}
\item[(4)] for all $\widetilde{a} \in V_{\la}$, the ideal $\widetilde{\mathcal{I}} = \sigma^{\ast}_{\la}(\mathcal{I})$ is principal and monomial at $\widetilde{a}$.
\end{enumerate}

\medskip
\noindent
I.B.a($m,n$). Let $\psi : M\to \mathbb{K}$ denote a function of class $\cQ$. 
Assume that $\Psi = (\Phi,\psi):(M,D) \to (N\times \mathbb{K},F)$ is a well-defined morphism, where
$F$ denotes the divisor given either by $E \times \mathbb{K}$ or by $(E \times \mathbb{K})\cup (N\times \{0\})$ (see Definition \ref{def:PreMonomial} and Remark \ref{rem:premonom2}).
Let $a\in M$. There exist open neighbourhoods
$V$ of $a$ and $W$ of $b=\Phi(a)$, and a finite number of commutative diagrams \eqref{eq:mainCommutative}, satisfying conditions (1)--(4) of I.A.a(m,n) above, together with

\begin{enumerate}
\item[(5)] for every $\widetilde{a} \in V_{\la}$, either
\begin{enumerate}
\item the transformed morphism $\Psi_{\la} = (\Phi_{\la}, \psi_{\la})=(\Phi_{\la},\psi \circ \sigma)$ is pre-monomial at $\widetilde{a}$,
\smallskip
\item[or (b)] there exists a regular vector field $X \in \De^{\Phi_{\la}}_{\ta}$.
\end{enumerate}
\end{enumerate}

\medskip
\noindent
I.B.b($m,n$). Let $\psi : M\to \mathbb{K}$ be a function of class $\cQ$. 
Assume that $\Psi = (\Phi,\psi):(M,D) \to (N\times \mathbb{K},F)$ is a well-defined morphism, where
$F$ denotes the divisor given either by $E \times \mathbb{K}$ or by $(E \times \mathbb{K})\cup (N\times \{0\})$.
Let $a\in M$. There exist open neighbourhoods
$V$ of $a$ and $W$ of $b=\Phi(a)$, and a finite number of commutative diagrams \eqref{eq:mainCommutative}, satisfying conditions (1)--(3) of I.A.a(m,n), together with

\begin{enumerate}
\item[(4)] for every $\widetilde{a} \in V_{\la}$ the morphism 
$\Psi_{\la} = (\Phi_{\la}, \psi_{\la})=(\Phi_{\la},\psi \circ \sigma)$ is pre-monomial at $\widetilde{a}$.
\end{enumerate}

\medskip
\noindent
II($m,n$). Let $a\in M$ and let $\psi: M \to \IK$ be a function of class $\cQ$; say $\psi(a)=0$. Assume that 
$\Psi = (\Phi,\psi):(M,D) \to (N\times \mathbb{K},F)$ is a well-defined 
morphism that is \emph{pre-monomial} at $a$, where
$F$ denotes either the induced divisor
$E \times \mathbb{K}$ or the extended divisor $(E \times \mathbb{K})\cup (N\times \{0\})$.
There exist open neighbourhoods
$V$ of $a$ and $W$ of $b=\Phi(a)$, and a finite number of commutative diagrams \begin{equation}\label{eq:mainCommutative2}
\begin{tikzcd}
(V_\la,D_{\la}) \arrow{d}{\Psi_{\la}} \arrow{r}{\s_\la} & (V,D) \arrow{d}{\Psi} \\
(W_\la,F_{\la}) \arrow{r}{\tau_\la} & (W,F) 
\end{tikzcd}
\end{equation}
satisfying conditions (1) and (2) of I.A.a(m,n), together with

\begin{enumerate}
\item[(3)] each $\Psi_\la$ is a monomial morphism;
\smallskip
\item[(4)] the pull-back $\sigma^{\ast}_{\la}(\De^{\Psi}) = \De^{\Psi_{\la}}$.
\end{enumerate}

\begin{remark}\label{rem:short}
We will use the following shorthand when referring to the assertions above:
I.A($m,n$) means I.A.a($m,n$), I.A.b($m,n$) and I.A.c($m,n$) (likewise, I.B($m,n$));
I($m,n$) means I.A($m,n$) and I.B($m,n$); I.A means I.A($m,n$), for all ($m,n$); etc.
\end{remark}

Before proving all the claims above, we show that our main theorems follow. 
Theorems \ref{thm:ideal} and \ref{thm:corideal}, in fact, are clearly equivalent to I.A.b
and I.A.c (respectively).

\begin{proof}[Proof of Theorems \ref{thm:mainR2}, \ref{thm:mainA2}]
Let $\Phi : (M,D) \to (N,E)$ denote a morphism, and let $a \in M$. Consider a local
coordinate system $\pmb{x} = (x_1,\ldots,x_n)$ for $N$ at $b =\Phi(a)$, such that 
$E=\{x_1 \cdots x_l =0\}$, for some $l\leq n$, and write $\Phi = (\Phi_1,\ldots, \Phi_n)$
with respect to these coordinates. For each $k=1,\ldots,n$, we consider the morphism 
$\Phi^{k} = (\Phi_1,\ldots, \Phi_k): (M,D) \to (N^{k},E^{k})$, defined in a neighbourhood of $a$,
where $N^k := \{x_{k+1}=\cdots =x_n=0\}$ and $E^k := \{x_1 \cdots x_{l'} =0\}$, $l'= \min\{k,l\}$. 
(In the algebraic case, this makes sense as usual with respect to
an \'etale neighbourhood of $b$, perhaps including extension of the base field.)

We will prove the following claim by induction on $k$. (Theorems \ref{thm:mainR2}, \ref{thm:mainA2} correspond to the case $k=n$.)

\medskip
\noindent
\emph{Claim.} For fixed $k$, $1\leq k\leq n$, there exist open neighbourhoods $V$ of $a$ in $M$ and 
$W^k$ of $b:=\Phi^{k}(a)$ in $N^k$, and a finite number of commutative diagrams
\begin{equation}\label{eq:mainCommutative3}
\begin{tikzcd}
(V_\la,D_{\la}) \arrow{d}{\Phi_{\la}^{k}} \arrow{r}{\s_\la} & (V,D) \arrow{d}{\Phi^{k}}\\
(W_\la^k,E_{\la}^k) \arrow{r}{\tau_\la} & (W^k,E^k)
\end{tikzcd}
\end{equation}
where

\begin{enumerate}
\item each $\s_\la$ and $\tau_\la$ is a composite of finitely many smooth local blowings-up (and local power substitutions);
\smallskip
\item the families of morphisms $\{\s_\la\}$ and $\{\tau_\la\}$ cover $V$ and $W^k$, 
(and there are compact subsets $K_\la \subset V_\la$, $L_\la \subset W^k_\la$, for all $\la$, such that
$\bigcup \s_\la(K_\la)$ and $\bigcup \tau_\la(L_\la)$ are (compact) neighbourhoods of $a$ and $b$,
respectively);
\smallskip
\item each $\Phi^k_\la$ is a monomial morphism.
\end{enumerate}

\smallskip
Assume that the claim is true for $0\leq k'<k$. Then we can assume (after transformation by 
a finite family of commutative diagrams of the form \eqref{eq:mainCommutative3}), that 
$\Phi^{k}=(\Theta^{r},0\ldots,0,\Phi_k)$, where $\Theta^r$ is a dominant monomial morphism 
of rank $r \leq k-1$. 
Consider the morphism $\widetilde{\Phi}^{k} := (\Theta^r,\Phi_k)$. By the assertions 
I.B.b($m,r$) and II($m,r$), there is a finite family of commutative diagrams of the form
\eqref{eq:mainCommutative3}, transforming $\widetilde{\Phi}^{k}$ to a monomial morphism.
Therefore, $\Phi^{k}$ is transformed to a monomial morphism by a finite family of commutative 
diagrams \eqref{eq:mainCommutative3}.
\end{proof}

\noindent
{\bf The inductive steps.} \emph{Given $(m,n)$, we will prove each of the assertions
listed above under the assumptions that all the assertions hold for $(m',n') < (m,n)$,
and the preceding assertions in the list hold for $(m,n)$.}

\medskip
To begin the induction, it is therefore enough to prove I.A.a($0,0$), which is trivial.
It follows that all of the assertions in the list are true, for all $(m,n)$.

\begin{remarks}\label{rem:indscheme}
(1) In fact, if $m=0$, then all the assertions in the list above are trivial, and, if $n=0$,
then all the assertions follow essentially from resolution of singularities---I.A($m,0$) from
desingularization of the ideal $\cI$ (Remark \ref{rk:ToroidalHullResolution}(1) is enough
for I.A.a($m,0$) and I.A.b($m,0$)), and  I.B($m,0$) from desingularization
of the ideal generated by $\psi$ (again Remark \ref{rk:ToroidalHullResolution}(1) is enough
for I.B.a($m,0$)); II($m,0$) follows from the fact that, if $n=0$, then every pre-monomial morphism is monomial (see Remarks \ref{rem:premonom1}).

\medskip\noindent
(2) In the proofs of each of our assertions, in the following subsections, 
it should be understood that we are making the inductive
assumptions above---we will not repeat the assumptions.

We will begin with the proof of I.A.a($m,n$), which is the most delicate.
(The proof of I.A.c($m,n$), for example, is simpler and follows the same 
essential steps.)
\end{remarks}

\begin{remark}\label{rem:Noeth1}
In the analytic or algebraic cases, we can prove I.A.b($m,n$) directly, assuming I.A.b($m',n'$) and II($m',n'$),
for $(m',n') < (m,n)$, by induction. The proof of II($m,n$) below already requires only I.A.b($m,n$) and, moreover,
I.A.c is obtained as a corollary of I.A.b, again by induction. In particular, in these cases,
the proofs of Theorems \ref{thm:ideal} and \ref{thm:corideal} 
do not depend on I.A.a and I.B; we can begin by proving I.A.b($m,n$) in a way that
mimics the proof of I.A.a($m,n$) below, but is simplified using Remark \ref{rem:BasisFormalIdeal}.
A brief idea will be given in Remark \ref{rem:Noeth2}.
\end{remark} 

\subsection{Desingularization relative to a
monomial morphism: assertions I.A}\label{subsec:I.A}

\begin{proof}[Proof of I.A.a($m,n$)]
We can assume that there is no $\Phi$-free variable in any $\Phi$-monomial coordinate system (otherwise, we would be done
since condition (5)(b) would be satisfied; see Lemma \ref{lem:BasisPhiDer} and Remark \ref{rk:FreeCoordinates}). 

We can monomialize the toroidal hull of $\mathcal{I}$ by combinatorial blowings-up (see 
Example \ref{ex:ToroidalHull} and Remarks
\ref{rk:ToroidalHullResolution}); then the transform of $\cI$ is principal and monomial at any $m$-point (a point lying in $m$
components of $D$) over $a$, and condition (4) is satisfied at such a point, by Corollary 
\ref{cor:DerCombinatorialBU}.
Therefore, we can assume that $a$ is not an $m$-point; i.e., there is a $\Phi$-monomial coordinate system 
$(u_1,\ldots, u_r,v_{r+1},\ldots, v_m)$ at $a$ with $r<m$. In particular, we can express the monomial morphism $\Phi$ 
as $\Phi = (\varphi,\psi)$ at $a$, with
\begin{equation}\label{eq:varphi0}
\text{either}\quad \psi(\pmb{u},\pmb{v},w) = w\quad \text{or}\quad \psi(\pmb{u},\pmb{v},w) = \pmb{u}^{\pmb{\beta}}(\xi + w),
\end{equation}
where $w$ is a $\varphi$-free coordinate and, in the second case, $\xi \neq 0$ and $\pmb{\beta}$ is $\mathbb{Q}$-linearly dependent on the exponents $\{\pmb{\alpha}_j\}$ of the monomial morphism $\varphi$ (using the notation of \eqref{eq:mon1});
here the variables $(v_{r+1},\ldots, v_m)$ above are written $(\bv,w)$.

The log derivations $\De^{\varphi}$, $\De^{\Phi}$ tangent to the morphisms $\vp$, $\Phi$
have monomial bases at $a$ given by Lemma \ref{lem:BasisPhiDer}:
\begin{equation}\label{eq:BasisDer}
\begin{aligned}
\De^{\varphi} &= \left(X^1,\ldots, X^{m-n},\,Y\right),\, \text{ where } X^i(\psi) = 0  \text{ and } Y = \frac{\partial}{\partial w},\\
\De^{\Phi} &= \left(X^1,\ldots, X^{m-n}\right).
\end{aligned}
\end{equation}
Lemma \ref{lem:PreservingLogAdapted} allows us to control the pull-back of $\De^{\Phi}$ (in order to get condition (4)), 
provided that we control the pull-back of $\De^{\varphi}$; this is used implicitly throughout the proof.

Consider the closure $\mathcal{I}_{\infty}^{\varphi}$ of $\mathcal{I}$ by $\De^{\varphi}$ (see Definition \ref{def:InvariantPairs}). 
By definition, $Y(\mathcal{I}_{\infty}^{\varphi}) \subset \mathcal{I}_{\infty}^{\varphi}$. Let 
$\mathcal{K} := \mathcal{I}_{\infty}^{\varphi}|_{w=0}$ denote the restriction of $\mathcal{I}_{\infty}^{\varphi}$ to $\{w=0\}$. 
Then $\mathcal{K}$ and $\mathcal{I}_{\infty}^{\varphi}$ induce the same ideal in the ring of formal power series in 
$(\pmb{u},\pmb{v},w)$ at $a$, by Lemma \ref{rk:BasisFormalIdeal}.
Since $\vp$ and $\mathcal{K}$ are both independent of $w$, we can apply I.A.b($m-1,n-1$) to $\varphi$ and $\mathcal{K}$
to reduce to the case that $\mathcal{K}$ is principal and monomial. We can also assume that \eqref{eq:varphi0} is
preserved; i.e., $\Phi$ is a monomial morphism. Therefore, $\mathcal{I}_{\infty}^{\varphi}$ is principal and monomial, 
by the division and quasianalyticity axioms of Definition \ref{def:quasian}.

By Lemma \ref{lem:NormalFormIdeal}, since $\mathcal{I}_{\infty}^{\varphi}$ is principal and monomial,
$$
d:= \nu_a(\cI,\vp) = \mu_{a}(\mathcal{I},\De^{\vp}) < \infty;
$$
moreover, $d=0$ if and only if $\mathcal{I}$ is principal monomial---in this case, we are done.

We now argue by induction on $d$.
We can assume that $w$ is the only $\varphi$-free variable since, otherwise, $\Phi$ would have a free variable and 
condition (5)(b) would be satisfied. 

As in Lemma \ref{lem:NormalFormIdeal}, $\mathcal{I}$ has a system of generators $F_\iota$, $\iota \in I$, of the form
\begin{equation}\label{eq:wprep}
F_\iota(\pmb{u},\pmb{v},w) = \pmb{u}^{\pmb{\gamma}}\left( \widetilde{F}_{\iota}(\pmb{u},\pmb{v},w) w^d + \sum_{j=0}^{d-1} f_{\iota j}(\pmb{u},\pmb{v})w^j  \right)
\end{equation}
where $\mathcal{I}_{\infty}^{\varphi} = (\pmb{u}^{\pmb{\gamma}})$, $\widetilde{F}_{\iota_d}(0)\neq 0$ for some $\io = \iota_d\in I$,
and $f_{\iota j}(a) =0$ for all $\iota \in I$, $0\leq j\leq d-1$. 
By the implicit function theorem (Definition \ref{def:quasian}, axiom (2)) applied to the equation
$\p^{d-1}F_{\io_d}(\bu,\bv,w)/\p w^{d-1} = 0$,
we can make a change of coordinates to reduce to the case
that $f_{\iota_d,d-1}=0$ (i.e., $F_{\iota_d}$ has Tschirnhausen form). 

But $\psi$ is now of the form
$\psi(\pmb{u},\pmb{v},w) = \pmb{u}^{\pmb{\beta}}(\xi + w + h(\pmb{u},\pmb{v}))$, where $h(a)=0$
(and we admit the possibility that $\bbe=\bzero$ and $\xi=0$ to cover both cases of \eqref{eq:varphi0}).
Let us assume, more generally, that
\begin{equation}\label{eq:varphi2}
\psi(\pmb{u},\pmb{v},w) = g(\bu,\bv) + \pmb{u}^{\pmb{\beta}}(\xi + w + h(\pmb{u},\pmb{v})),
\end{equation}
where $dg \wedge d\vp = 0$ (see Definition \ref{def:PreMonomial}) because we will apply I.B.a($m-1,n$) in the argument below,
and will therefore need to use this weaker assumption in 
the inductive step.

Consider the privileged ideal sheaves $\cK_j := (f_{\io j})_{\io \in I}$, $j=0,\ldots,d-1$. After applying
I.A.a($m-1,n-1$) to the monomial morphism $\vp$ and the ideals $\cK_j$, we can assume that either 
$\cK_j=(0)$ or $\cK_j$ is principal and monomial, $j=0,\ldots,d-1$; i.e.,
\begin{equation}\label{eq:BasisIdeal}
F_\iota(\pmb{u},\pmb{v},w) = \pmb{u}^{\pmb{\gamma}}\left( \widetilde{F}_\iota(\pmb{u},\pmb{v},w) w^d + \sum_{j=0}^{d-1} 
a_{\iota j}(\pmb{u},\pmb{v})\pmb{u}^{\br_j}w^j  \right), \quad \iota\in I,
\end{equation}
where 
$a_{\io_d,d-1}=0$ and, for each $j=0,\ldots,d-1$, either $a_{\io j} =0$ for all $\io$, or
there exists $\io(j)\in I$ such that $a_{\io(j) j}(0)\neq0$. Note that the blowings-up involved in this process have centres defined
by an ideal with generators depending only on $(\pmb{u},\pmb{v})$. It follows that, after a coordinate change in order to absorb a unit into $w$, \eqref{eq:varphi2} remains of the same form.

We now transform $\Phi$ to a pre-monomial morphism without changing the coordinate $w$:
Applying I.B.a($m-1,n$) to the morphism $\Phi|_{w=0} = (\varphi,\psi)|_{w=0}$, we can assume that
\begin{enumerate}
\item[either]there exists a $\Phi$-free variable (not $w$), so that condition (5)(b)) is satisfied; then
we can apply II($m,n-1$) to transform $\Phi$ to a monomial morphism satisfying conditions (4) and (5)(b)
as a consequence of II($m,n-1$)(4);
\smallskip
\item[or]$w$ is the only $\varphi$-free variable; in this case,
 since all previous blowings-up preserve the $w$-axis and $\psi = \psi|_{w=0} + w$, it follows from \eqref{eq:varphi2} that $\psi$ has the form
\begin{equation}\label{eq:varphi3}
\psi(\pmb{u},\pmb{v},w) =  g(\pmb{u},\pmb{v}) +\pmb{u}^{\pmb{\beta}}(\xi + w + \eta\,\pmb{u}^{\pmb{\lambda}} ),
\end{equation}
where $dg \wedge d\vp = 0$ and $\eta$ is a constant; moreover, we can assume that, if $\eta\neq 0$, then
$\bla$ is independent of the $\bal_j$ (otherwise, we can absorb $\eta\,\bu^{\bbe+\bla}$ into $g$).
\end{enumerate}

We have to treat the latter case.
Write $\widetilde{w} := w+\eta\,\pmb{u}^{\pmb{\lambda}}$. Then $\Phi=(\varphi,\psi)$ is pre-monomial in the 
coordinates $(\pmb{u},\pmb{v},\widetilde{w})$; i.e., in these coordinates, 
$\De^{\varphi}$, $\De^{\Phi}$ have monomial bases (given by Lemma \ref{lem:BasisPhiDer}), as follows:
\begin{equation}\label{eq:BasisDer3}
\begin{aligned}
\De^{\varphi} &= \left(Z^1,\ldots, Z^{m-n},\,Z\right), &&\text{ where } Z^i(\widetilde{w}) = 0  \text{ and } 
Z = \frac{\partial}{\partial \widetilde{w}},\\
\De^{\Phi}  &= \left(W^1,\ldots, W^{m-n}\right), &&\text{ where } W^i=Z^i
\end{aligned}
\end{equation}
(note that $Z = \partial/\partial \widetilde{w} = \partial/\partial w$). Consider the auxiliary sheaf of derivations
 $$
 \De:= \De^{\varphi} \cap \Der(-\log(w=0)). 
 $$
 By Lemma \ref{lem:ComputingInvariant} applied to $(F_{\iota_d})$ (i.e., applied with 
 $F=F_{\iota_d}$, $f_d=\pmb{u}^{\pmb{\gamma}} w^d \widetilde{F}_{\iota_d}$, $f_k =a_{\iota_d k}\pmb{u}^{\pmb{\gamma}+\br_k}w^k$ (for
 $k=0,\ldots,d-1$), $f=0$ and $X = w\,\partial/\partial w$), we have $\bu^{\bga}w^d \in \cI^{\De}_{d-1}$. Then, by Lemma \ref{lem:ComputingInvariant} applied to $(F_{\iota(j)})$ and $f=\bu^{\bga}w^d$ (i.e., applied with $F=F_{\iota(j)}$, 
 $f_k =a_{\iota(j) k}\pmb{u}^{\pmb{\gamma}+\br_k}w^k$ (for $k=0,\ldots,d-1$), $f=\bu^{\bga}w^d$ and $X = w\,\partial/\partial w$), we get $\bu^{\bga}\bu^{\br_j}w^j \in \cI^{\De}_{d-1}$, for all $j=0,\ldots,d-1$ for which not all $a_{\iota j}$ are zero. Therefore, by \eqref{eq:BasisIdeal} and Lemma \ref{lem:BasicPropClosure}(1),
\[
\mathcal{I}_{\infty}^{\De} = \mathcal{I}_{d-1}^{\De} = \pmb{u}^{\pmb{\gamma}}\left(w^d, \pmb{u}^{\br_j}w^j\right)
\]
(where we use the preceding condensed expression to denote the terms $\pmb{u}^{\br_{j}}w^{j}$, with $j=0,\ldots,d-2$, for which not all $a_{\iota j}$ are zero), and $\mu_a(\mathcal{I},\De) \leq d-1$.
 
We now proceed to ``partially" principalize $\mathcal{I}_{\infty}^{\De}$ in a way that preserves the property that
the morphism $\Phi$ is pre-monomial with respect to the coordinates $(\pmb{u},\pmb{v},\widetilde{w})$ (more precisely,
after suitable modifications, we will be able to assume that either $\mathcal{I}_{\infty}^{\De}$ is principal monomial or there exists a regular vector-field in $\Delta^{\Phi}$).

Let $\cK$ denote the closure of $\mathcal{I}_{\infty}^{\De}$ by $\De^{\Phi}$. Now, either $\eta =0$ (and $w=\tilde{w}$), or there exists a vector-field $X = W^j$ in $\De^{\Phi}$, for some $j$, such that $X(\pmb{u}^{\pmb{\lambda}}) =  C\pmb{u}^{\pmb{\lambda}}$, where $C$ is a non-zero constant. Recalling that $w = \tilde{w} -  \eta\pmb{u}^{\pmb{\lambda}}$ and applying Lemma \ref{lem:ComputingInvariant} again 
(more precisely, for each $j=1,\ldots,d$, we apply Lemma \ref{lem:ComputingInvariant} with $F= w^j=(\tilde{w} - \eta\pmb{u}^{\pmb{\lambda}})^j$, $f_k= (-1)^k\binom{j}{k}  \eta^k\pmb{u}^{k\pmb{\lambda}}w^{j-k}$ (for $k=0,\ldots,j$), $f= 0$ 
and $X$ as above), we obtain

\begin{equation}
\label{eq:CombinatorialK}
\mathcal{K} = \pmb{u}^{\pmb{\gamma}}\left( \pmb{u}^{\br_j} \widetilde{w}^i(\pmb{u}^{\pmb{\lambda}})^{j-i}  \right) 
= \pmb{u}^{\pmb{\gamma}}\left(\pmb{u}^{\br_j} w^i(\pmb{u}^{\pmb{\lambda}})^{j-i}  \right)
\end{equation}
(where $i=0,\ldots,j$, and either $0 \leq j \leq d-1$ and not all $a_{\iota j}$ are zero, or $j=d$ and $r_d=0$) and we conclude that $\mathcal{K}$ is closed by $\De^{\Phi}$ and $\De$.

We claim that $\mathcal{K}$ can be principalized by blowings-up that are combinatorial simultaneously 
with respect to the coordinate systems $(\pmb{u},w)$ and $(\pmb{u},\widetilde{w})$. 
To see this, we argue by induction on $|\pmb{\lambda}|$, allowing $\br_d \neq 0$ in \eqref{eq:CombinatorialK}.
Consider a first blowing-up with centre $\{w = u_k = 0\}$ (equivalently, $\{\widetilde{w} = u_k=0\}$),
for some $k$ such that $\lambda_k \neq 0$. 
Outside the strict transform of $\{w=0\}$ or of $\{\widetilde{w}=0\}$, the result follows from the fact that the pullback of $w$ (or 
of $\widetilde{w}$), equals $u$ times a unit, which implies that the pull back of $\mathcal{K}$ is generated by monomials in $\pmb{u}$.
Note that this argument is enough to treat the case $|\pmb{\lambda}|=1$ (because, in this case, 
after the first blowing-up, the strict transforms of $\{w=0\}$ and $\{\widetilde{w}=0 \}$ are disjoint). So we can suppose $|\pmb{\lambda}|>1$. Therefore, we only have to consider points in the strict transform 
of $\{w=0\}$, i.e., in the $u_k$-chart. After pulling back and factoring by $u_k$, $\mathcal{K}$ again satisfies \eqref{eq:CombinatorialK} with a new exponent $\widetilde{\pmb{\lambda}} = \pmb{\lambda} - \pmb{e}_k$,
(where $\pmb{e}_k$ means the exponent with $1$ in the $k$th place and $0$ elsewhere), so we can conclude by induction.

Principalizing $\cK$ in this way, since the blowings-up are combinatorial with respect to $(\pmb{u},\widetilde{w})$, 
it follows from Corollary \ref{cor:DerCombinatorialBU} that $\varphi$ is a monomial morphism, 
$\Phi$ is pre-monomial, and that $\De^{\Phi}$ satisfies the required pullback property. We note that, moreover, 
$\De^{\varphi}$ equals the pull-back of $\De$ (rather than of the original $\De^{\varphi}$, since we blow up with centres 
in $\{w=0\}$). Now, one of the following three possibilities holds.

\begin{enumerate}
\item[(i)] There exists a regular vector field in $\De^{\Phi}$, and therefore there exists a $\Phi$-free variable. In this case,
we can finish by applying II($m, n-1$) to $\Phi$, to obtain the conditions (3), (4) and (5)(b).
\smallskip
\item[(ii)] There is no regular vector field in $\De^{\Phi}$, but there is in $\De^{\varphi}$. Then, by 
Lemma \ref{lem:BasicPropClosure}(3), $\mathcal{I}_{\infty}^{\varphi}$ is principal and monomial, and, 
by Lemma \ref{lem:PreservingMu}, $\nu_{a}(\mathcal{I},\varphi) \leq d-1$. Since $\Phi$ is pre-monomial, we can finish 
by induction on $d$. 
\smallskip
\item[(iii)] There is no regular vector field either in $\De^{\Phi}$ or in $\De = \De^{\varphi}$. Then $\mathcal{I}$ is principal and 
monomial, by Lemma \ref{lem:BasicPropClosure}(3), and we can finish by applying II($m,n-1$) to $\Phi$.
\end{enumerate}
\vspace{-\baselineskip}
\end{proof}

\begin{proof}[Proof of I.A.b($m,n$)]
After applying I.A.a($m,n$) to the monomial morphism $\Phi$ and the ideal $\mathcal{I}$, 
we can assume that either $\mathcal{I}$ is principal and monomial at $a$ (so we are done),
or there exists a regular vector field $X$ in $\De^{\Phi}$ at $a$. 

In the latter case,
by Remark \ref{rk:FreeCoordinates}, there exists a $\Phi$-monomial coordinate system $(\pmb{u},\pmb{v},\pmb{w})$ at $a$ with at least one $\Phi$-free coordinate (i.e., where $\pmb{w}$ has at least one
component). Consider the ideal sheaf $\mathcal{K}:= \mathcal{I}|_{\pmb{w}=0}$;
by hypothesis, $\mathcal{K}$ is $\De^{\Phi}$-closed. Moreover, by Lemma \ref{rk:BasisFormalIdeal}, 
$\mathcal{K}$ and $\mathcal{I}_{\infty}^{\Phi}$ induce the same ideal in the ring of formal power
series in $(\pmb{u},\pmb{v},\pmb{w})$ at $a$. Since $\mathcal{K}$ is independent of the variables 
$\pmb{w}$, we can apply I.A.b($m-1,n$) to $\Phi|_{\pmb{w}=0}$ and $\cK$, to reduce to the case
that $\mathcal{K}$ is principal and monomial. By the division and quasianalyticity axioms, we 
conclude that $\mathcal{I}$ is principal and monomial.
\end{proof}

\begin{proof}[Proof of I.A.c($m,n$)]
Consider the closure $\mathcal{I}_{\infty}^{\Phi}$ of $\mathcal{I}$ by $\De^{\Phi}$ (see Definition \ref{def:ChainAndClosure}); by definition, 
$\De^{\Phi}(\mathcal{I}_{\infty}^{\Phi})\subset \mathcal{I}_{\infty}^{\Phi}$. After applying I.A.b($m,n$)
to the monomial morphism $\Phi$ and the ideal $\mathcal{I}_{\infty}^{\Phi}$, we can assume that
$\mathcal{I}_{\infty}^{\Phi}$ is principal and monomial at $a$. By Lemma \ref{lem:NormalFormIdeal}, 
$d:=\nu_{a}(\mathcal{I},\Phi)<\infty$; moreover, $d=0$ if and only if $\mathcal{I}$ is principal and
monomial at $a$. 

We now argue by induction on $d$ (where the induction hypothesis requires that 
$\mathcal{I}_{\infty}^{\Phi}$ be principal and monomial).
By Lemma \ref{lem:NormalFormIdeal}, there is a $\Phi$-monomial coordinate system 
$(\pmb{u},\pmb{v},\pmb{w})=(\pmb{u},\pmb{v},w_1,\widehat{\bw})$ at $a$, in which
$\mathcal{I}$ has a system of generators $F_\iota$, $\iota \in I$, of the form
\[
F_\iota(\pmb{u},\pmb{v},\bw) = \pmb{u}^{\pmb{\gamma}}\left( \widetilde{F}_{\iota}(\pmb{u},\pmb{v},\bw) w_1^d + \sum_{j=0}^{d-1} f_{\iota j}(\pmb{u},\pmb{v},\widehat{\bw})w_1^j  \right),
\]
where $\mathcal{I}_{\infty}^{\Phi} = (\pmb{u}^{\pmb{\gamma}})$, $\widetilde{F}_{\iota_d}(0)\neq 0$ for some $\io = \iota_d\in I$,
and $f_{\iota j}(a) =0$ for all $\iota \in I$, $0\leq j\leq d-1$. As in the proof of I.A.a($m,n$),
after a change of coordinates, we can assume that $f_{\iota_d,d-1}=0$.

Consider the ideal sheaves $\cK_j := (f_{\io j})_{\io \in I}$, $j=0,\ldots,d-1$. After applying
I.A.c.($m-1,n$) to the monomial morphism $\Phi|_{w_1=0}$ 
and the ideals $\cK_j$, we can assume that either 
$\cK_j=(0)$ or $\cK_j$ is principal and monomial, $j=0,\ldots,d-1$; i.e.,
\begin{equation}\label{eq:BasisIdeal2}
F_\iota(\pmb{u},\pmb{v},\bw) = \pmb{u}^{\pmb{\gamma}}\left( \widetilde{F}_\iota(\pmb{u},\pmb{v},\bw) w_1^d + \sum_{j=0}^{d-1} 
a_{\iota j}(\pmb{u},\pmb{v},\widehat{\bw})\pmb{u}^{\br_j}w_1^j  \right), \quad \iota\in I,
\end{equation}
where 
$a_{\io_d,d-1}=0$ and, for each $j=0,\ldots,d-1$, either $a_{\io j} =0$ for all $\io$, or
there exists $\io(j)\in I$ such that $a_{\io(j) j}(0)\neq0$.

We now make a codimension one blowing-up with center $\{w_1=0\}$ (which thus becomes a component
of the divisor). This blowing-up does not change any of the local normal forms (in particular, $\Phi$ remains
monomial), but it does change the derivations $\De^{\Phi}$ tangent to $\Phi$. In particular, $\partial/\partial w_1$ no longer belongs to $\De^{\Phi}$, but $X=w_1\,\partial/\partial w_1$ does.

By Lemma \ref{lem:ComputingInvariant} applied to $(F_{\iota_d})$ (i.e., applied with 
$F=F_{\iota_d}$, $f_d=\pmb{u}^{\pmb{\gamma}} w_1^d \widetilde{F}_{\iota_d}$, $f_k =a_{\iota_d k}\pmb{u}^{\pmb{\gamma}+\br_k}w_1^k$ 
(for $k=0,\ldots,d-1$), $f = 0$ and $X = w_1\,\partial/\partial w_1$), 
we have $\bu^{\bga}w_1^d \in \cI^{\De}_{d-1}$, where $\De = \De^{\Phi}$. Then, by Lemma \ref{lem:ComputingInvariant} applied to $(F_{\iota(j)})$ 
 with $f=\bu^{\bga}w_1^d$ (i.e., applied with $F=F_{\iota(j)}$, $f_k =a_{\iota(j) k}\pmb{u}^{\pmb{\gamma}+\br_k}w^k_1$ 
 (for $k=0,\ldots,d-1$), $f=\bu^{\bga}w^d_1$ and $X = w_1\,\partial/\partial w_1$), 
 we get $\bu^{\bga}\bu^{\br_j}w_{1}^j \in \cI^{\De}_{d-1}$, $j=0,\ldots,d-1$.
 Therefore, by \eqref{eq:BasisIdeal2} and Lemma \ref{lem:BasicPropClosure}(1),
\[
\mathcal{I}_{\infty}^{\De} = \mathcal{I}_{d-1}^{\De} 
= \pmb{u}^{\pmb{\gamma}}\left(w_1^d, \pmb{u}^{\br_j}w_1^j\right),
\]
and $\nu_a(\cI,\Phi) = \mu_a(\mathcal{I},\De) \leq d-1$.

After principalizing $\mathcal{I}_{\infty}^{\Phi}$ by combinatorial blowings-up, we can assume 
(using Corollary \ref{cor:DerCombinatorialBU} and Lemma \ref{lem:PreservingMu}(2)) that 
$\nu_a(\mathcal{I},\Phi)\leq d-1$, $\mathcal{I}_{\infty}^{\Phi}$ is principal monomial and $\Phi$ 
is monomial at $a$. The result follows by induction on $d$.
\end{proof}

\begin{remark}\label{rem:Noeth2}
In the analytic or algebraic cases, the proof of I.A.a($m,n$) above can be modified to prove
I.A.b($m,n$) directly; see Remark \ref{rem:Noeth1}. We leave the full details to the reader, but 
provide the following outline for guidance.

We can assume that there is no $\Phi$-free variable in any $\Phi$-monomial coordinate system; otherwise, we are done by I.A.b($m-1,n$),
because Remark \ref{rem:BasisFormalIdeal} guarantees that the ideal $\mathcal{I}$ has a system of generators independent of a $\Phi$-free variable.
Following the proof of I.A.a($m,n$) above, we can take $\mathcal{K} := \mathcal{I}_{\infty}^{\varphi}$; 
since $\vp$ and $\mathcal{K}$ are both independent of $w$ (by Remark \ref{rem:BasisFormalIdeal}), 
and can apply I.A.b($m-1,n-1$) to reduce to the case that $\mathcal{I}_{\infty}^{\varphi}$ is principal and monomial.

Arguing by induction on $d$ as in the proof of I.A.a($m,n$), we can assume that $w$ is the only $\varphi$-free variable
(otherwise, $\Phi$ has a free variable and we can conclude by I.A.b($m-1,n$), using
Remark \ref{rem:BasisFormalIdeal}).
Again by Remark \ref{rem:BasisFormalIdeal}, 
$\cI$ has a system of generators \eqref{eq:wprep}, where each $F_{\iota}$ is a eigenvector of $\Delta^{\Phi}$,
and we can assume that $\widetilde{F}_{\iota_d}=1$ (using the Weierstrass preparation theorem).
But $\psi$ is now of the form
$\psi(\pmb{u},\pmb{v},w) = \pmb{u}^{\pmb{\beta}}(\xi + w + h(\pmb{u},\pmb{v}))$, where $h(a)=0$
and $h$ is an eigenvector of $\De^\Phi$ and of $\De^\vp$, so that \eqref{eq:varphi2} can be replaced by
the assumption that $\psi(\pmb{u},\pmb{v},w) = g(\bu,\bv) + \pmb{u}^{\pmb{\beta}}(\xi + w )$, where $dg \wedge d\vp = 0$---this
is pre-monomial form.

The ideal sheaves $\cK_j$ are $\Delta^{\varphi}$- and $\Delta^{\Phi}$-closed. After applying
I.A.a($m-1,n-1$) to $\vp$ and the ideals $\cK_j$, we can assume that either 
$\cK_j=(0)$ or we have \eqref{eq:BasisIdeal}. We note that $\Phi$ is pre-monomial because condition (4) of I.A.b and Lemma \ref{lem:PreservingMu}
guarantee that the characterization of pre-monomial in Lemma \ref{lem:NormalFormPartialMonomial}(1) is preserved.

We can now simply introduce $\De$ as in the proof of I.A.a($m,n$), and can principalize $\mathcal{I}_{\infty}^{\De}$ by blowings-up that
are combinatorial with respect to the coordinates $(\pmb{u},w)$; this guarantees that the lifting of $\Phi$ is pre-monomial.
Then one of the following two possibilities holds.
\begin{enumerate}
\item[(i)] There is a regular vector field in $\De^{\Phi}$, and therefore a $\Phi$-free variable. In this case,
we can finish by applying II($m, n-1$) to $\Phi$ (which preserves the existence of a free variable at every point of the fibre)
and then applying I.A.b($m-1,n$) (since $\mathcal{I}$ has a system of generators independent of a $\Phi$-free variable, by
Remark \ref{rem:BasisFormalIdeal}).
\smallskip
\item[(ii)] There is no regular vector field in $\De^{\Phi}$. Then $\mathcal{I}$ is principal and 
monomial, by Lemma \ref{lem:BasicPropClosure}(3), and we can finish by applying II($m,n-1$) to $\Phi$.
\end{enumerate} 
\end{remark}

\subsection{Pre-monomialization: assertions I.B}\label{subsec:I.B}

\begin{proof}[Proof of I.B.a($m,n$)]
Consider the ideal sheaf $\mathcal{J}_1^{\Psi}:=\De^\Phi(\psi)$ (see Definition \ref{def:InvariantPartMon}).
After applying I.A.a($m,n$) to the monomial morphism $\Phi$ and the ideal $\mathcal{J}_1^{\Psi}$, 
we can assume that, either condition (5)(b) is satisfied, or $\mathcal{J}_1^{\Psi}$ is principal and
monomial at $a$, by Lemma \ref{lem:PreservingMu}(1). In the later case, $\nu_a(\Psi) = 1$,
and (5)(a) follows from  \ref{lem:NormalFormPartialMonomial}(1).
\end{proof}

\begin{proof}[Proof of I.B.b($m,n$)]
Consider the log differential closure $\mathcal{J}_{\infty}^{\Psi}$ of the ideal 
$\mathcal{J}_1^{\Psi}=\De^\Phi(\psi)$ (see Definition \ref{def:InvariantPartMon}); by
definition, $\De^{\Phi}(\mathcal{J}_{\infty}^{\Psi})\subset \mathcal{J}_{\infty}^{\Psi}$. After applying 
I.A.b($m,n$) to the monomial morphism $\Phi$ and the ideal $\mathcal{J}_{\infty}^{\Psi}$, 
we can assume that $\mathcal{J}_{\infty}^{\Psi}$ is principal and monomial at $a$. By Lemma \ref{lem:NormalFormPartialMonomial}, $d:= \nu_{a}(\Phi) <\infty$; moreover, $d=1$ if and only if 
the morphism $\Psi= (\Phi,\psi)$ is pre-monomial at $a$.

We argue by induction on $d$ (where the induction hypothesis requires that 
$\mathcal{J}_{\infty}^{\Psi}$ be principal and monomial).
By Lemma \ref{lem:NormalFormPartialMonomial}, there exists a $\Phi$-monomial coordinate system 
$(\pmb{u},\pmb{v},\pmb{w})=(\pmb{u},\pmb{v},w_1,\widehat{\bw})$ at $a$, such that 
\[
\psi(\pmb{u},\pmb{v},\pmb{w}) = g(\pmb{u},\pmb{v}) + \pmb{u}^{\pmb{\gamma}}\left( \widetilde{H}(\pmb{u},\pmb{v},\pmb{w}) w^d_1 + \sum_{j=0}^{d-2} h_{j}(\pmb{u},\pmb{v},\widehat{\pmb{w}})w^j_1  \right),
\]
where $\mathcal{J}_{\infty}^{\Psi} = (\pmb{u}^{\pmb{\gamma}})$, $\widetilde{H}(0)\neq 0$,
each $h_j(0)=0$, and $dg \wedge d\Phi= dg \wedge  d\varphi_1 \wedge \cdots \wedge d\varphi_n =0$. 
If $h_j = 0$ for all $j=0,\ldots, d-2$, then the morphism $\Psi$ is pre-monomial after blowing up with 
center $(w_1=0)$ in the source and $(\psi=0)$ in the target. In the following, therefore, we assume
that not all $h_j$ vanish identically.

After applying I.A.c($m-1,n$) to the monomial morphism $\Phi$ and the ideal sheaves $(h_j)$,
$j=1,\ldots, d-2$, we can assume that
\[
\psi(\pmb{u},\pmb{v},\pmb{w}) = g(\pmb{u},\pmb{v}) + \pmb{u}^{\pmb{\gamma}}\left( \widetilde{H}(\pmb{u},\pmb{v},\pmb{w}) w^d_1 + \sum_{j=1}^{d-2} b_{j}(\pmb{u},\pmb{v},\widehat{\pmb{w}})\pmb{u}^{r_j}w^j_1  + h_{0}(\pmb{u},\pmb{v},\widehat{\pmb{w}})\right),
\]
where either $b_{j} = 0$ or $b_{j}(a)\neq 0$, $j=1,\ldots, d-2$. (Note that we may apply I.A.c($m-1,n$), instead of I.A.c($m,n$), 
because both $\Phi$ and the ideal $(h_j)$ are independent of the variable $w_1$. Strictly speaking, we apply I.A.c($m-1,n$) to 
$\Phi|_{w_1=0}$ and $(h_j)|_{w_1=0}$, and we take the sequence of blowing-up with centres given by the products of those given by
I.A.c($m-1,n$) with the $w_1$-axis). 
Now, after applying I.B.b($m-1,n$) 
to $\Psi|_{w_1=0}= (\Phi,\psi)|_{w_1=0}$, we can assume that
\begin{equation}\label{eq:varphi}
\psi(\pmb{u},\pmb{v},\pmb{w}) = g(\pmb{u},\pmb{v}) + \pmb{u}^{\pmb{\gamma}}\left( \widetilde{H}(\pmb{u},\pmb{v},\pmb{w}) w^d_1+ \sum_{j=1}^{d-2} b_{j}(\pmb{u},\pmb{v},\widehat{\pmb{w}})\pmb{u}^{r_j}w^j_1  +\eta \,\pmb{u}^{\pmb{\lambda}}(\xi + w_2)^{\varepsilon}\right),
\end{equation}
where $\varepsilon,\,\eta \in \{0,1\}$ and, if $\varepsilon=1$, then $\xi\neq 0$ and $\Psi|_{w_1=0}$ is pre-monomial. In particular, either $\eta=0$ or there exists a vector field $Y \in \De^{\Phi}$ such that 
$Y(\pmb{u}^{\pmb{\gamma} +\pmb{\lambda}}(\xi + w_2)^{\varepsilon})
=\pmb{u}^{\pmb{\gamma} +\pmb{\lambda}}$ and $Y(w_1)= 0$. It follows that
\begin{equation}\label{eq:LastMonomial}
\eta\, \pmb{u}^{\pmb{\gamma} +\pmb{\lambda}} + \pmb{u}^{\pmb{\gamma}}O(w_1^d,\pmb{u}^{r_j}w_1^j) \in \mathcal{J}_1^{\Psi}.
\end{equation}

We now consider the auxiliary divisor $D':= D \cup \{w_1=0\}$, and we write $\Psi'$ to denote the morphism 
$\Psi$ with the divisor $D$ replaced by $D'$. (Note that $(\Psi')^{-1}(E)=D \subset D'$, but this will not 
intervene in the proof.) Using Lemma \ref{lem:ComputingInvariant} (applied with $X=w_1\,\partial/\partial w_1$, 
$F=X(\psi)$, $f_{0}=X(w_1^{d}\pmb{u}^{\pmb{\gamma}}\tilde{H})$, $f_j = X(b_{j} \pmb{u}^{\pmb{\gamma}+r_j} w_1^j)$ 
(for $i=1,\ldots,d-2$), and $f= 0$), 
we can conclude from \eqref{eq:varphi}, \eqref{eq:LastMonomial} and Lemma \ref{lem:BasicPropClosure}(1), that
\[
\mathcal{J}_{\infty}^{\Psi'} = \mathcal{J}_{d-1}^{\Psi'} 
= \pmb{u}^{\pmb{\gamma}}\left(w_1^d,\, \pmb{u}^{r_j}w_1^j,\, \eta\,\pmb{u}^{\pmb{\lambda}}\right).
\]
This implies that $\nu_a(\Psi')\leq d-1$. After making blowings-up that are combinatorial (with respect to 
$(\pmb{u},w_1)$), we can assume that $\mathcal{J}_{\infty}^{\Psi'}$ is principal and monomial. Now, one
of the following two possibilities holds.
\begin{enumerate}
\item[(i)] The transform of $\Psi$ locally coincides with that of $\Psi'$. In this case,
 we conclude (using Corollary \ref{cor:DerCombinatorialBU} and Lemma \ref{lem:PreservingMu}(2)) 
 that $\nu_{a}(\Psi)\leq d-1$, $\mathcal{J}_{\infty}^{\Psi}$ is principal and monomial, and $\Phi$ is monomial. So we can finish by induction on $d$.
\smallskip
\item[(ii)] We are at a point in the strict transform of $(w_1=0)$. In this later case, 
$\mathcal{J}_{\infty}^{\Psi'}$ is generated either by
the pull-back of $\pmb{u}^{\pmb{\gamma}+\pmb{\lambda}}$ (if $\eta\neq 0$) or by
the pull-back of $\pmb{u}^{\pmb{\gamma}}+r_jw_1^j$ (where $j$ is minimal). After possibly making a
codimension one blowing-up of $(w_1=0)$ in the source, we can conclude 
that $\Psi$ is pre-monomial, directly from \eqref{eq:varphi}.
\end{enumerate}
\vspace{-\baselineskip}
\end{proof}

\subsection{Monomialization: assertion II}\label{subsec:II}

\begin{proof}[Proof of II($m,n$)] 
We use the notation of Definition \ref{def:PreMonomial}. By Proposition \ref{prop:cleaning} or Proposition
\ref{prop:PreparationPreMonomial2}, we can assume that $\Psi$ has a remainder $g$ and relation $R$ at $a$,
such that $\rho := \rho_a(R) < \infty$ (we can assume that $g(a)=0$). (Recall that, in the real quasianalytic case,
there is, in fact, a relation of the form $R(\bx,\by,\bz,t)=t-h(\bx,\by,\bz)$, so that $\rho=1$; see Remark \ref{rem:cleaning}.)

In general, we will argue by induction on $\rho$. Let
$$
r_i(\bx,\by,\bz) := \frac{1}{i!} \frac{\p^iR}{\p t^i}(\bx,\by,\bz,0),\quad, i=0,1,\ldots,
$$
so that $R$ has a formal expansion $\sum r_i(\bx,\by,\bz)t^i$. Let $\cR$ denote the privileged ideal generated
by the $r_i$. The pulled back ideal $\Phi^*\cR$ is clearly $\De^\Phi$-closed. By applying I.A.b($m,n$) to the
monomial morphism $\Phi$ and the ideal $\Phi^*\cR$, we can reduce to the case that $\Phi^*\cR$ is a principal
monomial ideal, generated by $\Phi^*(r_{i_0})$, for some $i_0$. Then, after applying Lemma \ref{lem:PrincipalIdealTrick}
to $r_{i_0}$, we can also assume that $r_{i_0}$ is a monomial times a unit at $b=\Phi(a)$, and it follows from the
division axiom that $r_{i_o}$ divides all $r_i$. (Lemma \ref{lem:PreMonomial-ContollRelations} guarantees that the
pull-back of $R$ by the blowings-up in the target involved in I.A.b($m,n$) and Lemma \ref{lem:PrincipalIdealTrick}
is still a relation of order $\leq \rho$.) Moreover, since $r_{i_0}$ is a monomial independent of $t$, the quotient of $R$
by $r_{i_0}$ is still a relation; therefore, we can assume that
\begin{equation}\label{eq:PreparedRelation}
R = U(\pmb{x},\by,\bz,t) t^k + \sum_{i=0}^{k-1} a_i(\pmb{x},\pmb{y},\pmb{z})t^i,
\end{equation}
where $U(\pmb{x},\by,\bz,t)$ is a unit and $a_i(0) = 0$, $i=0,\ldots,k-1$. We can also assume that $k\geq \rho$
(otherwise, we are finished, by induction). We consider two cases:

\medskip\noindent
(1) $F$ is the induced divisor $E\times\IK$. In this case, $k=\rho$. (In the real quasianalytic case, with
$R$ as above, $\Psi$ is already a monomial morphism at $a$, after codimension one blowings-up
in the target and source if necessary, according to Remarks \ref{rem:premonom1} and \ref{rem:cleaning}.)
In general, by the implicit function theorem, after
a coordinate change in $t$, we can assume that
$$
R = U(\pmb{x},\by,\bz,t) t^\rho + \sum_{i=0}^{\rho-2} a_i(\pmb{x},\pmb{y},\pmb{z})t^i.
$$

By I.A.b($m,n$) applied to $\Phi$ and the ideal generated by $g$, we can assume also that
$g(\bu,\bv) = \bu^{\bde} V(\bu,\bv)$, where $V$ is a unit. Then, by further combinatorial blowings-up
in the source, we can reduce to the case that either $\bde < \bga$ or $\bga\leq\bde$ (in the notation of 
\eqref{eq:PreMonomial-Eq1}), as in the proof of 
Lemma \ref{lem:PreMonomial-ContollRelations2}. Suppose that  $\bga\leq \bde$. Then, in the second case
of \eqref{eq:PreMonomial-Eq1}, $\bga <\bde$, and $\Psi$ is a monomial morphism.
 In the third case, $\psi = \bu^{\bga}(\eta + \bu^{\bde-\bga} V + w_1)$. In this case, if $\bga < \bde$, or if $\bga = \bde$
 but $\eta + V$ is a unit, then $\Psi$ is a monomial morphism. If $\bga = \bde$ but $\eta + V$ is not a unit, then $\Psi$
 is a monomial morphism after the  codimension one blowing-up in the source with centre $\{\eta + V(\bu,\bv)  + w_1 = 0\}$.
 Such blowing-up has no effect on the log derivations $\De^\Psi$ because every $X\in\De^\Psi$ annihilates 
 $\psi,\, g$ and $\bu^{\bga}$, so that $X(w_1)=0$.
Therefore, we can assume that $\bde  < \bga$.

We can now make a codimension one blowing-up of the target with centre $\{t=0\}$; the effect is to make
$\{t=0\}$ a component of the divisor $F$, and the preceding assumptions
on $g$ guarantee that $\Psi$ is still a morphism. But $\De^\pi$ is now generated by $t\,\p/\p t$.
By Lemma \ref{lem:ComputingInvariant} applied to the ideal $\cI_R$ generated by $R$ (i.e., applied with $F=R$, 
$f_{\rho-1}=t^{\rho}U$, $f_i =a_{i}t^i$ (for $i=0,\ldots,\rho-2$), 
$f= 0$ and $X=t\,\p/\p t$), and Lemma \ref{lem:BasicPropClosure}(1),
\[
I_{R,\infty}^{\pi} = I_{R,\rho-1}^{\pi} = \left(t^{\rho},a_i(\pmb{x},\pmb{y},\pmb{z})t^i \right),
\]
so that $\rho_a(R) = \nu_{(b,0)}(I_R,\pi)\leq \rho - 1$. 

\medskip\noindent
(2) $F$ is the extended divisor $F=(E\times \mathbb{K}) \cup (N\times \{0\})$. In this case, we will reduce 
to case (1) without changing $\rho$: By applying I.A.b($m,n$) to the ideal generated by $g\cdot \prod_{i=0}^{k-1} \Phi^*(a_i)$,
and then applying Lemma \ref{lem:PrincipalIdealTrick} and using Lemma \ref{lem:PreMonomial-ContollRelations} (as in case (1)), 
we can reduce to the case that $g$ and $R$ are of the form \eqref{eq:control2}. By
Lemma \ref{lem:PreMonomial-ContollRelations2}, we have a finite number of commutative diagrams (as in the latter)
such that, at every $\ta \in \s_{\la}^{-1}(a)$, for each $\la$, we have either a relation of order $< \rho$, or 
a relation of order $\rho$, but $F = E\times\IK$ (i.e., we reduce to case (1)). 
\end{proof}

This completes the proof of our main theorems.

\bibliographystyle{amsplain}

\begin{thebibliography}{99}

\bibitem{ADK}
D. Abramovich, J. Denef and K. Karu, \textit{Weak toroidalization over non-closed fields},
Manuscripta Math. \textbf{142} (2013), 257--271.

\bibitem{AK}
D. Abramovich and K. Karu, \textit{Weak semistable reduction in characteristic 0},
Invent. Math. \textbf{139} (2000), 241--273.

\bibitem{AKMW}
D. Abramovich, K. Karu, K. Matsuki and J. W{\l}odarczyk,
\textit{Torification and factorization of birational maps}, J. Amer. Math. Soc. \textbf{15} (2002), 531--572.

\bibitem{ATW}
D. Abramovich, M. Temkin and J. W{\l}odarczyk,
\textit{Principalization of ideals on toroidal orbifolds}, J. Eur. Math. Soc. 
\textbf{22} (2020), 3805--3866.


\bibitem{ATWprog}
D. Abramovich, M. Temkin and J. W{\l}odarczyk,
\textit{Relative desingularization and principalization of ideals}, 
arXiv:2003.03659 [math.AG], 2020.


\bibitem{ALT} 
K. Adiprasito, G. Liu and M. Temkin, 
\textit{Semistable reduction in characteristic $0$}, arXiv:1810.03131.

\bibitem{AkKing}
S. Akbulut and H. King, \textit{Topology of real algebraic sets},
Mathematical Sciences Research Institute Publications \textbf{25}, Springer, New York, 1992.

\bibitem{BdSJA} A. Belotto da Silva, \textit{Global resolution of singularities subordinated to a $1$-dimensional foliation}, 
J. Algebra, \textbf{447} (2016), 397--423.

\bibitem{BdS} A. Belotto da Silva, \textit{Local resolution of ideals subordinated to a foliation}, Rev. R. Acad. Cienc. Exactas Fs. Nat. Ser. A Mat. RACSAM \textbf{110} (2016), 841--862.

\bibitem{BdSIbero} A. Belotto da Silva, \textit{Local monomialization of a system of first integrals of Darboux type}, 
Rev. Mat. Iberoamericana \textbf{34} (2018), 967--1000.

\bibitem{BBC} A. Belotto da Silva, E. Bierstone and M. Chow,
\textit{Composite quasianalytic functions},
Compositio Math. \textbf{154} (2018), 1960-1973.

\bibitem{BBAdvances}
{A. Belotto da Silva, E. Bierstone, V. Grandjean and P. Milman}
\textit{Resolution of singularities of the cotangent sheaf of a singular variety}
{Adv. Math.} \textbf{307}, (2017), 780--832.

\bibitem{BMihes} 
E. Bierstone and P.D. Milman,
\textit{Semianalytic and subanalytic sets},
Inst. Hautes Etudes Sci. Publ. Math. \textbf{67} (1988), 5--42

\bibitem{BMarcan}
E. Bierstone and P.D. Milman,
\textit{Arc-analytic functions}
Invent. Math. \textbf{101} (1990), 411--424.

\bibitem{BMinv}
E. Bierstone and P.D. Milman,
\textit{Canonical desingularization in characteristic zero by
blowing up the maximum strata of a local invariant},
Invent. Math. \textbf{128} (1997), 207--302.

\bibitem{BMselecta}
E. Bierstone and P.D. Milman,
\textit{Resolution of singularities in Denjoy-Carleman classes}, Selecta Math. (N.S.)
\textbf{10} (2004), 1--28.

\bibitem{BMfunct}
E. Bierstone and P.D. Milman,
\textit{Functoriality in resolution of singularities}, Publ. Res. Inst. Math. Sci. Kyoto Univ.
\textbf{44} (2008), 609--639.

\bibitem{BMV}
E. Bierstone, P.D. Milman and G. Valette,
\textit{Arc-quasianalytic functions}, Proc. Amer. Math. Soc. \textbf{143} (2015),
3915--3925.

\bibitem{BP}
E. Bierstone and A. Parusi\'nski,
\textit{Global smoothing of a subanalytic set}, 
Duke Math. J. \textbf{167} (2018), 3115--3128.

\bibitem{C1} S.D. Cutkosky, \textit{Local monomialization and factorization of morphisms},
Ast\'erisque \textbf{260}, 1999.

\bibitem{C1Fourier} S.D. Cutkosky, \textit{Local monomialization of transcendental
extensions}, Ann. Inst. Fourier (Grenoble) \textbf{55} (2005), 1517--1586.

\bibitem{C2} S.D. Cutkosky, \textit{Toroidalization of dominant morphisms of 3-folds},
Mem. Amer. Math. Soc. \textbf{190}, 2007.

\bibitem{C3} S.D. Cutkosky, \textit{Local monomialization of analytic maps},
Adv. Math. \textbf{307} (2017), 833--902.

\bibitem{D}
J. Denef,
\textit{Monomialization of morphisms and $p$-adic quantifier elimination},
Proc. Amer. Math. Soc. \textbf{141} (2013), 2569--2574.

\bibitem{EH}
P.M. Eakin and G.A. Harris,
\textit{When $F(f)$ convergent implies $f$ is convergent},
Math. Ann. \textbf{229} (1977), 201--210.

\bibitem{Ful}
W. Fulton,
\textit{Introduction to toric varieties},
Ann. of Math. Studies \textbf{131}, Princeton Univ. Press, Princeton, 1993.

\bibitem{Gab}
A.M. Gabrielov,
\textit{Formal relations among analytic functions},
Math. USSR-Izv. \textbf{7} (1973), 1056--1088.

\bibitem{HiroPisa}
H. Hironaka,
\textit{Introduction to real-analytic sets and real-analytic maps},
Quaderni dei Gruppi di Ricerca Matematica del Consiglio Nazionale delle Ricerche,
Istituto Matematico ``L. Tonelli'' dell'Universit\`a di Pisa, Pisa, 1973.

\bibitem{Horm}
L. H\"ormander,
\textit{The analysis of linear partial differential operators I},
Springer-Verlag, Berlin-Heidelberg-New York, 1983.

\bibitem{Jaffe} E. Jaffe, 
\textit{Pathological phenomena in Denjoy-Carleman classes},
Can. J. Math. \textbf{68} (2016), 88--108.

\bibitem{Ko} J. Koll\'ar, 
\textit{Partial resolution by toroidal blow-ups}, Tunis. J. Math. \textbf{1} (2019), 3--12.

\bibitem{Kom}
H. Komatsu,
\textit{The implicit function theorem for ultradifferentiable mappings},
Proc. Japan Acad. \textbf{55} (1979), 69--72.

\bibitem{Loj} S. {\L}ojasiewicz,
\textit{Ensembles semialg\'ebriques}, Inst. Hautes Etudes Sci., Bures-sur-Yvette, 1965; notes of lectures at Universit\'e d'Orsay, 
re-edited by M. Coste, 2006, \url{https://perso.univ-rennes1.fr/michel.coste/Lojasiewicz.pdf}.

\bibitem{Malg}
B. Malgrange,
\textit{Frobenius avec singularit\'es, 2. Le cas g\'en\'eral},
Invent. Math. \textbf{39} (1977), 67--89.

\bibitem{Mandel}
S. Mandelbrojt,
\textit{S\'eries adh\'erentes, r\'egularisation des suites, applications},
Collection Borel, Gauthiers-Villars, Paris, 1952.

\bibitem{Mil}
C. Miller, 
\textit{Infinite differentiability in polynomially bounded $o$-minimal structures},
Proc. Amer. Math. Soc. \textbf{123} (1995), 2551--2555.

\bibitem{NSV}
F. Nazarov, M. Sodin and A. Volberg,
\textit{Lower bounds for quasianalytic functions. I. How to control smooth functions},
Math. Scand. \textbf{95} (2004), 59--79.

\bibitem{Nelim} 
K.J. Nowak,
\textit{Quantifier elimination in quasianalytic structures via non-standard analysis}, 
Ann. Polon. Math. \textbf{114} (2015), 235--267.

\bibitem{Ogus}
A. Ogus, \textit{Lectures on logarithmic algebraic geometry},
Cambridge Univ. Press, Cambridge, 2018.

\bibitem{RSer}
J.-P. Rolin and T. Servi,
\textit{Quantifier elimination and rectilinearization theorem for generalized
quasianalytic algebras}, Proc. London Math. Soc. (3) \textbf{110} (2015), 1207--1247.

\bibitem{RSW}
J.-P. Rolin, P. Speissegger and A. J. Wilkie,
\textit{Quasianalytic Denjoy-Carleman classes and o-minimality},
J. Amer. Math. Soc. \textbf{16} (2003), 751--777.

\bibitem{Rou}
C. Roumieu,
\textit{Ultradistributions d\'efinies sur $\IR^n$ et sur certaines classes de vari\'et\'es diff\'erentiables},
J. Analyse Math. \textbf{10} (1962--63), 153--192.

\bibitem{Th1}
V. Thilliez,
\textit{On quasianalytic local rings},
Expo. Math. \textbf{26} (2008), 1--23.

\bibitem{W}
J. W{\l}odarczyk, \textit{Toroidal varieties and the weak factorization theorem}, Invent. Math. 
\textbf{154} (2003), 223--331.

\end{thebibliography}

\end{document}